\documentclass[a4paper]{article}

\usepackage[T1]{fontenc}
\usepackage{amsmath, amssymb, amsthm}
\usepackage{mathtools}
\usepackage{tikz-cd}

\usepackage{dsfont}

\usepackage[backref=true,backrefstyle=none]{biblatex}
\DefineBibliographyStrings{english}{%
  backrefpage  = {\ensuremath{\uparrow}},
  backrefpages = {\ensuremath{\uparrow}},
}
\usepackage{xcolor}
\usepackage[colorlinks, linkcolor=blue!50!black, citecolor=green!40!black, urlcolor=blue!50!black]{hyperref}
\usepackage[capitalise, noabbrev]{cleveref}

\usepackage{tikz}
\usepackage{tikz-3dplot}
\usetikzlibrary{angles}
\usepackage{adjustbox}

\usepackage{subcaption}

\usepackage{graphicx}

\usepackage{booktabs}
\usepackage{array}

\DeclareMathOperator{\Hom}{Hom}
\DeclareMathOperator{\End}{End}

\DeclareMathOperator{\id}{id}
\DeclareMathOperator{\Id}{Id}
\newcommand{\op}{\mathrm{op}}
\DeclareMathOperator{\pr}{pr}
\DeclareMathOperator{\inn}{in}
\DeclareMathOperator{\ev}{ev}
\DeclareMathOperator{\coev}{coev}

\newcommand{\N}{\mathds{N}}
\newcommand{\Z}{\mathds{Z}}
\newcommand{\Q}{\mathds{Q}}
\newcommand{\R}{\mathds{R}}
\newcommand{\C}{\mathds{C}}

\newcommand{\cat}[1]{\mathbf{#1}}
\newcommand{\Mat}{\cat{Mat}}
\newcommand{\Krn}{\cat{Krn}}
\newcommand{\Rel}{\cat{Rel}}
\newcommand{\Hilb}{\cat{Hilb}}
\newcommand{\fdVect}{\cat{fdVect}}
\newcommand{\fgpMod}{\cat{fgpMod}}
\newcommand{\Herm}{\cat{Herm}}
\newcommand{\DiffOp}{\cat{DiffOp}}
\newcommand{\VB}{\cat{VB}}

\newcommand{\llangle}{\langle\!\langle}
\newcommand{\rrangle}{\rangle\!\rangle}
\newcommand{\conj}[1]{\overline{#1}}

\DeclareMathOperator{\Lie}{Lie}
\DeclareMathOperator{\GL}{GL}
\DeclareMathOperator{\ad}{ad}
\DeclareMathOperator{\Ad}{Ad}
\let\origO\O
\let\origo\o
\let\origu\u
\renewcommand{\O}{\TextOrMath{\origO}{\operatorname{O}}}

\DeclareMathOperator{\U}{U}
\DeclareMathOperator{\SU}{SU}
\DeclareMathOperator{\Sp}{Sp}
\DeclareMathOperator{\gl}{\mathfrak{gl}}
\renewcommand{\o}{\TextOrMath{\origo}{\operatorname{\mathfrak{o}}}}

\renewcommand{\u}{\TextOrMath{\origu}{\operatorname{\mathfrak{u}}}}
\DeclareMathOperator{\Der}{Der}

\DeclareMathOperator{\Ker}{Ker}
\DeclareMathOperator{\tr}{tr}
\DeclareMathOperator{\Pic}{Pic}
\DeclareMathOperator{\grad}{grad}
\DeclareMathOperator{\dvg}{div}

\DeclareMathOperator{\Diff}{Diff}
\DeclareMathOperator{\Sym}{Sym}
\DeclareMathOperator{\gr}{gr}

\DeclareMathOperator*{\colim}{colim}

\theoremstyle{plain}
\newtheorem{theorem}{Theorem}[section]
\newtheorem{lemma}[theorem]{Lemma}
\newtheorem{corollary}[theorem]{Corollary}
\newtheorem{proposition}[theorem]{Proposition}

\theoremstyle{definition}
\newtheorem{definition}[theorem]{Definition}
\newtheorem{example}[theorem]{Example}

\theoremstyle{remark}
\newtheorem{remark}[theorem]{Remark}

\AddToHook{env/theorem/begin}{\crefalias{theorem}{theorem}}
\AddToHook{env/lemma/begin}{\crefalias{theorem}{lemma}}
\AddToHook{env/corollary/begin}{\crefalias{theorem}{corollary}}
\AddToHook{env/proposition/begin}{\crefalias{theorem}{proposition}}
\AddToHook{env/definition/begin}{\crefalias{theorem}{definition}}
\AddToHook{env/example/begin}{\crefalias{theorem}{example}}
\AddToHook{env/remark/begin}{\crefalias{theorem}{remark}}

\crefname{theorem}{Theorem}{Theorems}
\Crefname{theorem}{Theorem}{Theorems}
\crefname{lemma}{Lemma}{Lemmas}
\Crefname{lemma}{Lemma}{Lemmas}
\crefname{corollary}{Corollary}{Corollaries}
\Crefname{corollary}{Corollary}{Corollaries}
\crefname{proposition}{Proposition}{Propositions}
\Crefname{proposition}{Proposition}{Propositions}
\crefname{definition}{Definition}{Definitions}
\Crefname{definition}{Definition}{Definitions}
\crefname{example}{Example}{Examples}
\Crefname{example}{Example}{Examples}
\crefname{remark}{Remark}{Remarks}
\Crefname{remark}{Remark}{Remarks}
\crefname{part}{Part}{Parts}
\Crefname{part}{Part}{Parts}

\title{Dagger Categories in Riemannian Geometry}
\author{%
  Jón Hákon Garðarsson\textsuperscript{\dag}
  \qquad
  Paolo Perrone\textsuperscript{\dag}
  \\[1ex]
  \normalsize \textsuperscript{\dag}University of Oxford%
}
\date{}

\begin{document}

\maketitle

\begin{abstract}
A dagger on a category assigns to every morphism \(f \colon X \to Y\) a morphism \(f^\dagger \colon Y \to X\), contravariantly and involutively, in the way the transpose reverses a real matrix.
We show that on the categories of linear algebra and differential geometry, daggers encode metric and Riemannian structures.

As is known, on a single vector space, an anti-involution of the endomorphism ring is the adjoint of a form that is determined up to a scalar.
But more generally, on the category of all finite-dimensional vector spaces, a dagger is the adjoint for an involution of the scalars together with a unique unimodular Hermitian form on every space.
The same holds for finitely generated projective modules over a commutative ring, and hence, through the Serre--Swan correspondence, for vector bundles, where a dagger is a choice of bundle metrics.

The dagger structure expresses notions of orthogonality categorically: orthogonal complements and normal bundles acquire universal properties, and Hermitian modules form a dagger compact category.

Moreover, daggers on Lie groups, Lie groupoids and their infinitesimal counterparts capture unitary representations and metric connections: a connection is metric precisely when its parallel transport is a dagger functor, or, infinitesimally, when its splitting of the Atiyah sequence commutes with the daggers.

Finally, we classify the daggers on the locally graded category of differential operators of a Lie--Rinehart algebra.
They correspond to metric-divergence structures, consisting of a real structure, a divergence, and a unimodular Hermitian form on every finitely generated projective module, and they act on operators as formal adjoints; Green's formula expresses the difference between an operator and its adjoint as a divergence.
\end{abstract}

\tableofcontents

\section{Introduction}
\label{sec:intro}

Klein's Erlangen programme characterises a geometry by the group of transformations that preserve its structure \cite{klein1893}.
Eilenberg and Mac Lane presented category theory as a continuation of this programme, with groups of symmetries replaced by categories of structure-preserving maps \cite[p.~237]{eilenberg-maclane1945}.
The categorical point of view has been most successful on the algebraic and topological sides of geometry.
Metric notions such as inner products, orthogonality and unitarity fit less easily: restricting to the maps that preserve them, such as the isometries, discards most maps of interest, while the category of all maps does not see the metrics at all.

A \emph{dagger} on a category assigns to every morphism \(f \colon X \to Y\) a morphism \(f^\dagger \colon Y \to X\), contravariantly and involutively, like the transpose of a matrix \cite{selinger2007, karvonen2018}.
It keeps all the maps and encodes the metric data as an operation on them: adjoints with respect to scalar products form a dagger on the category of finite-dimensional real vector spaces, and we will see that every dagger on this category arises in this way.
In this paper we develop this observation into the statement that anti-involutions on algebraic structures are metric structures, and we follow it through a sequence of such structures, from a single vector space to the category of differential operators on a manifold.

The argument has three stages.
First, we recover forms from daggers: on a single vector space an anti-involution of its endomorphism ring determines a form up to scale \cite{knus-et-al1998}, while a dagger on the category of all finite-dimensional vector spaces determines a unique form on every space.
Through the Serre--Swan correspondence, the latter result becomes a classification of daggers on vector bundles by bundle metrics.
The dagger then expresses orthogonality, including orthogonal complements and normal bundles.

Second, we study Lie groups and groupoids and their infinitesimal counterparts.
Differentiating a dagger identifies skew-adjoint infinitesimal symmetries, and for a metric bundle the groupoid and algebroid daggers characterise metric parallel transport and metric connections.

Finally, we pass from module maps to differential operators of every order; over a manifold these are operators between bundles.
On the operators from functions to functions, a dagger supplies a real structure and a divergence.
On the operators between modules, it also supplies a unimodular Hermitian form on every module, and it acts as the formal adjoint.

\paragraph{Related work.}
The categorical treatment of adjoints and orthogonality includes dagger compact categories \cite{selinger2007}, dagger kernels \cite{heunen-jacobs2010}, and the reconstruction of dagger categories from anti-involutions and Hermitian fixed points \cite{stehouwer-steinebrunner2024}.
These works provide a general setting for daggers, in which the adjoint for inner products on Hilbert spaces is the motivating example; our classification goes in the converse direction and identifies the geometric data carried by a dagger on the categories studied here.
It works over any field or commutative ring, the involution of the scalars is part of the data rather than fixed in advance, additivity of the dagger is derived rather than assumed (\cref{prop:dagger-additive}), and no positivity is imposed, so that indefinite forms, like the Minkowski one, appear on the same footing as inner products.
For connections, the characterisation of the functors on the thin path groupoid that arise as parallel transport is due to Schreiber and Waldorf \cite{schreiber-waldorf2009}, while the modular class of Lie algebroids and Lie--Rinehart algebras provides a related setting for the divergence terms \cite{evens-lu-weinstein1999,huebschmann1999}.
The classical instance of our daggers on enveloping algebras is the transposition of the Weyl algebra, with \(x^\dagger = x\) and \(\partial^\dagger = -\partial\), which underlies the passage between left and right modules in the theory of \(D\)-modules; \cref{thm:anti-involutions-D} classifies the filtration-preserving anti-involutions of the enveloping algebra of any Lie--Rinehart algebra, by a real structure and a divergence.
The enveloping algebra of a Lie--Rinehart algebra also has a duality theory of its own \cite{lambre-le-meur2018}.
The jet comonad and its coalgebras appeared in Marvan's description of PDEs \cite{marvan1987}; synthetic and linear-logic treatments were developed later \cite{khavkine-schreiber2017,wallbridge2020}.
Flood, Mantegazza and Winther construct noncommutative holonomic, semiholonomic and nonholonomic jet functors, and show that linear differential operators of finite order form a category \cite[Proposition~10.3 and Corollary~10.6]{flood-mantegazza-winther2025}.
Our use of a locally graded jet comonad retains the order of each differential operator, and the dagger classification then relates this order structure to formal adjoints.

\paragraph{Contributions.}
\label{par:contributions}
The main results are the following.
\begin{itemize}
    \item Every dagger on the category of finite-dimensional vector spaces over a field \(K\) is the adjoint for a unique involution of \(K\) and a unique family of unimodular Hermitian forms, one on every space and standard on \(K\) (\cref{thm:daggers-are-forms}).
    Additivity of the dagger follows from the biproducts (\cref{prop:dagger-additive}), and over \(\R\) every dagger is linear and comes from scalar products (\cref{cor:real-daggers}).
    The same holds for finitely generated projective modules over a commutative ring, and hence, through the Serre--Swan correspondence, for vector bundles, where a dagger is a choice of bundle metrics (\cref{thm:daggers-are-metrics,cor:daggers-on-bundles}).
    \item We characterise when the standard functor from matrices is a dagger equivalence, and when two daggers on \(\fdVect_K\) are related by a linear dagger equivalence (\cref{prop:mat-dagger-equivalence,prop:dagger-equivalent-structures}).
    \item Orthogonal sums, orthogonal splittings, orthogonal squares and orthogonal complements have dagger-categorical universal properties (\cref{sec:orthogonality}).
    An orthogonal complement is an isometric kernel (\cref{prop:complement-is-kernel}), complements exist for Hermitian modules over any commutative ring with involution (\cref{prop:complements-exist-modules}), and the normal bundle of a submanifold is the orthogonal complement of its tangent bundle (\cref{ex:normal-bundle}).
    \item Hermitian modules form a dagger compact category, with the conjugate module as dual (\cref{thm:herm-dagger-compact}).
    \item A dagger on a Lie group differentiates to a dagger on its Lie algebra, and the unitary subgroup differentiates to the skew-adjoint elements (\cref{thm:differentiate-dagger}).
    For a metric bundle, the fibrewise adjoint on \(\GL(E)\) differentiates to the dagger on \(\gl(E)\) (\cref{thm:differentiate-groupoid-dagger}).
    A connection, flat or curved, is metric if and only if its parallel transport is a dagger functor (\cref{thm:metric-iff-dagger-transport}), and a splitting of the Atiyah sequence commutes with the daggers if and only if the connection is metric (\cref{cor:atiyah}).
    \item Filtration-preserving daggers on the enveloping algebra \(D\) of a Lie--Rinehart algebra correspond to real structures with a real divergence (\cref{thm:anti-involutions-D}); for the tangent algebroid of a manifold with its standard real structure, every real divergence differs from the divergence of a positive density by a real closed \(1\)-form (\cref{prop:divergences-differ}).
    \item Daggers on the locally graded category \(\DiffOp\) of differential operators correspond to metric-divergence structures, and act on operators as the formal adjoint (\cref{thm:formal-adjoints}).
    Green's formula expresses the difference between an operator and its adjoint as a divergence, so an invariant integral turns the dagger into an adjoint pairing on sections (\cref{prop:green-formula,cor:invariant-integral-adjoint}).
    The metric and the divergence are independent data (\cref{prop:reweighting,ex:circle-divergence}).
\end{itemize}

\paragraph{Reading guide.}
\Cref{sec:dagger-categories} recalls dagger categories.
\Cref{part:point} (\cref{sec:matrix-algebras,sec:lie-groups,sec:vector-spaces,sec:orthogonality}) needs only linear algebra and basic Lie theory.
\Cref{part:bundles} (\cref{sec:vector-bundles,sec:orthogonal-bundles,sec:lie-groupoids,sec:lie-algebroids}) uses manifolds, vector bundles and connections \cite{lee2013, lee2018}.
\Cref{part:diff} (\cref{sec:enveloping-algebra,sec:jets,sec:formal-adjoints}) is algebraic, with a manifold with a density as the running example, and the polynomial line \(K[x]\) as a second example in which everything can be computed by hand.
\Cref{sec:outlook} concludes the paper and discusses further directions.
Routine verifications are collected in \cref{sec:appendix}.

\paragraph{Acknowledgements.}
We would like to thank Matt Di Meglio, Tobias Fritz, Chris Heunen and Franco Rota for fruitful discussions and pointers to the literature.

\paragraph{Origin.}
Part of this work was carried out by the first author for his Master's dissertation at the University of Oxford.

\section{Dagger categories}
\label{sec:dagger-categories}

This section recalls the definition of a dagger category and the names of the morphisms that are compatible with the dagger \cite{selinger2007, karvonen2018, heunen-vicary2019}.
Dagger categories with one object, such as a matrix ring with the transpose, are first used in \cref{sec:matrix-algebras} and are introduced there.
For a category \(\cat C\), we write \(\cat C^\op\) for its opposite category, with the same objects and the morphisms reversed, and \(\Id\) for the identity functor.

\begin{definition}
    \label{def:dagger-category}
    A \emph{dagger} on a category \(\cat C\) assigns to each morphism \(f \colon X \to Y\) a morphism \(f^\dagger \colon Y \to X\) such that
    \[
        (g \circ f)^\dagger = f^\dagger \circ g^\dagger, \qquad
        \id_X^\dagger = \id_X, \qquad
        f^{\dagger\dagger} = f.
    \]
    Equivalently, it is an identity-on-objects functor \(\dagger \colon \cat C^\op \to \cat C\) with \(\dagger \circ \dagger^\op = \Id\).
    A \emph{dagger category} is a category with a dagger.
\end{definition}

\begin{example}
    \label{ex:dagger-categories}
    The category \(\Mat_\R\) has the natural numbers as objects and real \(m \times n\) matrices as morphisms \(n \to m\).
    The transpose \(M^\dagger = M^\top\) is a dagger on it, since \((M N)^\top = N^\top M^\top\).
    On the category \(\Mat_\C\) of complex matrices, the conjugate transpose \(M^\dagger = \conj M^\top\) is a dagger, where \(\conj M\) is the matrix of complex conjugate entries.
\end{example}

\begin{example}
    \label{ex:hilb-dagger}
    The category \(\Hilb\) of Hilbert spaces and bounded linear maps has the adjoint as its dagger.
    For \(\varphi \colon V \to W\), the adjoint \(\varphi^\dagger \colon W \to V\) is the unique bounded linear map with \(\langle \varphi(v), w \rangle = \langle v, \varphi^\dagger(w) \rangle\) for all \(v \in V\) and \(w \in W\), where \(\langle -, - \rangle\) denotes the inner products.
\end{example}

\begin{example}
    \label{ex:rel-dagger}
    The category \(\Rel\) has sets as objects and relations \(R \subseteq X \times Y\) as morphisms \(X \to Y\), composed in the usual way.
    Its dagger is the opposite relation \(R^\dagger = \{(y, x) \mid (x, y) \in R\}\).
\end{example}

\begin{example}
    \label{ex:krn-dagger}
    The category \(\Krn\) has standard Borel probability spaces as objects and measure-preserving Markov kernels, up to almost sure equality, as morphisms \cite{dahlqvist-et-al2018, ensarguet-perrone2024, perrone-vanbelle2024}.
    Its dagger is Bayesian inversion: for a kernel \(f \colon X \to Y\), the kernel \(f^\dagger \colon Y \to X\) is characterised by the fact that \(f\) and \(f^\dagger\) give the same joint distribution on \(X \times Y\).
\end{example}

\begin{definition}
    \label{def:dagger-vocabulary}
    In a dagger category,
    \begin{itemize}
        \item a morphism \(u \colon X \to Y\) is \emph{unitary} if it is invertible with \(u^{-1} = u^\dagger\);
        \item a morphism \(i \colon X \to Y\) is an \emph{isometry} if \(i^\dagger \circ i = \id_X\);
        \item a morphism \(p \colon X \to Y\) is a \emph{coisometry} if \(p \circ p^\dagger = \id_Y\);
        \item a morphism \(s \colon X \to X\) is \emph{self-adjoint} if \(s^\dagger = s\);
        \item a morphism \(e \colon X \to X\) is a \emph{projector} if \(e \circ e = e = e^\dagger\).
    \end{itemize}
    If the hom-sets are abelian groups, \(s \colon X \to X\) is \emph{skew-adjoint} if \(s^\dagger = -s\).
\end{definition}

A morphism is unitary if and only if it is both an isometry and a coisometry, and \(f\) is an isometry if and only if \(f^\dagger\) is a coisometry.
For every isometry \(i\), the composite \(i \circ i^\dagger\) is a projector.
In matrices, each of these notions picks out, among all morphisms of a given kind, the ones compatible with the Euclidean geometry of \(\R^n\).

\begin{example}
    \label{ex:mat-vocabulary}
    In \(\Mat_\R\) with the transpose, the unitaries are the orthogonal matrices, and the isometries are the matrices whose columns are orthonormal for the dot product.
    The coisometries are the matrices with orthonormal rows, the self-adjoint morphisms are the symmetric matrices, and the skew-adjoint ones are the antisymmetric matrices.
    The projectors are the matrices of orthogonal projections onto subspaces of \(\R^n\).
    The dagger singles out the orthogonal projections among the idempotent matrices.
    For example, the idempotent
    \(
        \bigl(\begin{smallmatrix} 1 & 1 \\ 0 & 0 \end{smallmatrix}\bigr)
    \)
    projects onto the first coordinate axis along the line spanned by \((1, -1)\), and it is not a projector, because this line is not orthogonal to the axis.
\end{example}

\begin{example}
    \label{ex:unitary-dagger}
    Every groupoid has the \emph{unitary dagger} \(g^\dagger = g^{-1}\).
    For it, every morphism is unitary.
\end{example}

\begin{definition}
    \label{def:dagger-functor}
    A functor \(F\) between dagger categories is a \emph{dagger functor} if \(F(f^\dagger) = F(f)^\dagger\).
    A \emph{dagger equivalence} is a pair of dagger functors \(F \colon \cat C \to \cat D\) and \(G \colon \cat D \to \cat C\) with unitary natural isomorphisms \(\Id \cong G F\) and \(F G \cong \Id\).
\end{definition}

A functor is \emph{unitarily essentially surjective} if every object of the codomain is unitarily isomorphic to an object in the image.

\begin{lemma}
    \label{lem:dagger-equivalence-criterion}
    A dagger functor is part of a dagger equivalence if and only if it is fully faithful and unitarily essentially surjective.
\end{lemma}

\begin{proof}
    A dagger equivalence is an equivalence, so its functors are fully faithful, and the unitary isomorphisms \(F(G(Y)) \cong Y\) show that \(F\) is unitarily essentially surjective.
    Conversely, let \(F \colon \cat C \to \cat D\) be a fully faithful and unitarily essentially surjective dagger functor.
    For each object \(Y\) of \(\cat D\), choose an object \(G(Y)\) of \(\cat C\) and a unitary \(u_Y \colon F(G(Y)) \to Y\).
    For \(g \colon Y \to Y'\), let \(G(g)\) be the unique morphism with \(F(G(g)) = u_{Y'}^\dagger \circ g \circ u_Y\), which exists because \(F\) is full and is unique because \(F\) is faithful.
    As in the usual construction of an inverse equivalence, this makes \(G\) a functor and \(u\) a natural isomorphism \(F G \cong \Id\).
    Since \(F\) is faithful and preserves the dagger, \(F(G(g)^\dagger) = (u_{Y'}^\dagger \circ g \circ u_Y)^\dagger = F(G(g^\dagger))\) gives \(G(g^\dagger) = G(g)^\dagger\), so \(G\) is a dagger functor.
    The counit \(u\) is unitary by construction, and the unit \(X \to G(F(X))\) is the unique morphism whose image under \(F\) is the unitary \(u_{F(X)}^\dagger\), so it is unitary as well, because \(F\) is a faithful dagger functor.
\end{proof}

In \cref{sec:vector-spaces} these notions become the familiar ones for Hermitian forms: unitaries preserve a form, isometries embed one Hermitian space in another, and projectors are orthogonal projections.

\part{Linear algebra and Lie theory}
\label{part:point}

\section{Matrix rings and forms}
\label{sec:matrix-algebras}

We begin with a single finite-dimensional vector space \(V\) over a field \(K\).
The algebraic structure is the ring \(\End_K(V)\), and the involution is an anti-involution of this ring: an additive map that reverses products and squares to the identity.
We show that such a map induces an involution \(a \mapsto \conj a\) of the scalars (\cref{lem:involution-on-scalars}), and that it is the adjoint of a unimodular form on \(V\) that is determined up to a non-zero scalar factor (\cref{thm:matrix-involutions}).
The unitary elements of the group \(\GL(V)\) of invertible linear maps are then the isometries of this form, which gives the orthogonal, Lorentz, symplectic and unitary groups.
The conjugate transpose of complex matrices shows that the involution of the scalars can be complex conjugation; for this reason the forms defined here are sesquilinear, and they are used unchanged in \cref{sec:vector-spaces,sec:vector-bundles,sec:formal-adjoints}.

\subsection{Dagger rings and dagger groups}
\label{sec:dagger-rings}

\begin{definition}
    \label{def:dagger-ring}
    A \emph{dagger ring} is a ring \(R\), not necessarily commutative, with an \emph{anti-involution}: an additive map \(r \mapsto r^\dagger\) such that
    \[
        (r s)^\dagger = s^\dagger r^\dagger, \qquad 1^\dagger = 1, \qquad r^{\dagger\dagger} = r.
    \]
\end{definition}

\begin{definition}
    \label{def:dagger-group}
    A \emph{dagger group} is a group \(G\) with a map \(g \mapsto g^\dagger\) such that \((g h)^\dagger = h^\dagger g^\dagger\), \(1^\dagger = 1\) and \(g^{\dagger\dagger} = g\); equivalently, a dagger category with one object in which every morphism is invertible.
    Its \emph{unitary subgroup} is
    \[
        \U(G) = \{u \in G \mid u^{-1} = u^\dagger\}.
    \]
\end{definition}

The units of a dagger ring form a dagger group.
The set \(\U(G)\) is a subgroup, because \((u v)^\dagger = v^\dagger u^\dagger = v^{-1} u^{-1}\) for unitary \(u\) and \(v\).

\begin{example}
    \label{ex:dagger-rings}
    The ring \(M(n, \R)\) of real \(n \times n\) matrices with the transpose is a dagger ring.
    Its group of units \(\GL(n, \R)\) is a dagger group whose unitary subgroup is the orthogonal group \(\O(n)\).
    Likewise, the ring \(M(n, \C)\) of complex matrices with the conjugate transpose is a dagger ring, and the unitary subgroup of \(\GL(n, \C)\) is the unitary group \(\U(n)\).
\end{example}

\begin{example}
    \label{ex:commutative-dagger-rings}
    In a commutative ring, \((a b)^\dagger = b^\dagger a^\dagger = a^\dagger b^\dagger\), so a commutative dagger ring is a commutative ring with an involutive automorphism, which we write \(a \mapsto \conj a\).
    Examples are \(\C\) with complex conjugation, and any commutative ring with the identity map.
\end{example}

For the unitary dagger \(g^\dagger = g^{-1}\) of \cref{ex:unitary-dagger}, every element is unitary and \(\U(G) = G\).

\subsection{Sesquilinear forms}
\label{sec:anti-involutions-matrix}

Let \(K\) be a field and \(V\) a non-zero finite-dimensional vector space over \(K\).

\begin{lemma}[{\cite[\S2]{knus-et-al1998}}]
    \label{lem:involution-on-scalars}
    Every anti-involution of the ring \(\End_K(V)\) preserves the centre \(K \cdot \id_V\), and so restricts to an involutive automorphism \(a \mapsto \conj a\) of \(K\).
    It is semilinear for this automorphism: \((a \varphi)^\dagger = \conj a\, \varphi^\dagger\).
\end{lemma}

\begin{proof}
    Let \(z\) be central.
    For every \(\varphi\), we have \(z^\dagger \circ \varphi^\dagger = (\varphi \circ z)^\dagger = (z \circ \varphi)^\dagger = \varphi^\dagger \circ z^\dagger\), and since every endomorphism is of the form \(\varphi^\dagger\), the element \(z^\dagger\) is central as well.
    The centre of \(\End_K(V)\) consists of the multiples \(a \id_V\) of the identity, and we identify it with \(K\).
    On this commutative ring the dagger is additive, reverses and hence preserves products, fixes \(1\) and squares to the identity, so it is an involutive automorphism of \(K\).
    Finally, \(a \varphi = (a \id_V) \circ \varphi\), so \((a \varphi)^\dagger = \varphi^\dagger \circ (\conj a \id_V) = \conj a\, \varphi^\dagger\).
\end{proof}

The hypothesis \(V \neq 0\) makes \(K \cdot \id_V\) a copy of \(K\); the ring \(\End_K(0)\) is zero.
The lemma is observed in \cite{knus-et-al1998}, where an involution likewise means an arbitrary additive anti-automorphism of order two.
It says that a dagger on \(\End_K(V)\) makes \(K\) a commutative dagger ring, and the next definitions are relative to that involution.

For the rest of this section, fix an involutive automorphism \(a \mapsto \conj a\) of \(K\), possibly the identity.
A map \(\varphi\) between vector spaces is \emph{semilinear} if it is additive and \(\varphi(a v) = \conj a\, \varphi(v)\).
We write \(V^\vee\) for the dual space of linear forms on \(V\), and \(\varphi^\vee \colon W^\vee \to V^\vee\), \(\alpha \mapsto \alpha \circ \varphi\), for the dual of a linear map \(\varphi \colon V \to W\).
In a basis \(e_1, \dots, e_n\) of \(V\), the coordinates of a vector \(v\) are written \(v^i\), so that \(v = \sum_i v^i e_i\).

\begin{definition}
    \label{def:sesquilinear}
    A \emph{sesquilinear form} on \(V\) is a biadditive map \(\langle -, - \rangle \colon V \times V \to K\) that is linear in the first variable and semilinear in the second:
    \[
        \langle a v, w \rangle = a \langle v, w \rangle, \qquad
        \langle v, a w \rangle = \conj a \langle v, w \rangle.
    \]
    It is \emph{Hermitian} if \(\langle w, v \rangle = \conj{\langle v, w \rangle}\) and \emph{skew-Hermitian} if \(\langle w, v \rangle = -\conj{\langle v, w \rangle}\).
    When the involution of \(K\) is the identity, these are called \emph{symmetric} and \emph{antisymmetric}.
    The form is \emph{unimodular} if the semilinear map
    \[
        \flat \colon V \to V^\vee, \qquad w \mapsto w^\flat = \langle -, w \rangle
    \]
    is bijective; its inverse is written \(\sharp\).
\end{definition}

The map \(\flat\) is semilinear because \((a w)^\flat = \conj a\, w^\flat\).
Since \(V\) and \(V^\vee\) have the same finite dimension, the form is unimodular as soon as \(\flat\) is injective, that is, as soon as \(\langle v, w \rangle = 0\) for all \(v\) implies \(w = 0\).\footnote{%
    Over a field, unimodular therefore means the same as non-degenerate.
    Over the rings of functions in \cref{sec:vector-bundles} the two notions differ, and bijectivity of \(\flat\) is the one we need.
    For example, the form \(\langle a, b \rangle = 2 a b\) on \(\Z\) is non-degenerate, but \(\flat\) is not surjective.%
}
A \emph{scalar product} is a unimodular symmetric form.
An \emph{inner product} is a positive definite scalar product over \(\R\), or a positive definite Hermitian form over \(\C\) with complex conjugation.

\begin{definition}
    \label{def:hermitian-space}
    A \emph{Hermitian space} is a finite-dimensional vector space with a unimodular Hermitian form.
    The space \(K\) carries the \emph{standard form} \(\langle a, b \rangle = a \conj b\).
\end{definition}

\begin{definition}
    \label{def:adjoint}
    Let \(V\) carry a unimodular sesquilinear form.
    The \emph{adjoint} of \(\varphi \in \End_K(V)\) is the unique \(\varphi^\dagger \in \End_K(V)\) with
    \[
        \langle \varphi(v), w \rangle = \langle v, \varphi^\dagger(w) \rangle \qquad \text{for all } v, w \in V.
    \]
\end{definition}

For each \(w\), the linear form \(\langle \varphi(-), w \rangle\) equals \(u^\flat\) for a unique \(u \in V\), and \(\varphi^\dagger(w) = u\) defines the adjoint; in terms of \(\flat\) and \(\sharp\), \(\varphi^\dagger = \sharp \circ \varphi^\vee \circ \flat\), which is linear as a composite of two semilinear maps and a linear one.
For \(\R^n\) with the dot product the adjoint of a matrix is its transpose, and for \(\C^n\) with \(\langle v, w \rangle = \sum_i v^i \conj{w^i}\) it is the conjugate transpose.

\subsection{Involutions and forms}
\label{sec:involutions-are-forms}

\begin{theorem}[{\cite[Theorem~4.2]{knus-et-al1998}}]
    \label{thm:matrix-involutions}
    Let \(\dagger\) be an anti-involution of the ring \(\End_K(V)\), inducing the involution \(a \mapsto \conj a\) of \(K\) (\cref{lem:involution-on-scalars}).
    Then \(\dagger\) is the adjoint of a unimodular sesquilinear form on \(V\), and the forms with adjoint \(\dagger\) are exactly the non-zero scalar multiples of any one of them.
    \begin{enumerate}
        \item If the involution of \(K\) is the identity, these forms are all symmetric or all antisymmetric.
        \item Otherwise, some of them are Hermitian, and the Hermitian ones are unique up to a non-zero factor \(c\) with \(\conj c = c\).
    \end{enumerate}
    Conversely, the adjoint of every unimodular Hermitian or skew-Hermitian form is an anti-involution of \(\End_K(V)\) inducing the given involution of \(K\).
\end{theorem}

\begin{proof}[Sketch of proof]
    We follow \cite[Chapter~I]{knus-et-al1998}, where the theorem is stated for central simple algebras.
    For existence, choose a basis of \(V\) and let \(\langle -, - \rangle_0\) be the form \(\langle v, w \rangle_0 = \sum_i v^i \conj{w^i}\).
    It is unimodular and Hermitian, and its adjoint \(\varphi \mapsto \varphi^{\dagger_0}\) is the conjugate transpose in this basis.
    The map \(\varphi \mapsto (\varphi^{\dagger_0})^\dagger\) is a ring automorphism of \(\End_K(V)\), and it is \(K\)-linear because both anti-involutions are semilinear for the same involution of \(K\).
    By the Skolem--Noether theorem it is inner, so there is \(\psi \in \GL(V)\) with \((\varphi^{\dagger_0})^\dagger = \psi \circ \varphi \circ \psi^{-1}\) for all \(\varphi\), and hence \(\varphi^\dagger = \psi \circ \varphi^{\dagger_0} \circ \psi^{-1}\), since \(\dagger_0\) is involutive (see the last paragraph of the proof).
    The form \(\langle v, w \rangle = \langle v, \psi^{-1}(w) \rangle_0\) is sesquilinear and unimodular, and
    \[
        \langle \varphi(v), w \rangle
        = \langle v, \varphi^{\dagger_0}(\psi^{-1}(w)) \rangle_0
        = \langle v, \psi^{-1}(\varphi^\dagger(w)) \rangle_0
        = \langle v, \varphi^\dagger(w) \rangle,
    \]
    so its adjoint is \(\dagger\).

    Every non-zero multiple of a form has the same adjoint.
    Conversely, let \(\llangle -, - \rrangle\) be another unimodular sesquilinear form with adjoint \(\dagger\).
    Write \(\flat'\) for the flat map of \(\llangle -, - \rrangle\).
    The map \(\chi = \sharp \circ \flat'\) is linear and satisfies \(\llangle v, w \rrangle = \langle v, \chi(w) \rangle\).
    Comparing \(\llangle \varphi(v), w \rrangle = \langle v, \varphi^\dagger(\chi(w)) \rangle\) with \(\llangle v, \varphi^\dagger(w) \rrangle = \langle v, \chi(\varphi^\dagger(w)) \rangle\) shows that \(\chi\) commutes with every \(\varphi^\dagger\), hence with every endomorphism, so \(\chi = c \id_V\) and \(\llangle -, - \rrangle = \conj c \langle -, - \rangle\).

    For the symmetry type, the form \(\conj{\langle w, v \rangle}\) is unimodular and sesquilinear in \((v, w)\), and its adjoint is again \(\dagger\) because \(\varphi^{\dagger\dagger} = \varphi\).
    Therefore \(\conj{\langle w, v \rangle} = \varepsilon \langle v, w \rangle\) for some \(\varepsilon \neq 0\), and applying this twice gives \(\varepsilon \conj \varepsilon = 1\).
    If the involution of \(K\) is the identity, then \(\varepsilon = \pm 1\), and every multiple \(a \langle -, - \rangle\) satisfies the same identity with the same \(\varepsilon\).
    Otherwise, \(a \langle -, - \rangle\) is Hermitian if and only if \(a = \conj a \varepsilon\).
    This holds for \(a = 1 + \varepsilon\) when \(\varepsilon \neq -1\), and for \(a = b - \conj b\), with \(b\) any scalar such that \(\conj b \neq b\), when \(\varepsilon = -1\).
    If \(\langle -, - \rangle\) is Hermitian, then \(\varepsilon = 1\), and \(a \langle -, - \rangle\) is Hermitian if and only if \(\conj a = a\).

    For the converse, let the form be Hermitian or skew-Hermitian, with \(\langle w, v \rangle = \pm \conj{\langle v, w \rangle}\).
    The adjoint is additive, reverses composition and fixes the identity, by the uniqueness in \cref{def:adjoint}.
    Using the symmetry twice, \(\langle \varphi^\dagger(v), w \rangle = \pm \conj{\langle w, \varphi^\dagger(v) \rangle} = \pm \conj{\langle \varphi(w), v \rangle} = \langle v, \varphi(w) \rangle\), so \(\varphi^{\dagger\dagger} = \varphi\).
    Finally, \(\langle a v, w \rangle = \langle v, \conj a w \rangle\) shows that \((a \id_V)^\dagger = \conj a \id_V\).
\end{proof}

\begin{remark}
    \label{rem:hermitian-vs-skew}
    When the involution of \(K\) is non-trivial, the scalar factor in \cref{thm:matrix-involutions} can change a Hermitian form into a skew-Hermitian one.
    If a non-zero scalar \(b\) satisfies \(\conj b = -b\), for example \(b = a - \conj a\) with \(\conj a \neq a\), then multiplying a Hermitian form by \(b\) gives a skew-Hermitian form with the same adjoint.
    Over \(\C\) with complex conjugation, \(b = i\) turns the standard Hermitian form into a skew-Hermitian one.
    When the involution is the identity, the dichotomy of \cref{thm:matrix-involutions} is an invariant of the anti-involution, except in characteristic \(2\), where symmetric and antisymmetric forms coincide; there the finer invariant is whether \(\langle v, v \rangle = 0\) for all \(v\) \cite[Theorem~4.2]{knus-et-al1998}.
\end{remark}

Both ambiguities, the scalar factor and (outside characteristic \(2\)) the antisymmetric forms, disappear in \cref{sec:vector-spaces}, where the dagger is defined on all spaces at once (\cref{rem:scale-fixed}).

A map that preserves a form also preserves its non-zero multiples, so the maps preserving the form depend only on the anti-involution.
For a unimodular Hermitian or skew-Hermitian form on \(V\), an invertible \(\varphi\) is unitary for the adjoint if and only if
\[
    \langle \varphi(v), \varphi(w) \rangle = \langle v, w \rangle \qquad \text{for all } v, w \in V,
\]
because \(\langle \varphi(v), \varphi(w) \rangle = \langle v, \varphi^\dagger(\varphi(w)) \rangle\) and the form is unimodular.
The unitary subgroup of \(\GL(V)\) is therefore the group of automorphisms that preserve the form, which we call its \emph{isometry group}; these automorphisms are the invertible isometries in the sense of \cref{def:dagger-vocabulary} (see \cref{sec:hermitian-spaces}).

\begin{example}
    \label{ex:dot-product-form}
    The dot product on \(\R^n\) is a scalar product, its adjoint is the transpose, and its unitary group is \(\O(n)\).
\end{example}

\begin{example}
    \label{ex:minkowski-form}
    The Minkowski space \(\R^{1,3}\) (\cref{rem:signature}) carries the scalar product
    \[
        \langle v, w \rangle = -v^0 w^0 + v^1 w^1 + v^2 w^2 + v^3 w^3,
    \]
    which is indefinite.
    The adjoint of a matrix is its transpose with the signs of the off-diagonal entries in the first row and the first column changed, and the unitary group is the Lorentz group \(\O(1,3)\).
\end{example}

\begin{example}
    \label{ex:symplectic-form}
    On \(\R^{2n}\), the form
    \[
        \langle v, w \rangle = \sum_{i = 1}^n \bigl(v^i w^{n+i} - v^{n+i} w^i\bigr)
    \]
    is antisymmetric and unimodular, and its unitary group is the symplectic group \(\Sp(2n, \R)\).
\end{example}

\begin{example}
    \label{ex:standard-hermitian-form}
    With complex conjugation, the standard Hermitian form on \(\C^n\),
    \[
        \langle v, w \rangle = \sum_i v^i \conj{w^i},
    \]
    has the conjugate transpose as its adjoint and \(\U(n)\) as its unitary group.
\end{example}

\begin{remark}
    \label{rem:signature}
    Let \(\langle -, - \rangle\) be a scalar product on a real vector space \(V\).
    In a basis \(e_1, \dots, e_n\), the form is determined by its components, since \(\langle v, w \rangle = \sum_{i, j} v^i w^j \langle e_i, e_j \rangle\).
    The components depend on the basis: if \(e'_j = \sum_k a^k_j e_k\), then \(\langle e'_i, e'_j \rangle = \sum_{k, l} a^k_i a^l_j \langle e_k, e_l \rangle\).
    The form itself is a geometric object, independent of any basis.
    There is always a basis with \(\langle e_i, e_j \rangle = \pm \delta_{ij}\), where \(\delta_{ij}\) is \(1\) for \(i = j\) and \(0\) otherwise, and by Sylvester's law of inertia the number \(p\) of minus signs and the number \(q\) of plus signs are the same for every such basis.
    The pair \((p, q)\) is the \emph{signature} of the form.
    We write \(\R^{p,q}\) for \(\R^{p+q}\) with the scalar product
    \[
        \langle v, w \rangle = -\sum_{i = 1}^{p} v^i w^i + \sum_{i = p+1}^{p+q} v^i w^i,
    \]
    which has signature \((p, q)\), and \(\O(p, q)\) for its isometry group.
    The Minkowski space \(\R^{1,3}\) has signature \((1, 3)\); there the coordinates are numbered \(v^0, \dots, v^3\).
    Two scalar products on \(V\) are \emph{isometric} if a linear automorphism of \(V\) carries one to the other, and this happens if and only if they have the same signature.
    The same holds for Hermitian forms on complex spaces with complex conjugation, since \(\langle v, v \rangle\) is then real.
\end{remark}

\section{Lie groups and Lie algebras}
\label{sec:lie-groups}

Lie groups add smooth structure to the picture.
In this section \(K\) is \(\R\) or \(\C\), and Hermitian spaces are taken over \(K\).
In \cref{sec:matrix-algebras} the unitary group of an anti-involution of \(\End_K(V)\) appeared as a subgroup of \(\GL(V)\); here the group itself is the structure, a Lie group \(G\) with a smooth anti-involution, and we differentiate it.
The derivative is an anti-involution of the Lie algebra \(\mathfrak g\) (\cref{thm:differentiate-dagger}), and the condition \(u^\dagger = u^{-1}\) on group elements becomes the condition \(X^\dagger = -X\) on Lie algebra elements.
For \(\GL(V)\) with the adjoint of a form, this says that the Lie algebra of the isometry group consists of the skew-adjoint maps.
Representations behave in the same way: a representation of a connected Lie group preserves a form if and only if its derivative acts by skew-adjoint maps (\cref{prop:unitary-rep-differentiates}).
In \cref{sec:lie-groupoids,sec:lie-algebroids} the same statement for groupoids and algebroids says that parallel transport preserves a metric if and only if the connection is metric.
Throughout, Lie algebras are real, and \(\gl(n, \C)\) is regarded as a real Lie algebra.

\begin{definition}
    \label{def:dagger-lie-group}
    A \emph{dagger Lie group} is a Lie group \(G\) with a smooth map \(g \mapsto g^\dagger\) making it a dagger group.
\end{definition}

The unitary subgroup \(\U(G)\) is the fixed-point set of the involutive automorphism \(g \mapsto (g^\dagger)^{-1}\), hence a closed Lie subgroup.

\begin{example}
    \label{ex:dagger-lie-groups}
    The group \(\GL(n, \R)\) with the transpose is a dagger Lie group with unitary subgroup \(\O(n)\), and \(\GL(n, \C)\) with the conjugate transpose is one with unitary subgroup \(\U(n)\).
    More generally, let \(V\) be a Hermitian space (\cref{def:hermitian-space}).
    The adjoint \(\varphi \mapsto \varphi^\dagger\) is a real-linear map on \(\End_K(V)\), so its restriction to the open subset \(\GL(V)\) is smooth, and \(\GL(V)\) is a dagger Lie group whose unitary subgroup is the isometry group of the form.
    The Minkowski space gives the Lorentz group \(\O(1,3)\) in this way.
\end{example}

The unitary dagger \(g^\dagger = g^{-1}\) of \cref{ex:unitary-dagger} makes every Lie group a dagger Lie group, since inversion is smooth.

\begin{definition}
    \label{def:dagger-lie-algebra}
    A \emph{dagger Lie algebra} is a real Lie algebra \(\mathfrak g\) with a linear map \(X \mapsto X^\dagger\) such that
    \[
        [X, Y]^\dagger = [Y^\dagger, X^\dagger] \qquad \text{and} \qquad X^{\dagger\dagger} = X.
    \]
    Its \emph{unitary subalgebra} is \(\u(\mathfrak g) = \{X \in \mathfrak g \mid X^\dagger = -X\}\), the skew-adjoint elements.
    The \emph{unitary dagger} on \(\mathfrak g\) is \(X^\dagger = -X\).
\end{definition}

The first condition says that \(\dagger\) is a Lie algebra homomorphism from \(\mathfrak g\) with the opposite bracket \([X, Y]^\op = [Y, X]\) to \(\mathfrak g\).
The unitary subalgebra is a Lie subalgebra, since for skew-adjoint \(X\) and \(Y\) we have \([X, Y]^\dagger = [-Y, -X] = -[X, Y]\).

\begin{example}
    \label{ex:gl-transpose}
    The algebra \(\gl(n, \R)\) of real \(n \times n\) matrices, with the commutator bracket \([X, Y] = X Y - Y X\), is a dagger Lie algebra with the transpose, because \([X, Y]^\top = [Y^\top, X^\top]\).
    Its unitary subalgebra is the algebra \(\o(n)\) of antisymmetric matrices.
\end{example}

\begin{remark}
    \label{rem:cartan-involution}
    Putting \(\theta(X) = -X^\dagger\) identifies daggers on \(\mathfrak g\) with involutive Lie algebra automorphisms \(\theta\).
    For \(\gl(n, \R)\) with the transpose, \(\theta(X) = -X^\top\) is the Cartan involution.
    Its fixed points form the Lie algebra \(\o(n)\) of the maximal compact subgroup \(\O(n)\) of \(\GL(n, \R)\).
    The map \(\theta\) preserves brackets, since \(\theta([X, Y]) = -[Y^\dagger, X^\dagger] = [\theta(X), \theta(Y)]\), and conversely \(X^\dagger = -\theta(X)\) is a dagger for every involutive automorphism \(\theta\).
    The unitary subalgebra \(\u(\mathfrak g)\) is the fixed-point set of \(\theta\).
    For Cartan involutions and the symmetric spaces they define, see also \cite{borel1998}.
    For \(\gl(n, \C)\) with the conjugate transpose, \(\theta(X) = -\conj X^\top\) is again the Cartan involution; see \cite[\S VII.2, Example~2]{knapp2002}, where \(\theta(X) = -\conj X^\top\) is the Cartan involution of any connected closed group of matrices that is closed under conjugate transpose, such as the identity components of \(\GL(n, \R)\) and \(\GL(n, \C)\).
    This \(\theta\) reappears as the real structure in \cref{sec:enveloping-algebra}.
\end{remark}

For a Lie group \(G\) we write \(\Lie(G)\) for its Lie algebra, identified with the tangent space at the identity, and \(\exp\) for its exponential map.
To differentiate a dagger, we compare \(\Lie(G)\) with the Lie algebra of the opposite group \(G^\op\), which is the same manifold with the product \(g \cdot h = h g\).

\begin{lemma}
    \label{lem:opposite-lie-algebra}
    Identify the Lie algebras of \(G\) and \(G^\op\) with the tangent space at the identity.
    Then the bracket of \(\Lie(G^\op)\) is the opposite \([X, Y]^\op = [Y, X]\) of the bracket of \(\Lie(G)\).
\end{lemma}

\begin{proof}
    Inversion \(g \mapsto g^{-1}\) is a smooth map \(G \to G^\op\), and it is a group isomorphism, because \((g h)^{-1} = h^{-1} g^{-1}\) is the product of \(g^{-1}\) and \(h^{-1}\) in \(G^\op\).
    Its derivative at the identity is \(X \mapsto -X\), and the derivative of a Lie group isomorphism is a Lie algebra isomorphism.
    Writing \([-, -]'\) for the bracket of \(\Lie(G^\op)\), this gives \([-X, -Y]' = -[X, Y]\), that is, \([X, Y]' = [Y, X]\).
\end{proof}

\begin{theorem}
    \label{thm:differentiate-dagger}
    Let \(G\) be a dagger Lie group with Lie algebra \(\mathfrak g\).
    The derivative of \(\dagger\) at the identity is a dagger on \(\mathfrak g\), and the Lie algebra of \(\U(G)\) is \(\u(\mathfrak g)\).
\end{theorem}

\begin{proof}
    The equation \((g h)^\dagger = h^\dagger g^\dagger\) says that \(\dagger\) is a Lie group homomorphism \(G^\op \to G\), and it fixes the identity.
    Its derivative \(X \mapsto X^\dagger\) at the identity is therefore a linear map, and a Lie algebra homomorphism from \(\Lie(G^\op)\) to \(\mathfrak g\).
    By \cref{lem:opposite-lie-algebra}, this means \([Y, X]^\dagger = [X^\dagger, Y^\dagger]\), which is the first condition of \cref{def:dagger-lie-algebra}.
    Since \(1^\dagger = 1\), applying the chain rule to \(g^{\dagger\dagger} = g\) at the identity gives \(X^{\dagger\dagger} = X\).

    For the second claim, note that \(G\) and \(G^\op\) have the same one-parameter subgroups \(t \mapsto \exp(t X)\), since these subgroups are commutative, so the two groups have the same exponential map.
    Lie group homomorphisms commute with the exponential map \cite[(1.82)]{knapp2002}, and applied to \(\dagger \colon G^\op \to G\) this says that the square
    \[
        \begin{tikzcd}
            \mathfrak g \rar["\dagger"] \dar["\exp"'] & \mathfrak g \dar["\exp"] \\
            G \rar["\dagger"] & G
        \end{tikzcd}
    \]
    commutes: \(\exp(X)^\dagger = \exp(X^\dagger)\).
    Consider the automorphism \(\Theta(g) = (g^\dagger)^{-1}\) of \(G\), whose fixed points form \(\U(G)\).
    The square gives
    \[
        \Theta(\exp(t X)) = \exp(t X^\dagger)^{-1} = \exp(-t X^\dagger) \qquad \text{for all } t \in \R.
    \]
    The Lie algebra of the closed subgroup \(\U(G)\) consists of the \(X\) with \(\exp(t X) \in \U(G)\) for all \(t\); see \cite[\S I.10]{knapp2002}.
    If \(X^\dagger = -X\), the displayed equation shows that \(\exp(t X)\) is fixed by \(\Theta\) for all \(t\).
    Conversely, if \(\exp(t X) = \Theta(\exp(t X)) = \exp(-t X^\dagger)\) for all \(t\), then differentiating at \(t = 0\) gives \(X = -X^\dagger\).
\end{proof}

Differentiating the last equation at \(t = 0\) shows that the involution \(\theta\) of \cref{rem:cartan-involution} is the derivative of the automorphism \(\Theta\).

\begin{example}
    \label{ex:gl-adjoint}
    Let \(V\) be a Hermitian space.
    The dagger of \(\GL(V)\) differentiates to the adjoint on \(\gl(V)\), and \(\u(\gl(V))\) consists of the skew-adjoint maps.
    The first claim holds because the dagger of \(\GL(V)\) is the restriction of the real-linear map \(\varphi \mapsto \varphi^\dagger\) on \(\End_K(V) = \gl(V)\), which is its own derivative.
    A map \(X\) is skew-adjoint if and only if \(\langle X(v), w \rangle = -\langle v, X(w) \rangle\) for all \(v, w\), and by \cref{thm:differentiate-dagger} these are the \(X\) for which every \(\exp(t X)\) preserves the form.
    For \(\R^n\) with the dot product, this gives the antisymmetric matrices \(\o(n)\), and for \(\C^n\) with the standard Hermitian form, the skew-Hermitian matrices \(\u(n)\).
    The same holds in every signature: for the space \(\R^{p,q}\) of \cref{rem:signature}, the Lie algebra \(\o(p,q)\) of the isometry group \(\O(p,q)\) consists of the maps that are skew-adjoint for the indefinite form.
\end{example}

\begin{definition}
    \label{def:unitary-representation}
    Let \(V\) be a Hermitian space, and give \(\GL(V)\) and \(\gl(V)\) the adjoint.
    A representation \(\pi \colon G \to \GL(V)\) of a group is \emph{unitary} if it is a dagger functor (\cref{def:dagger-functor}) from \(G\) with the unitary dagger to \(\GL(V)\), both regarded as one-object categories, that is, \(\pi(g^{-1}) = \pi(g)^\dagger\).
    A representation \(\pi \colon \mathfrak g \to \gl(V)\) of a Lie algebra is \emph{unitary} if \(\pi(X^\dagger) = \pi(X)^\dagger\) for the unitary dagger on \(\mathfrak g\), that is, \(\pi(X)^\dagger = -\pi(X)\).
\end{definition}

A group representation is unitary if and only if every \(\pi(g)\) is unitary, that is, preserves the form of \(V\).
A Lie algebra representation is unitary if and only if it takes values in the skew-adjoint maps \(\u(\gl(V))\).
A representation of a Lie group is a Lie group homomorphism \(\pi \colon G \to \GL(V)\), in particular smooth.

\begin{proposition}
    \label{prop:unitary-rep-differentiates}
    If \(\pi \colon G \to \GL(V)\) is a unitary representation of a Lie group, then its derivative \(\mathfrak g \to \gl(V)\) is unitary.
    If \(G\) is connected, the converse holds.
\end{proposition}

\begin{proof}
    Let \(d\pi\) be the derivative of \(\pi\), so that \(\pi(\exp(X)) = \exp(d\pi(X))\) for \(X \in \mathfrak g\).
    If \(\pi\) is unitary, it takes values in the closed subgroup \(\U(\GL(V))\), so its derivative takes values in the Lie algebra of \(\U(\GL(V))\), which is \(\u(\gl(V))\) by \cref{thm:differentiate-dagger} and \cref{ex:gl-adjoint}.

    Conversely, suppose that \(d\pi(X)\) is skew-adjoint for every \(X\).
    By \cref{thm:differentiate-dagger}, \(d\pi(X)\) lies in the Lie algebra of \(\U(\GL(V))\), so \(\pi(\exp(X)) = \exp(d\pi(X))\) is unitary.
    The unitaries form a subgroup, and when \(G\) is connected it is generated by the image of the exponential map, so \(\pi(g)\) is unitary for every \(g \in G\).
\end{proof}

\begin{example}
    \label{ex:disconnected-counterexample}
    Connectedness is needed for the converse of \cref{prop:unitary-rep-differentiates}: the discrete group \(\Z\) is a Lie group with zero Lie algebra, so the derivative of its representation on \(\R\) in which \(n\) acts by multiplication by \(2^n\) is unitary, but the representation is not unitary for the form \(\langle a, b \rangle = a b\).
\end{example}

\begin{example}
    \label{ex:unitary-representations}
    The cyclic group of order \(n\) acts on \(\R^2\) by rotations through multiples of \(2\pi / n\), and the dihedral group adds the reflections; these are the symmetries of a regular \(n\)-gon centred at the origin.
    Rotations and reflections preserve the dot product, so both representations are unitary.
\end{example}

\begin{example}
    \label{ex:permutation-representation}
    A group \(G\) acting on a finite set \(S\) acts on the space \(K^S\) with basis \((e_s)_{s \in S}\) by \(\pi(g)(e_s) = e_{g s}\).
    Permutations of the basis preserve the form with \(\langle e_s, e_t \rangle = \delta_{st}\), so this permutation representation is unitary.
\end{example}

\begin{example}
    \label{ex:tensor-representations}
    Unitary representations are closed under the usual constructions.
    If \(\pi\) and \(\pi'\) are unitary on \(V\) and \(V'\), then \(\pi(g) \otimes \pi'(g)\) is unitary for the form \(\langle v \otimes v', w \otimes w' \rangle = \langle v, w \rangle \langle v', w' \rangle\) on \(V \otimes V'\), and the induced representation on the exterior power \(\Lambda^k V\) is unitary for the form with \(\langle v_1 \wedge \dots \wedge v_k, w_1 \wedge \dots \wedge w_k \rangle = \det(\langle v_i, w_j \rangle)\).
\end{example}

\begin{example}
    \label{ex:dual-representation}
    If \(\pi\) is unitary on \(V\), the dual representation \(g \mapsto \pi(g^{-1})^\vee\) on \(V^\vee\) is unitary for the form \(\langle v^\flat, w^\flat \rangle = \langle w, v \rangle\), because \(\pi(g^{-1})^\vee(v^\flat) = \pi(g)(v)^\flat\).
\end{example}

\begin{example}
    \label{ex:lorentz-representation}
    The Lorentz group \(\O(1,3)\) acts unitarily on the Minkowski space, whose form is indefinite, and its Lie algebra acts by maps that are skew-adjoint for the Minkowski form.
\end{example}

\begin{example}
    \label{ex:quantum-symmetries}
    In quantum mechanics, the states of a system form a Hilbert space (\cref{ex:hilb-dagger}), and its symmetries act by unitary representations.
    For spin the space is finite-dimensional: the group \(\SU(2)\) acts on \(\C^2\) with the standard Hermitian form, describing the spin of an electron.
    By \cref{prop:unitary-rep-differentiates}, the Lie algebra acts by skew-adjoint maps.
    Physicists write these as \(i\) times self-adjoint maps, the observables, which is possible because a map \(X\) is skew-adjoint for a Hermitian form if and only if \(-i X\) is self-adjoint.
\end{example}

\section{Vector spaces and Hermitian forms}
\label{sec:vector-spaces}

In \cref{sec:matrix-algebras} we studied anti-involutions of the endomorphism ring of a single vector space, and each one turned out to be the adjoint of a form that is determined up to a scalar factor.
In this section we consider a dagger on the category of all finite-dimensional vector spaces, which assigns adjoints to linear maps between different spaces as well.
We show in \cref{thm:daggers-are-forms} that such a dagger is the adjoint for an involution of the scalars together with a unimodular Hermitian form on every space, and that these data are unique once the form on the one-dimensional space \(K\) is taken to be the standard one.
The proofs in this section are the model for those for vector bundles in \cref{sec:vector-bundles} and for differential operators in \cref{sec:formal-adjoints}.

Throughout, \(K\) is a field and \(\fdVect_K\) is the category of finite-dimensional vector spaces over \(K\) and linear maps.

\subsection{Hermitian spaces}
\label{sec:hermitian-spaces}

\begin{definition}
    \label{def:herm-K}
    Fix an involutive automorphism \(a \mapsto \conj a\) of \(K\).
    The category \(\Herm_K\) has Hermitian spaces (\cref{def:hermitian-space}) as objects and all linear maps as morphisms.
\end{definition}

The morphisms of \(\Herm_K\) are all linear maps; the forms are used to define their adjoints.

\begin{proposition}
    \label{prop:herm-dagger}
    Every linear map \(\varphi \colon V \to W\) between Hermitian spaces has a unique \emph{adjoint} \(\varphi^\dagger \colon W \to V\) with
    \[
        \langle \varphi(v), w \rangle = \langle v, \varphi^\dagger(w) \rangle \qquad \text{for all } v \in V,\ w \in W.
    \]
    Adjoints make \(\Herm_K\) a dagger category, and \((a \varphi)^\dagger = \conj a\, \varphi^\dagger\).
    In particular, \(\varphi \mapsto \varphi^\dagger\) is linear if and only if the involution of \(K\) is the identity.
\end{proposition}

\begin{proof}
    For each \(w \in W\), the map \(\langle \varphi(-), w \rangle\) is a linear form on \(V\), and since the form of \(V\) is unimodular, it equals \(\langle -, v \rangle\) for exactly one \(v \in V\).
    Taking \(\varphi^\dagger(w)\) to be this \(v\) gives the unique map satisfying the equation above.
    In terms of the musical maps, \(\varphi^\dagger = \sharp \circ \varphi^\vee \circ \flat\), and this map is linear because \(\flat\) and \(\sharp\) are both semilinear.

    The dagger axioms follow from the uniqueness of adjoints.
    For a second map \(\psi \colon W \to U\), the equations
    \[
        \langle \psi(\varphi(v)), u \rangle = \langle \varphi(v), \psi^\dagger(u) \rangle = \langle v, \varphi^\dagger(\psi^\dagger(u)) \rangle
    \]
    show that \(\varphi^\dagger \circ \psi^\dagger\) is the adjoint of \(\psi \circ \varphi\), and the identity is clearly its own adjoint.
    To see that \(\varphi\) is the adjoint of \(\varphi^\dagger\), we use Hermitian symmetry twice:
    \[
        \langle \varphi^\dagger(w), v \rangle = \conj{\langle v, \varphi^\dagger(w) \rangle} = \conj{\langle \varphi(v), w \rangle} = \langle w, \varphi(v) \rangle.
    \]
    Finally, \(\langle a \varphi(v), w \rangle = a \langle v, \varphi^\dagger(w) \rangle = \langle v, \conj a\, \varphi^\dagger(w) \rangle\), so the adjoint of \(a \varphi\) is \(\conj a\, \varphi^\dagger\).
\end{proof}

On a single space \(V\), the adjoint restricts to the anti-involution of \(\End_K(V)\) that \cref{thm:matrix-involutions} associates with the form of \(V\).
A linear map \(i \colon U \to V\) is an isometry in \(\Herm_K\) if and only if \(\langle i(u), i(u') \rangle = \langle u, u' \rangle\) for all \(u, u' \in U\), because \(\langle i(u), i(u') \rangle = \langle u, i^\dagger(i(u')) \rangle\).
The unitaries are the invertible isometries, and the unitary subgroup of \(\GL(V)\) (\cref{def:dagger-group}) is the group of isometries of the form of \(V\).

\begin{example}
    \label{ex:hermitian-spaces}
    With the identity involution, \(\R^n\) with the dot product is a Hermitian space with unitary group \(\O(n)\), and so is the Minkowski space \(\R^{1,3}\) of \cref{ex:minkowski-form}, whose unitary group is the Lorentz group \(\O(1,3)\).
    With complex conjugation, \(\C^n\) with \(\langle v, w \rangle = \sum_i v^i \conj{w^i}\) is a Hermitian space with unitary group \(\U(n)\).
    For any involution, \(K\) with the standard form \(\langle a, b \rangle = a \conj b\) is a Hermitian space.
\end{example}

The form of a Hermitian space can be recovered from the dagger of \(\Herm_K\), once vectors are regarded as linear maps.

\begin{definition}
    \label{def:elements-as-maps}
    For a vector \(v\) in \(V\), write \(\hat v \colon K \to V\) for the linear map \(a \mapsto a v\).
\end{definition}

Every linear map \(K \to V\) is \(\hat v\) for a unique \(v\), namely its value at \(1\), and \(\varphi \circ \hat v = \widehat{\varphi(v)}\) for every linear map \(\varphi\).
Endomorphisms of \(K\) are multiplications by scalars, and we identify \(\End_K(K)\) with \(K\), so that \(\hat a = a\) for \(a \in K\).

\begin{lemma}
    \label{lem:flat-is-dagger}
    In \(\Herm_K\), the adjoint of \(\hat v\) is the linear form \(v^\flat = \langle -, v \rangle\).
    Consequently \(\langle u, v \rangle = \hat v^\dagger \circ \hat u\), and for every linear map \(\varphi \colon V \to W\) the square
    \[
        \begin{tikzcd}
            W \rar["\varphi^\dagger"] \dar["\flat"'] & V \dar["\flat"] \\
            W^\vee \rar["\varphi^\vee"] & V^\vee
        \end{tikzcd}
    \]
    commutes.
\end{lemma}

\begin{proof}
    For \(a \in K\) and \(w \in V\), the standard form of \(K\) gives
    \(
        \langle \hat v(a), w \rangle = a \langle v, w \rangle = \langle a, \langle w, v \rangle \rangle
    \),
    so \(\hat v^\dagger(w) = \langle w, v \rangle = v^\flat(w)\).
    Composing with \(\hat u\) gives \(\hat v^\dagger \circ \hat u = \langle u, v \rangle\).
    For the square, \(\varphi^\vee(w^\flat) = \langle \varphi(-), w \rangle = \langle -, \varphi^\dagger(w) \rangle = \varphi^\dagger(w)^\flat\).
\end{proof}

\begin{remark}
    \label{rem:dagger-vs-dual}
    A finite-dimensional vector space is isomorphic to its dual, but a natural isomorphism would have to compare the covariant identity functor with the contravariant duality functor \((-)^\vee\).
    The dagger of \(\Herm_K\) is contravariant as well, and the square of \cref{lem:flat-is-dagger} compares it with the duality functor.
    When the involution is the identity, \(\flat\) is linear, and the square says that \(\flat\) is a natural isomorphism from \(\dagger\) to \((-)^\vee\).
    For a non-trivial involution, \(\flat\) is semilinear.
    The same square explains why we let forms be semilinear in the second variable: \(v^\flat = \hat v^\dagger\) is a linear form in its argument, while \((\widehat{a v})^\dagger = (\hat v \circ a)^\dagger = \conj a\, \hat v^\dagger\) depends semilinearly on \(v\).
\end{remark}

\subsection{Classification of daggers}
\label{sec:daggers-are-forms}

We now start from an arbitrary dagger on \(\fdVect_K\) and recover an involution of \(K\) and forms on all spaces.
We first show that the dagger is additive.
For daggers compatible with the biproducts this is well known \cite{heunen-vicary2019}; the following proof assumes no compatibility between the dagger and the biproducts.

\begin{proposition}
    \label{prop:dagger-additive}
    Every dagger on \(\fdVect_K\) is additive: \((\varphi + \psi)^\dagger = \varphi^\dagger + \psi^\dagger\).
\end{proposition}

\begin{proof}
    We use only that composition is additive in each variable and that \(\fdVect_K\) has a zero object and biproducts.

    A zero map \(V \to W\) factors through the zero space as \(V \to 0 \to W\), so its dagger factors as \(W \to 0 \to V\) and is again zero.

    Let \(W \oplus W\) be the biproduct, with injections \(\inn_1, \inn_2\) and projections \(\pr_1, \pr_2\).
    Then \(\pr_j \circ \inn_i\) is the identity for \(i = j\) and zero otherwise, and
    \[
        \inn_1 \circ \pr_1 + \inn_2 \circ \pr_2 = \id_{W \oplus W}.
    \]
    Applying the dagger to the first equations shows that \(\inn_j^\dagger \circ \pr_i^\dagger\) is the identity for \(i = j\) and zero otherwise.
    We claim that
    \begin{equation}
        \label{eq:dagger-biproduct}
        \pr_1^\dagger \circ \inn_1^\dagger + \pr_2^\dagger \circ \inn_2^\dagger = \id_{W \oplus W}.
    \end{equation}
    Composing either side with \(\inn_j^\dagger\) gives \(\inn_j^\dagger\), so it suffices to show that two maps \(\chi, \chi'\) into \(W \oplus W\) are equal as soon as \(\inn_j^\dagger \circ \chi = \inn_j^\dagger \circ \chi'\) for \(j = 1, 2\).
    Taking daggers, such maps satisfy \(\chi^\dagger \circ \inn_j = \chi'^\dagger \circ \inn_j\), and composing with \(\pr_j\) and adding gives \(\chi^\dagger = \chi'^\dagger\), hence \(\chi = \chi'\).

    Now let \(\varphi, \psi \colon V \to W\).
    Their sum factors through \(W \oplus W\), as the map \(\inn_1 \circ \varphi + \inn_2 \circ \psi\) with components \(\varphi\) and \(\psi\), followed by the codiagonal \(\pr_1 + \pr_2\), which adds the two components:
    \[
        \begin{tikzcd}[column sep=large]
            V \rar["\inn_1 \circ \varphi + \inn_2 \circ \psi"] \drar["\varphi + \psi"'] & W \oplus W \dar["\pr_1 + \pr_2"] \\
            & W
        \end{tikzcd}
    \]
    Taking daggers and inserting \eqref{eq:dagger-biproduct} in the middle,
    \[
        (\varphi + \psi)^\dagger
        = \sum_{j = 1, 2} \bigl(\pr_j \circ (\inn_1 \circ \varphi + \inn_2 \circ \psi)\bigr)^\dagger \circ \bigl((\pr_1 + \pr_2) \circ \inn_j\bigr)^\dagger
        = \varphi^\dagger + \psi^\dagger. \qedhere
    \]
\end{proof}

\begin{theorem}
    \label{thm:daggers-are-forms}
    Let \(\dagger\) be a dagger on \(\fdVect_K\).
    There are a unique involutive automorphism of \(K\) and a unique family of unimodular Hermitian forms, one on every space, whose adjoints are \(\dagger\) and whose form on \(K\) is the standard one.
\end{theorem}

\begin{proof}
    We begin with the scalars.
    The dagger restricts to \(\End_K(K) = K\), where it is additive by \cref{prop:dagger-additive}.
    Since \(K\) is commutative, \((a b)^\dagger = b^\dagger a^\dagger = a^\dagger b^\dagger\), and together with \(1^\dagger = 1\) and \(a^{\dagger\dagger} = a\) this makes \(a \mapsto a^\dagger\) an involutive automorphism of \(K\), which we write \(a \mapsto \conj a\).

    Next we show that the dagger is semilinear for this involution.
    For \(v \in V\) and \(a \in K\) we have \((a \id_V) \circ \hat v = \hat v \circ a\), and applying the dagger gives \(\hat v^\dagger \circ (a \id_V)^\dagger = \conj a\, \hat v^\dagger\).
    The dagger is a bijection from \(\Hom(K, V)\) to \(\Hom(V, K) = V^\vee\), so every linear form on \(V\) is \(\hat v^\dagger\) for some \(v\), and since linear forms separate points, \((a \id_V)^\dagger = \conj a \id_V\).
    Hence \((a \varphi)^\dagger = (\varphi \circ a \id_V)^\dagger = \conj a\, \varphi^\dagger\) for every linear map \(\varphi\).

    \Cref{lem:flat-is-dagger} tells us which forms to take:
    \[
        \langle u, v \rangle = \hat v^\dagger \circ \hat u \quad \in \End_K(K) = K.
    \]
    This is biadditive by \cref{prop:dagger-additive}, linear in \(u\) because \(\widehat{a u} = a\, \hat u\), and semilinear in \(v\) because \((a \hat v)^\dagger = \conj a\, \hat v^\dagger\).
    It is Hermitian, since \(\langle v, u \rangle = \hat u^\dagger \circ \hat v = (\hat v^\dagger \circ \hat u)^\dagger = \conj{\langle u, v \rangle}\).
    Since \(\langle u, v \rangle = \hat v^\dagger(u)\), the map \(\flat\) is the composite of the bijections \(v \mapsto \hat v\) and \(\hat v \mapsto \hat v^\dagger\), so the form is unimodular.
    On \(K\) itself \(\hat b = b\), so \(\langle a, b \rangle = \conj b\, a\) is the standard form.

    The dagger is the adjoint for these forms: for \(\varphi \colon V \to W\), \(v \in V\) and \(w \in W\),
    \[
        \langle \varphi(v), w \rangle
        = \hat w^\dagger \circ \varphi \circ \hat v
        = (\varphi^\dagger \circ \hat w)^\dagger \circ \hat v
        = \langle v, \varphi^\dagger(w) \rangle.
    \]

    For uniqueness, suppose that \(\dagger\) is the adjoint for some involution and some forms, standard on \(K\).
    Then \cref{lem:flat-is-dagger} gives \(\langle u, v \rangle = \hat v^\dagger \circ \hat u\), so the forms are the ones constructed above.
    The involution is determined as well, because for the standard form the adjoint of multiplication by \(a\) is multiplication by \(\conj a\).
\end{proof}

Together with \cref{prop:herm-dagger}, the theorem says that a dagger on \(\fdVect_K\) is the same thing as an involution of \(K\) together with a choice of unimodular Hermitian form on every space, standard on \(K\).
Each such choice identifies \(\fdVect_K\) with a full subcategory of \(\Herm_K\) that contains exactly one Hermitian space over each vector space.

\begin{corollary}
    \label{cor:real-daggers}
    Every dagger on \(\fdVect_\R\) is linear, and it is the adjoint for a unique family of scalar products, one on every space and standard on \(\R\).
\end{corollary}

\begin{proof}
    The only automorphism of the field \(\R\) is the identity, so the involution of \cref{thm:daggers-are-forms} is trivial, the forms are symmetric, and the dagger is linear by \cref{prop:herm-dagger}.
\end{proof}

\begin{remark}
    \label{rem:complex-daggers}
    Over \(\C\), the continuous involutive automorphisms are the identity and complex conjugation.
    The first gives daggers coming from symmetric bilinear forms on complex spaces, and the second gives daggers coming from Hermitian forms in the usual sense.
    Discontinuous involutive automorphisms of \(\C\) also exist, although constructing one requires the axiom of choice, and each of them gives its own family of daggers.
\end{remark}

\begin{remark}
    \label{rem:scale-fixed}
    On a single space, an anti-involution of \(\End_K(V)\) determines a form only up to a non-zero scalar, and for the identity involution this form may be antisymmetric (\cref{thm:matrix-involutions}).
    In \cref{thm:daggers-are-forms} neither ambiguity remains.
    The space \(K\) is one of the objects, the involution of \(K\) is the dagger itself, and the form of every other space is determined by \(\langle u, v \rangle = \hat v^\dagger \circ \hat u\).
    This form is Hermitian because the dagger is involutive; in characteristic other than \(2\) this excludes the antisymmetric forms of \cref{thm:matrix-involutions}.
\end{remark}

\subsection{Metrics and bases}
\label{sec:metric-vs-basis}

In coordinates, the adjoint of a real matrix with respect to the dot product is its transpose; this subsection makes that statement precise, and asks when two daggers on \(\fdVect_K\) agree up to a change of coordinates.

\begin{definition}
    \label{def:orthonormal-basis}
    A basis \(e_1, \dots, e_n\) of a Hermitian space is \emph{orthonormal} if \(\langle e_i, e_j \rangle = \pm \delta_{ij}\).
\end{definition}

The category \(\Mat_K\) of matrices over \(K\) is defined as for \(\R\) and \(\C\) in \cref{ex:dagger-categories}.
The standard functor \(\Mat_K \to \fdVect_K\) sends \(n\) to \(K^n\) and a matrix to the linear map it defines in the standard bases.

\begin{proposition}
    \label{prop:mat-dagger-equivalence}
    Let \(\dagger\) be a dagger on \(\fdVect_K\), with involution \(a \mapsto \conj a\) and forms \(\langle -, - \rangle\) as in \cref{thm:daggers-are-forms}, and give \(\Mat_K\) the conjugate transpose \(M^\dagger = \conj M^\top\).
    The standard functor \(\Mat_K \to \fdVect_K\) is a dagger functor if and only if the form of every \(K^n\) is the standard form \(\langle v, w \rangle = \sum_i v^i \conj{w^i}\).
    It is then a dagger equivalence if and only if every space has a basis with \(\langle e_i, e_j \rangle = \delta_{ij}\).
\end{proposition}

\begin{proof}
    If every \(K^n\) has the standard form, then for an \(m \times n\) matrix \(M\) with entries \(M^i_j\),
    \[
        \langle M v, w \rangle = \sum_{i, j} M^i_j v^j \conj{w^i} = \sum_j v^j \conj{\sum_i \conj{M^i_j} w^i} = \langle v, M^\dagger w \rangle,
    \]
    so the functor preserves the dagger.
    Conversely, suppose it is a dagger functor.
    The columns \(e_j \colon 1 \to n\) of the identity matrix are sent to the maps \(\hat e_j\).
    Since the form of \(K = K^1\) is standard (\cref{thm:daggers-are-forms}), \cref{lem:flat-is-dagger} shows that the number \(\langle e_i, e_j \rangle = \hat e_j^\dagger \circ \hat e_i\) is the image of the \(1 \times 1\) matrix \(e_j^\dagger e_i = \delta_{ij}\), and the form is standard.

    The functor is fully faithful, so by \cref{lem:dagger-equivalence-criterion} it is a dagger equivalence if and only if every space \(V\) is unitarily isomorphic to some \(K^n\).
    A unitary \(K^n \to V\) is determined by the images of the standard basis vectors, which form a basis of \(V\) with \(\langle e_i, e_j \rangle = \delta_{ij}\), and every such basis arises in this way.
\end{proof}

Over \(\R\) with the identity, and over \(\C\) with complex conjugation, the Gram--Schmidt process shows that a Hermitian space has a basis with \(\langle e_i, e_j \rangle = \delta_{ij}\) if and only if its form is positive definite.
For positive definite forms, then, choosing a metric amounts to choosing which bases count as orthonormal.
The Minkowski space has orthonormal bases in the sense of \cref{def:orthonormal-basis}, but none with \(\langle e_i, e_j \rangle = \delta_{ij}\).
Over other fields positivity is not enough: the form \(\langle a, b \rangle = 2 a b\) on \(\Q\) is positive definite, but \(2 c^2 = 1\) has no rational solution.

\begin{proposition}
    \label{prop:dagger-equivalent-structures}
    Let \(\dagger\) and \(\ddagger\) be daggers on \(\fdVect_K\), with forms \(\langle -, - \rangle\) and \(\llangle -, - \rrangle\).
    There is an identity-on-objects linear functor \(\fdVect_K \to \fdVect_K\) that is an equivalence and sends \(\dagger\) to \(\ddagger\) if and only if the two daggers have the same involution of \(K\) and there are linear automorphisms \(\varphi_V\) of every space \(V\) with
    \[
        \langle u, v \rangle = \llangle \varphi_V(u), \varphi_V(v) \rrangle .
    \]
\end{proposition}

\begin{proof}
    Let \(F\) be such a functor.
    Since \(F\) is linear, it is the identity on \(\End_K(K)\), which is \(K\).
    Define \(\varphi_V(v) = F(\hat v)(1)\), so that \(F(\hat v) = \widehat{\varphi_V(v)}\).
    The map \(\varphi_V\) is linear, and it is injective because \(F\) is faithful, so it is an automorphism.
    For \(\psi \colon V \to W\), applying \(F\) to \(\psi \circ \hat v = \widehat{\psi(v)}\) gives \(F(\psi)(\varphi_V(v)) = \varphi_W(\psi(v))\), hence \(F(\psi) = \varphi_W \circ \psi \circ \varphi_V^{-1}\).
    Because \(F\) is the identity on \(K\) and sends \(\dagger\) to \(\ddagger\), the two daggers agree on \(K\), so they have the same involution, and
    \[
        \langle u, v \rangle = F(\hat v^\dagger \circ \hat u) = F(\hat v)^\ddagger \circ F(\hat u) = \llangle \varphi_V(u), \varphi_V(v) \rrangle.
    \]
    Conversely, given such automorphisms, \(F(\psi) = \varphi_W \circ \psi \circ \varphi_V^{-1}\) defines an identity-on-objects linear isomorphism of categories, and
    \[
        \llangle F(\psi)(\varphi_V(v)), \varphi_W(w) \rrangle
        = \langle \psi(v), w \rangle
        = \langle v, \psi^\dagger(w) \rangle
        = \llangle \varphi_V(v), F(\psi^\dagger)(\varphi_W(w)) \rrangle,
    \]
    which shows \(F(\psi)^\ddagger = F(\psi^\dagger)\).
\end{proof}

The condition on the involutions follows from the condition on the forms for \(V = K\), where both forms are standard.
Write \(b \mapsto b^\ddagger\) for the involution of \(\ddagger\) and let \(\varphi_K\) be multiplication by \(c\).
Then \(a \conj b = \llangle c a, c b \rrangle = c\, c^\ddagger a\, b^\ddagger\) for all \(a, b\); taking \(a = b = 1\) gives \(c\, c^\ddagger = 1\), and then \(\conj b = b^\ddagger\).

Two daggers are therefore related by such a functor, that is, \((\fdVect_K, \dagger)\) and \((\fdVect_K, \ddagger)\) are isomorphic by an identity-on-objects linear dagger functor, when, space by space, their forms are isometric.
Over \(\R\), Sylvester's law of inertia (\cref{rem:signature}) turns this into the condition that the forms have the same signature on every space.
For example, the form with matrix \(\mathrm{diag}(1, 1)\) on \(\R^2\) is carried to the one with matrix \(\mathrm{diag}(1, 2)\) by \(\varphi(e_1) = e_1\), \(\varphi(e_2) = e_2 / \sqrt 2\), while \(\mathrm{diag}(1, 1)\) and \(\mathrm{diag}(1, -1)\) are not.
A dagger that puts a Minkowski form on \(\R^4\) is not equivalent to one whose forms are all positive definite, and by \cref{prop:mat-dagger-equivalence} it is not equivalent to \(\Mat_\R\) with the transpose.

\section{Orthogonality in dagger categories}
\label{sec:orthogonality}

With the forms of \cref{sec:vector-spaces} in place, the familiar constructions of linear algebra acquire an orthogonal version: direct sums become orthogonal sums, images of idempotents become orthogonal projections, and kernels become orthogonal complements.
In a dagger category each of these is a universal property with an extra equation involving the dagger, and in \(\Herm_K\) the extra equation says that two subspaces are orthogonal.
We state the definitions in an arbitrary dagger category with a zero object, and prove what they mean in \(\Herm_K\), where forms may be indefinite.
The same definitions apply to vector bundles in \cref{sec:orthogonal-bundles}, where the normal bundle of a submanifold is an orthogonal complement.
Readers interested chiefly in connections and differential operators only need orthogonal morphisms (\cref{sec:orthogonal-biproducts}) and orthogonal complements (\cref{sec:orthogonal-complements}) and can then continue with \cref{sec:vector-bundles}; orthogonal squares (\cref{sec:orthogonal-squares}) are used only in the bundle examples of \cref{sec:orthogonal-bundles}.

Throughout, \(\cat C\) is a dagger category with a zero object \(0\).
For objects \(X\) and \(Y\), the composite \(X \to 0 \to Y\) is the \emph{zero morphism} \(0 \colon X \to Y\).
Its dagger is the composite \(Y \to 0 \to X\), which is again zero, and every composite with a zero morphism is zero.

\subsection{Orthogonal morphisms and biproducts}
\label{sec:orthogonal-biproducts}

\begin{definition}
    \label{def:orthogonal}
    Morphisms \(f \colon X \to Z\) and \(g \colon Y \to Z\) are \emph{orthogonal}, written \(f \perp g\), if \(f^\dagger \circ g = 0\).
\end{definition}

Orthogonality is symmetric, since \(g^\dagger \circ f = (f^\dagger \circ g)^\dagger\) and the dagger of a zero morphism is zero.

\begin{example}
    \label{ex:orthogonal-vectors}
    In \(\Herm_K\), \cref{lem:flat-is-dagger} gives \(\hat v^\dagger \circ \hat w = \langle w, v \rangle\), so \(\hat v \perp \hat w\) means \(\langle w, v \rangle = 0\).
    More generally, linear maps \(\varphi \colon U \to W\) and \(\psi \colon V \to W\) are orthogonal if and only if \(\langle \varphi(u), \psi(v) \rangle = 0\) for all \(u \in U\) and \(v \in V\), because \(\langle u, \varphi^\dagger(\psi(v)) \rangle = \langle \varphi(u), \psi(v) \rangle\) and the form of \(U\) is unimodular.
    For isometries, this says that their images are orthogonal subspaces of \(W\).
\end{example}

\begin{example}
    \label{ex:krn-independence}
    In the category of finite probability spaces and stochastic matrices, the one-point space is a zero object, and random variables on a probability space are coisometries.
    Two random variables \(f\) and \(g\) on the same space are independent if and only if \(f^\dagger \perp g^\dagger\) \cite[Example~3.6]{di-meglio-et-al2025}.
    This category is the finite part of \(\Krn\) (\cref{ex:krn-dagger}).
\end{example}

Recall that a \emph{biproduct} of \(X_1\) and \(X_2\) is an object \(X_1 \oplus X_2\) with injections \(\inn_j \colon X_j \to X_1 \oplus X_2\) and projections \(\pr_j \colon X_1 \oplus X_2 \to X_j\), such that \(\pr_j \circ \inn_j = \id\), \(\pr_j \circ \inn_k = 0\) for \(j \neq k\), the projections make it a product, and the injections make it a coproduct.

\begin{definition}
    \label{def:orthogonal-biproduct}
    An \emph{orthogonal biproduct} of \(X_1\) and \(X_2\) is a biproduct \(X_1 \oplus X_2\) whose injections and projections satisfy \(\pr_j = \inn_j^\dagger\).
\end{definition}

The injections of an orthogonal biproduct are orthogonal isometries, since
\[
    \inn_j^\dagger \circ \inn_j = \pr_j \circ \inn_j = \id
    \qquad \text{and} \qquad
    \inn_1^\dagger \circ \inn_2 = \pr_1 \circ \inn_2 = 0.
\]
Orthogonal biproducts are unique up to a unitary.
If \(X_1 \oplus' X_2\) is another one, the comparison isomorphism \(\varphi\), determined by \(\pr'_j \circ \varphi = \pr_j\), satisfies \(\varphi^\dagger \circ \inn'_j = (\pr'_j \circ \varphi)^\dagger = \inn_j\), so \(\varphi^\dagger\) is the inverse comparison isomorphism.
Here \(\varphi \circ \inn_k = \inn'_k\), because \(\pr'_j \circ \varphi \circ \inn_k = \pr_j \circ \inn_k\) for all \(j\), so \(\varphi^\dagger \circ \varphi \circ \inn_k = \inn_k\) and \(\varphi^\dagger \circ \varphi = \id\); the other composite is the identity by the symmetric argument.

\begin{proposition}
    \label{prop:orthogonal-biproduct-herm}
    In \(\Herm_K\), the orthogonal biproduct of \(V\) and \(W\) is \(V \oplus W\) with the form \(\langle v + w, v' + w' \rangle = \langle v, v' \rangle + \langle w, w' \rangle\), and its injections are orthogonal isometries.
\end{proposition}

\begin{proof}
    The given form on \(V \oplus W\) is Hermitian, and it is unimodular because its flat map is the direct sum of the flat maps of \(V\) and \(W\).
    Since \(\langle v, v' + w' \rangle = \langle v, v' \rangle\), the adjoint of the injection \(\inn_1\) is the projection \(\pr_1\), and likewise \(\inn_2^\dagger = \pr_2\), so the biproduct is orthogonal.
    Conversely, if some unimodular Hermitian form on the vector space \(V \oplus W\) has \(\inn_j^\dagger = \pr_j\), then
    \[
        \langle \inn_i(v), \inn_j(v') \rangle = \langle v, \pr_i(\inn_j(v')) \rangle
    \]
    vanishes for \(i \neq j\) and equals the form of \(V\) or \(W\) for \(i = j\), so the form is the one given.
    Any other orthogonal biproduct is unitarily isomorphic to this one.
\end{proof}

\begin{definition}
    \label{def:orthogonal-splitting}
    An \emph{orthogonal splitting} of an idempotent \(e \colon X \to X\) is an isometry \(i \colon Y \to X\) with \(i \circ i^\dagger = e\).
\end{definition}

An orthogonal splitting is a splitting of \(e\) in the usual sense, with the retraction \(i^\dagger\).
Like orthogonal biproducts, orthogonal splittings are unique up to a unitary when they exist \cite[Lemma~3.10]{selinger2008}.

\begin{proposition}
    \label{prop:orthogonal-splitting}
    An idempotent with an orthogonal splitting is a projector.
    In \(\Herm_K\), every projector has an orthogonal splitting.
\end{proposition}

\begin{proof}
    If \(e = i \circ i^\dagger\), then \(e^\dagger = i^{\dagger\dagger} \circ i^\dagger = e\).

    Let \(e\) be a projector on a Hermitian space \(V\), and let \(Y = e(V)\), with the restricted form.
    The subspaces \(e(V)\) and \((\id - e)(V)\) are orthogonal, because \(\langle e(v), w - e(w) \rangle = \langle v, e(w) - e(e(w)) \rangle = 0\), and together they span \(V\).
    If \(y \in Y\) is orthogonal to all of \(Y\), then \(\langle v, y \rangle = \langle e(v), y \rangle + \langle v - e(v), y \rangle = 0\) for every \(v \in V\), so \(y = 0\).
    The restricted form is therefore unimodular, and \(Y\) is a Hermitian space.
    Let \(i \colon Y \to V\) be the inclusion.
    For \(u \in Y\) and \(v \in V\) we have \(\langle i(u), v \rangle = \langle e(u), v \rangle = \langle u, e(v) \rangle\), so \(i^\dagger\) is \(e\) with its codomain restricted to \(Y\).
    Hence \(i \circ i^\dagger = e\), and \(i^\dagger \circ i = \id_Y\) because \(e\) is the identity on its image.
\end{proof}

The restriction of a unimodular form to a subspace need not be unimodular, for example on a line spanned by a vector \(v\) of the Minkowski space with \(\langle v, v \rangle = 0\).
The proof shows that this does not happen for the image of a projector.

\subsection{Orthogonal squares}
\label{sec:orthogonal-squares}

\begin{definition}
    \label{def:orthogonal-square}
    An \emph{orthogonal square} is a square of isometries
    \[
        \begin{tikzcd}
            W \rar["k"] \dar["h"'] & Y \dar["j"] \\
            X \rar["i"] & Z
        \end{tikzcd}
    \]
    with \(i \circ h = j \circ k\) and \(i^\dagger \circ j = h \circ k^\dagger\).
    With \(W = 0\), this says that \(i \perp j\).
\end{definition}

The first equation says that the square commutes.
The second says that the square obtained by reversing \(k\) and \(i\) with the dagger,
\[
    \begin{tikzcd}
        W \dar["h"'] & Y \lar["k^\dagger"'] \dar["j"] \\
        X & Z \lar["i^\dagger"]
    \end{tikzcd}
\]
also commutes.
Applying the dagger to every arrow turns an orthogonal square into a square of coisometries, and these are the \emph{independent squares} of \cite{stein2025, di-meglio-et-al2025}, where they describe conditional independence in probability; orthogonal squares themselves appear as the \emph{co-independent squares} of \cite{di-meglio-et-al2025}.

\begin{example}
    \label{ex:relatively-orthogonal-planes}
    Let \(V\) and \(W\) be planes in \(\R^3\) that meet in a line \(U\), each with the restriction of the dot product, and consider the square of inclusions
    \[
        \begin{tikzcd}
            U \rar["k"] \dar["h"'] & W \dar["j"] \\
            V \rar["i"] & \R^3
        \end{tikzcd}
    \]
    in \(\Herm_\R\).
    The square commutes.
    Choose unit vectors \(u \in U\), \(v \in V\) orthogonal to \(u\), and \(w \in W\) orthogonal to \(u\).
    Then \(u, v\) and \(u, w\) are orthonormal bases of \(V\) and \(W\), and the adjoints of the inclusions are the orthogonal projections, for example \(i^\dagger(v') = \langle v', u \rangle u + \langle v', v \rangle v\) for \(v' \in \R^3\).
    On the basis of \(W\), the two sides of the second equation are
    \[
        i^\dagger(j(u)) = u = h(k^\dagger(u)), \qquad
        i^\dagger(j(w)) = \langle w, v \rangle v, \qquad
        h(k^\dagger(w)) = 0.
    \]
    The square is orthogonal if and only if \(\langle v, w \rangle = 0\), that is, if and only if the planes are perpendicular (\cref{fig:relatively-orthogonal}).
    For planes meeting in \(U\), being perpendicular is a relative form of orthogonality: the planes are not orthogonal as subspaces, but their parts orthogonal to \(U\) are.
\end{example}

\begin{figure}[tbp]
    \centering
    \begin{subfigure}{0.4\textwidth}
        \centering
        \tdplotsetmaincoords{70}{110}
        \begin{tikzpicture}[tdplot_main_coords, fill opacity=0.2, scale=0.6]
            \draw[fill=red] (-3,0,-3) -- (-3,0,0) -- (3,0,0) -- (3,0,-3) -- cycle;
            \draw[fill=blue] (-3,-3,0) -- (-3,3,0) -- (3,3,0) -- (3,-3,0) -- cycle;
            \draw[fill=red] (-3,0,3) -- (-3,0,0) -- (3,0,0) -- (3,0,3) -- cycle;
            \draw[thick] (-3,0,0) -- (3,0,0);
            \coordinate (O) at (-3, 0, 0);
            \coordinate (A) at (-3, 0, 3);
            \coordinate (B) at (-3, 3, 0);
            \pic[draw] {right angle=A--O--B};
        \end{tikzpicture}
        \caption{}
    \end{subfigure}
    \quad
    \begin{subfigure}{0.4\textwidth}
        \centering
        \tdplotsetmaincoords{70}{110}
        \begin{tikzpicture}[tdplot_main_coords, fill opacity=0.2, scale=0.6]
            \draw[fill=green] (-3,2,-3) -- (-3,0,0) -- (3,0,0) -- (3,2,-3) -- cycle;
            \draw[fill=blue] (-3,-3,0) -- (-3,3,0) -- (3,3,0) -- (3,-3,0) -- cycle;
            \draw[fill=green] (-3,-2,3) -- (-3,0,0) -- (3,0,0) -- (3,-2,3) -- cycle;
            \draw[thick] (-3,0,0) -- (3,0,0);
            \coordinate (O) at (-3, 0, 0);
            \coordinate (A) at (-3, -2, 3);
            \coordinate (B) at (-3, 3, 0);
            \pic[draw] {angle=B--O--A};
        \end{tikzpicture}
        \caption{}
    \end{subfigure}
    \caption{Two planes meeting in a line: perpendicular (a) and not (b).}
    \label{fig:relatively-orthogonal}
\end{figure}
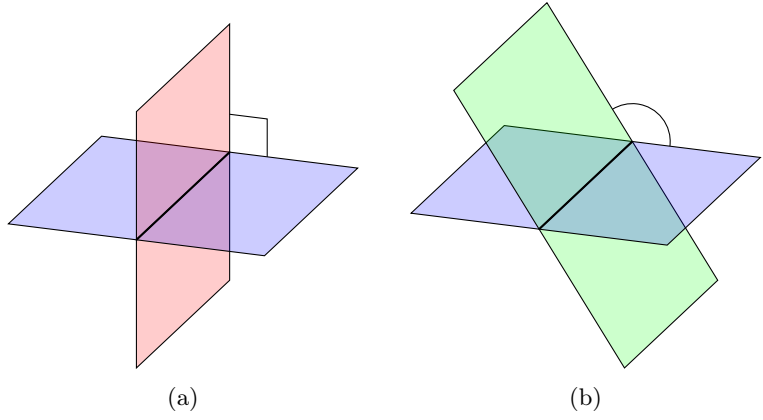

\begin{definition}
    \label{def:orthogonal-pushout}
    An orthogonal square is an \emph{orthogonal pushout} if every orthogonal square with the same upper-left corner \(h, k\) factors through it by a unique morphism.
\end{definition}

Explicitly, for the square of \cref{def:orthogonal-square} to be an orthogonal pushout means that for every orthogonal square formed by \(h\), \(k\) and isometries \(i' \colon X \to Z'\) and \(j' \colon Y \to Z'\), there is a unique morphism \(f \colon Z \to Z'\) making the diagram
\[
    \begin{tikzcd}
        W \rar["k"] \dar["h"'] & Y \dar["j"] \ar[ddr, "j'", bend left] & \\
        X \rar["i"'] \ar[drr, "i'"', bend right] & Z \drar["f", dashed] & \\
        & & Z'
    \end{tikzcd}
\]
commute, that is, \(f \circ i = i'\) and \(f \circ j = j'\).
A \emph{pushout} has the same property for all pairs of morphisms \(i'\), \(j'\) with \(i' \circ h = j' \circ k\).

In \(\Herm_K\) orthogonal squares can be described by subspaces.
For an isometry \(i\), the projector \(i \circ i^\dagger\) has the same image as \(i\), since \(i \circ i^\dagger \circ i = i\).

\begin{lemma}
    \label{lem:orthogonal-square-intersection}
    In \(\Herm_K\), the images of the two paths through an orthogonal square satisfy \(i(h(W)) = i(X) \cap j(Y)\).
\end{lemma}

\begin{proof}
    The inclusion \(i(h(W)) = j(k(W)) \subseteq i(X) \cap j(Y)\) holds because the square commutes.
    Using the second equation of \cref{def:orthogonal-square} and then the first,
    \[
        i \circ i^\dagger \circ j \circ j^\dagger
        = i \circ h \circ k^\dagger \circ j^\dagger
        = (i \circ h) \circ (j \circ k)^\dagger
        = (i \circ h) \circ (i \circ h)^\dagger.
    \]
    The left side is the identity on \(i(X) \cap j(Y)\), since both projectors fix it, and the right side has image \(i(h(W))\) because \(i \circ h\) is an isometry.
\end{proof}

\begin{proposition}
    \label{prop:orthogonal-pushout-herm}
    In \(\Herm_K\), the square
    \[
        \begin{tikzcd}
            0 \rar \dar & W \dar["\inn_2"] \\
            V \rar["\inn_1"'] & V \oplus W
        \end{tikzcd}
    \]
    formed by the injections of the orthogonal biproduct is an orthogonal pushout.
    An orthogonal square in \(\Herm_K\) is an orthogonal pushout if and only if it is a pushout.
\end{proposition}

\begin{proof}
    We prove the second statement first, for an orthogonal square as in \cref{def:orthogonal-square}.
    A pushout is in particular an orthogonal pushout, since the factorisation property is then required for all pairs \(i'\), \(j'\) with \(i' \circ h = j' \circ k\).

    Conversely, suppose that the square is an orthogonal pushout.
    We first show that \(i(X) + j(Y) = Z\).
    Otherwise there is a non-zero linear map \(\chi \colon Z \to Z\) that vanishes on \(i(X) + j(Y)\), and both \(\id_Z\) and \(\id_Z + \chi\) factor the square through itself, contradicting uniqueness.
    Now let \(i' \colon X \to Z'\) and \(j' \colon Y \to Z'\) be linear maps with \(i' \circ h = j' \circ k\).
    Every element of \(Z\) has the form \(i(v) + j(v')\), and we put \(f(i(v) + j(v')) = i'(v) + j'(v')\).
    This is well defined: if \(i(v) = -j(v')\), then this vector lies in \(i(X) \cap j(Y) = i(h(W))\) by \cref{lem:orthogonal-square-intersection}, so \(v = h(w)\) and \(-v' = k(w)\) for some \(w \in W\), since \(i\) and \(j\) are injective, and \(i'(v) + j'(v') = i'(h(w)) - j'(k(w)) = 0\).
    The map \(f\) is linear, satisfies \(f \circ i = i'\) and \(f \circ j = j'\), and is unique with these properties, so the square is a pushout.

    For the first statement, the injections of the orthogonal biproduct are orthogonal isometries by \cref{prop:orthogonal-biproduct-herm}, so the square is orthogonal, and it is a pushout because \(V \oplus W\) is a coproduct.
\end{proof}

In the factorisation through the biproduct square, the morphism is \(f = i' \circ \pr_1 + j' \circ \pr_2\), and it is an isometry, since \(f^\dagger \circ f = \inn_1 \circ \pr_1 + \inn_2 \circ \pr_2 = \id\) when \(i'\) and \(j'\) are orthogonal isometries.
The biproduct is thus the orthogonal sum of \(V\) and \(W\) in the geometric sense: the smallest space containing orthogonal copies of both.

\begin{proposition}
    \label{prop:pushout-sequence}
    Let \(i \colon U \to W\) and \(j \colon V \to W\) be orthogonal isometries in \(\Herm_K\).
    The square
    \[
        \begin{tikzcd}
            0 \rar \dar & V \dar["j"] \\
            U \rar["i"'] & W
        \end{tikzcd}
    \]
    is an orthogonal pushout if and only if
    \[
        \begin{tikzcd}
            0 \rar & U \rar["i"] & W \rar["j^\dagger"] & V \rar & 0
        \end{tikzcd}
    \]
    is short exact.
\end{proposition}

\begin{proof}
    The square is orthogonal because \(i \perp j\).
    By \cref{prop:orthogonal-pushout-herm} and its proof, it is an orthogonal pushout if and only if \(i(U) + j(V) = W\).
    The sequence is exact at \(U\) and at \(V\), since \(i\) is injective and \(j^\dagger\) is surjective, both because \(i^\dagger \circ i = \id\) and \(j^\dagger \circ j = \id\).
    Moreover \(j^\dagger \circ i = (i^\dagger \circ j)^\dagger = 0\), so the image of \(i\) lies in the kernel of \(j^\dagger\).

    If \(i(U) + j(V) = W\) and \(j^\dagger(i(u) + j(v)) = 0\), then \(v = 0\), so the kernel of \(j^\dagger\) is the image of \(i\).
    Conversely, if the sequence is exact, then for every \(w \in W\) the vector \(w - j(j^\dagger(w))\) lies in the kernel of \(j^\dagger\), hence equals \(i(u)\) for some \(u\), and \(w = i(u) + j(j^\dagger(w))\).
\end{proof}

\subsection{Orthogonal complements}
\label{sec:orthogonal-complements}

\begin{definition}
    \label{def:orthogonal-complement}
    An \emph{orthogonal complement} of an isometry \(i \colon X \to Z\) is an isometry \(i^\perp \colon X^\perp \to Z\) with \(i \perp i^\perp\), such that every morphism \(j \colon Y \to Z\) with \(i \perp j\) factors as \(j = i^\perp \circ f\) for a unique \(f\).
\end{definition}

In a diagram, the factorisation reads
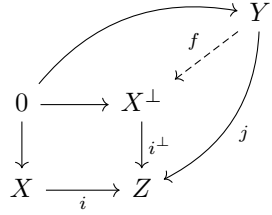
\[
    \begin{tikzcd}
        & & Y \ar[dl, "f"', dashed] \ar[ddl, "j", bend left] \\
        0 \rar \dar \ar[urr, bend left] & X^\perp \dar["i^\perp"] & \\
        X \rar["i"'] & Z &
    \end{tikzcd}
\]
where the square is orthogonal and \(i \perp j\).
When \(j\) is an isometry, so is the factoring morphism \(f\), because \(f^\dagger \circ f = f^\dagger \circ (i^\perp)^\dagger \circ i^\perp \circ f = j^\dagger \circ j = \id_Y\).
In particular, every isometry orthogonal to \(i\) factors through \(i^\perp\) by an isometry, so \(i^\perp\) is the largest isometry orthogonal to \(i\), which is the classical description of the orthogonal complement of a subspace.
Orthogonal complements are unique up to a unitary: if \(i^\perp\) and \(j\) are both orthogonal complements of \(i\), then \(j = i^\perp \circ f\) and \(i^\perp = j \circ f'\), and since isometries are monomorphisms, \(f \circ f'\) and \(f' \circ f\) are identities.

A \emph{kernel} of a morphism \(f \colon Z \to X\) is an equaliser of \(f\) and the zero morphism: a morphism \(k \colon Y \to Z\) with \(f \circ k = 0\) through which every \(g\) with \(f \circ g = 0\) factors uniquely.

\begin{definition}
    \label{def:isometric-kernel}
    An \emph{isometric kernel} of a morphism \(f\) is a kernel of \(f\) that is an isometry.
\end{definition}

\begin{proposition}
    \label{prop:complement-is-kernel}
    An orthogonal complement of an isometry \(i\) is the same as an isometric kernel of \(i^\dagger\).
\end{proposition}

\begin{proof}
    A morphism \(j\) is orthogonal to \(i\) if and only if \(i^\dagger \circ j = 0\), that is, if and only if \(j\) equalises \(i^\dagger\) and \(0\) in the diagram
    \[
        \begin{tikzcd}
            Y \rar["j"] & Z \rar[shift left, "i^\dagger"] \rar[shift right, "0"'] & X.
        \end{tikzcd}
    \]
    Hence the universal property of an orthogonal complement in \cref{def:orthogonal-complement} is the universal property of a kernel of \(i^\dagger\), and both notions ask for an isometry.
\end{proof}

\begin{proposition}
    \label{prop:complements-exist-spaces}
    In \(\Herm_K\) every isometry \(i\) has an orthogonal complement, namely an orthogonal splitting of the projector \(\id - i \circ i^\dagger\).
\end{proposition}

\begin{proof}
    The map \(e = \id - i \circ i^\dagger\) is self-adjoint, and it is idempotent because
    \[
        e \circ e = \id - 2\, i \circ i^\dagger + i \circ i^\dagger \circ i \circ i^\dagger = e.
    \]
    By \cref{prop:orthogonal-splitting}, it has an orthogonal splitting \(k\), with \(k \circ k^\dagger = e\) and \(k^\dagger \circ k = \id\).
    We show that \(k\) is an isometric kernel of \(i^\dagger\), which is an orthogonal complement of \(i\) by \cref{prop:complement-is-kernel}.
    First, \(i^\dagger \circ e = i^\dagger - i^\dagger \circ i \circ i^\dagger = 0\), so \(i^\dagger \circ k = i^\dagger \circ e \circ k = 0\), since \(e \circ k = k \circ k^\dagger \circ k = k\).
    If \(i^\dagger \circ g = 0\) for a linear map \(g\), then \(e \circ g = g\), so \(g = k \circ (k^\dagger \circ g)\), and this factorisation is unique because \(k\) is an isometry.
\end{proof}

\begin{corollary}
    \label{cor:complement-sequence}
    In \(\Herm_K\), an isometry \(j \colon V \to W\) is an orthogonal complement of an isometry \(i \colon U \to W\) if and only if \(i \perp j\) and
    \[
        \begin{tikzcd}
            0 \rar & U \rar["i"] & W \rar["j^\dagger"] & V \rar & 0
        \end{tikzcd}
    \]
    is short exact.
\end{corollary}

\begin{proof}
    By \cref{prop:complements-exist-spaces,prop:complement-is-kernel}, the orthogonal complements of \(i\) are the isometric kernels of \(i^\dagger\).
    Since the morphisms of \(\Herm_K\) are all linear maps, and the kernel of \(i^\dagger\) carries some unimodular Hermitian form (for example \(\sum_i v^i \conj{w^i}\) in a basis) and so can serve as a test object, an isometry \(j\) is a kernel of \(i^\dagger\) if and only if its image is the kernel of \(i^\dagger\).
    Let \(j\) be orthogonal to \(i\), so that \(j(V)\) is contained in the kernel of \(i^\dagger\).
    We claim that \(j(V)\) is the kernel of \(i^\dagger\) if and only if \(i(U) + j(V) = W\).
    If it is, then every \(w \in W\) decomposes as \(w = i(i^\dagger(w)) + (w - i(i^\dagger(w)))\), with the second term in the kernel of \(i^\dagger\) and hence in \(j(V)\).
    Conversely, if \(i(U) + j(V) = W\) and \(i^\dagger(i(u) + j(v)) = 0\), then \(u = 0\), so the kernel of \(i^\dagger\) is contained in \(j(V)\).
    By the proof of \cref{prop:pushout-sequence}, this holds if and only if the sequence is short exact.
\end{proof}

Together with \cref{prop:pushout-sequence}, the corollary says that \(j\) is an orthogonal complement of \(i\) if and only if the square formed by \(i\), \(j\) and the zero space is an orthogonal pushout.

\part{Bundles and connections}
\label{part:bundles}

\section{Vector bundles and bundle metrics}
\label{sec:vector-bundles}

In \cref{sec:vector-spaces} a dagger on all finite-dimensional vector spaces turned out to be a Hermitian form on every space.
In this section the single point is replaced by a manifold \(M\), and vector spaces by vector bundles over \(M\).
By the Serre--Swan theorem, bundles are the same as finitely generated projective modules over the ring of functions, so the structure we study is the category of such modules over a commutative ring \(A\).
We show in \cref{thm:daggers-are-metrics} that a dagger on it is an involution of \(A\) together with a unimodular Hermitian form on every module, which for \(A = C^\infty(M)\) is a bundle metric on every bundle, and that its unitaries are the isometric bundle isomorphisms.
The proofs of \cref{sec:vector-spaces} carry over, with dual bases in place of bases.
The last subsection shows that the forms also make the category dagger compact, with the conjugate module as the dual.

Throughout, \(A\) is a commutative ring and \(\fgpMod_A\) is the category of finitely generated projective \(A\)-modules and \(A\)-linear maps.
The dual of a module \(E\) is \(E^\vee = \Hom_A(E, A)\), the module of linear forms on \(E\).

For a vector bundle \(E\) over a manifold \(M\), the smooth sections form a module \(\Gamma(E)\) over the ring \(C^\infty(M)\) of smooth functions, and a bundle map \(E \to F\) gives a \(C^\infty(M)\)-linear map \(\Gamma(E) \to \Gamma(F)\).

\begin{theorem}[Serre--Swan, {\cite[Theorem~A.4]{fritz2024}}, {\cite[Theorems~12.29 and 12.32]{nestruev2020}}]
    \label{thm:serre-swan}
    For a manifold \(M\) with finitely many connected components, taking sections gives an equivalence between the category \(\VB_M\) of real (or complex) vector bundles over \(M\) and \(\fgpMod_A\) for \(A = C^\infty(M)\) (or \(C^\infty(M, \C)\)).
\end{theorem}

In particular, bundle maps \(E \to F\) are the same as \(A\)-linear maps \(\Gamma(E) \to \Gamma(F)\), and the dual bundle \(E^\vee\) has the dual module \(\Gamma(E)^\vee\) as its sections.
The correspondence is the smooth analogue of Swan's theorem for topological vector bundles over compact Hausdorff spaces \cite{swan1962}.
Here a vector bundle may have different ranks over different components of \(M\), matching the fact that a projective module over \(C^\infty(M)\) can have different ranks there; \cite{nestruev2020} treats connected \(M\), where the rank is constant.
The finiteness hypothesis ensures that modules of sections are finitely generated: over infinitely many components, a bundle whose rank grows from one component to the next has a module of sections that is not finitely generated.
Both references treat real bundles; their proofs apply verbatim to complex bundles and \(C^\infty(M, \C)\).
From now on, manifolds have finitely many connected components, and we identify a bundle with its module of sections.

\begin{definition}
    \label{def:dual-basis}
    A \emph{dual basis} of an \(A\)-module \(E\) is a finite family \(e_i \in E\), \(\xi^i \in E^\vee\) with \(u = \sum_i \xi^i(u)\, e_i\) for all \(u \in E\).
\end{definition}

A module has a dual basis if and only if it is finitely generated projective \cite[Theorem~A.2]{fritz2024}.
For a bundle, a dual basis is a finite family of sections \(e_i\) together with linear forms \(\xi^i\) that recover every section from its components \(\xi^i(u)\).
The values \(e_i(x)\) span each fibre but are in general linearly dependent, and a dual basis is a global replacement for a local frame.

\begin{example}
    \label{ex:sphere-dual-basis}
    Let \(\langle -, - \rangle\) be the dot product on \(\R^3\), and write a point of the sphere \(S^2 \subseteq \R^3\) as \(x = (x^1, x^2, x^3)\).
    Sections of the tangent bundle \(TS^2\) are the smooth maps \(u \colon S^2 \to \R^3\) with \(\langle u(x), x \rangle = 0\) for all \(x\).
    Write \(\partial_1, \partial_2, \partial_3\) for the standard basis of \(\R^3\), and let \(e_i(x) = \partial_i - x^i x\) be its orthogonal projection onto the tangent plane at \(x\) and \(\xi^i(u) = u^i\) the \(i\)-th component of \(u\).
    Then
    \[
        \sum_i \xi^i(u)\, e_i(x) = \sum_i u^i(x) \bigl(\partial_i - x^i x\bigr) = u(x) - \langle u(x), x \rangle\, x = u(x),
    \]
    so \((e_i, \xi^i)\) is a dual basis of \(\Gamma(TS^2)\) with three elements.
    By the hairy ball theorem, \(TS^2\) has no nowhere vanishing section, so it has no global frame.
\end{example}

\subsection{Hermitian modules}
\label{sec:hermitian-modules}

\begin{definition}
    \label{def:herm-A}
    Fix an involutive automorphism \(f \mapsto \conj f\) of \(A\).
    Sesquilinear, Hermitian and unimodular forms on an \(A\)-module are defined as in \cref{def:sesquilinear}, with values in \(A\).
    A \emph{Hermitian module} is a finitely generated projective module with a unimodular Hermitian form, and \(\Herm_A\) is the category of Hermitian modules and all \(A\)-linear maps.
    The module \(A\) carries the standard form \(\langle f, g \rangle = f \conj g\).
\end{definition}

\begin{example}
    \label{ex:bundle-metrics}
    Let \(A = C^\infty(M)\) with the identity involution.
    An \(A\)-bilinear map \(\Gamma(E) \times \Gamma(E) \to A\) is an \(A\)-linear map \(\Gamma(E \otimes E) \to A\), since \(\Gamma(E \otimes E) \cong \Gamma(E) \otimes_A \Gamma(E)\) \cite[Theorem~A.4]{fritz2024}, so by \cref{thm:serre-swan} it is a bundle map from \(E \otimes E\) to the trivial line bundle \(M \times \R\): a family of bilinear forms \(\langle -, - \rangle_x\) on the fibres, with \(\langle u, v \rangle(x) = \langle u(x), v(x) \rangle_x\) smooth for all sections \(u\) and \(v\).
    For a symmetric form, \(\flat\) is a bundle map \(E \to E^\vee\), and by \cref{thm:serre-swan} it is bijective on sections if and only if it is an isomorphism of bundles, that is, if and only if every \(\langle -, - \rangle_x\) is unimodular.
    Hence a unimodular symmetric form on \(\Gamma(E)\) is a \emph{pseudo-Riemannian bundle metric} on \(E\), of any signature, and a Riemannian bundle metric when it is positive definite.
    We call a real bundle with a pseudo-Riemannian bundle metric a \emph{metric bundle}.
\end{example}

\begin{example}
    \label{ex:hermitian-bundle-metrics}
    For a complex bundle \(E\) and \(A = C^\infty(M, \C)\) with complex conjugation, the argument of \cref{ex:bundle-metrics} shows that a unimodular Hermitian form on \(\Gamma(E)\) is a smooth family of unimodular Hermitian forms on the fibres, that is, a Hermitian bundle metric of any signature.
\end{example}

\begin{proposition}
    \label{prop:herm-module-dagger}
    Adjoints make \(\Herm_A\) a dagger category, with \((f \varphi)^\dagger = \conj f\, \varphi^\dagger\).
\end{proposition}

\begin{proof}
    The proof of \cref{prop:herm-dagger} applies word for word, with \(A\) in place of \(K\).
    It uses only that \(\flat\) is bijective, so that each linear form \(\langle \varphi(-), w \rangle\) equals \(\langle -, v \rangle\) for a unique \(v\), and that the forms are Hermitian.
\end{proof}

As in \cref{def:elements-as-maps}, an element \(u\) of a module \(E\) gives the linear map \(\hat u \colon A \to E\), \(f \mapsto f u\), and every linear map \(A \to E\) is of this form.
The computation in the proof of \cref{lem:flat-is-dagger} shows that in \(\Herm_A\) the adjoint of \(\hat v\) is \(v^\flat\), so that \(\langle u, v \rangle = \hat v^\dagger \circ \hat u\) as an element of \(\End_A(A) = A\).
A linear map \(i \colon E \to F\) of Hermitian modules is an isometry if and only if \(\langle i(u), i(v) \rangle = \langle u, v \rangle\) for all \(u, v \in E\).
For metric bundles this says that \(i\) restricts to a linear isometric embedding \(E_x \to F_x\) on every fibre, and the unitaries are the bundle isomorphisms that preserve the metrics.

\begin{theorem}
    \label{thm:daggers-are-metrics}
    Let \(\dagger\) be a dagger on \(\fgpMod_A\).
    There are a unique involutive automorphism of \(A\) and a unique family of unimodular Hermitian forms, one on every finitely generated projective module, whose adjoints are \(\dagger\) and whose form on \(A\) is the standard one.
\end{theorem}

\begin{proof}
    The proof is that of \cref{thm:daggers-are-forms}, with \(A\) in place of \(K\), and we only give the steps that change.
    The category \(\fgpMod_A\) has a zero object and biproducts, and the proof of \cref{prop:dagger-additive} uses nothing else, so the dagger is additive.
    It restricts to \(\End_A(A) = A\), where, as before, it is an involutive automorphism \(f \mapsto \conj f\) of the commutative ring \(A\).

    Next we show that \((f \id_E)^\dagger = \conj f \id_E\) for every module \(E\) and every \(f \in A\).
    Applying the dagger to \((f \id_E) \circ \hat v = \hat v \circ f\) gives \(\hat v^\dagger \circ (f \id_E)^\dagger = \conj f\, \hat v^\dagger\).
    The dagger is a bijection from \(\Hom_A(A, E)\) to \(\Hom_A(E, A) = E^\vee\), so every linear form \(\alpha\) on \(E\) is \(\hat v^\dagger\) for some \(v\), and \(\alpha \circ (f \id_E)^\dagger = \conj f\, \alpha\).
    In \cref{thm:daggers-are-forms} we concluded because linear forms separate points.
    For a module, we take a dual basis \((e_i, \xi^i)\) of \(E\) and apply this to the forms \(\xi^i\):
    \[
        (f \id_E)^\dagger(u) = \sum_i \xi^i\bigl((f \id_E)^\dagger(u)\bigr)\, e_i = \sum_i \conj f\, \xi^i(u)\, e_i = \conj f u .
    \]
    Hence \((f \varphi)^\dagger = \conj f\, \varphi^\dagger\) for every linear map \(\varphi\).

    The forms are \(\langle u, v \rangle = \hat v^\dagger \circ \hat u\) in \(\End_A(A) = A\), and the arguments of \cref{thm:daggers-are-forms} apply verbatim and show that they are sesquilinear and Hermitian, that they are standard on \(A\), and that \(\dagger\) is their adjoint.
    For unimodularity, note that \(\hat v^\dagger \circ \hat u\) is multiplication by \(\hat v^\dagger(u)\), so \(v^\flat = \hat v^\dagger\).
    Thus \(\flat\) is the composite of the bijections \(E \to \Hom_A(A, E)\), \(v \mapsto \hat v\), and \(\Hom_A(A, E) \to E^\vee\), \(\hat v \mapsto \hat v^\dagger\).
    Uniqueness follows from \(\langle u, v \rangle = \hat v^\dagger \circ \hat u\) in \(\Herm_A\), as for vector spaces.
\end{proof}

As in \cref{lem:flat-is-dagger}, the dagger is then compared with the duality functor by the square \(\varphi^\vee \circ \flat = \flat \circ \varphi^\dagger\), which holds in \(\Herm_A\) with the same proof.

Together with \cref{prop:herm-module-dagger}, the theorem says that a dagger on \(\fgpMod_A\) is the same thing as a way of making it into \(\Herm_A\), for some involution of \(A\), with the standard form on \(A\).
The dagger is \(A\)-linear if and only if the involution is the identity, since \((f \id_A)^\dagger = \conj f \id_A\).

\begin{corollary}
    \label{cor:daggers-on-bundles}
    For real bundles, the \(C^\infty(M)\)-linear daggers on \(\VB_M\) correspond bijectively to families of pseudo-Riemannian bundle metrics, one on every bundle and standard on the trivial line bundle.
\end{corollary}

\begin{proof}
    By \cref{thm:serre-swan}, taking sections is fully faithful, and it sends the trivial line bundle to \(A = C^\infty(M)\), direct sums of bundles to biproducts and the zero bundle to the zero module.
    Hence the proof of \cref{thm:daggers-are-metrics} applies to \(\VB_M\) word for word, with \(\Gamma(E)\) in place of \(E\).
    A \(C^\infty(M)\)-linear dagger has the identity as its involution, and by \cref{ex:bundle-metrics} a unimodular symmetric form on \(\Gamma(E)\) is a pseudo-Riemannian bundle metric on \(E\).
\end{proof}

\begin{remark}
    \label{rem:nonlinear-bundle-daggers}
    A dagger on the category \(\VB_M\) of real vector bundles that is not \(C^\infty(M)\)-linear has a non-trivial involution of \(C^\infty(M)\).
    If this involution is \(\R\)-linear, as it is when the dagger is \(\R\)-linear, then it is \(f \mapsto f \circ \tau\) for a unique diffeomorphism \(\tau\) of \(M\) \cite[Theorem~1]{mrcun2005}, and \(\tau\) is an involution.
    The form then satisfies \(\langle u, f v \rangle = (f \circ \tau) \langle u, v \rangle\), and its value at \(x\) depends only on \(u(x)\) and \(v(\tau(x))\): it pairs the fibre \(E_x\) with the fibre \(E_{\tau(x)}\).
    For the antipodal map of the sphere, such a form pairs the fibres over antipodal points.
    For complex bundles, a \(\C\)-antilinear involutive automorphism of \(C^\infty(M, \C)\) is \(f \mapsto \conj{f \circ \tau}\), by the same theorem applied to its composite with complex conjugation, and \(\tau = \id\) gives the Hermitian bundle metrics of \cref{ex:hermitian-bundle-metrics}.
    We do not classify the daggers with \(\tau \neq \id\); which bundles admit unimodular pairings between the fibres over \(x\) and \(\tau(x)\), and hence when such daggers exist, we leave open.
\end{remark}

\begin{remark}
    \label{rem:existence-of-metrics}
    Every real vector bundle has a Riemannian bundle metric \cite[Lemma~12.25]{nestruev2020}, and the same partition-of-unity argument gives every complex bundle a positive definite Hermitian metric.
    By \cref{prop:herm-module-dagger}, \(\VB_M\) therefore has daggers, both linear ones and, for complex bundles, ones with complex conjugation as involution.
    Over a general ring, daggers with a given involution need not exist.
    For the identity involution, a unimodular symmetric form on a projective module \(E\) of rank one is an isomorphism \(\flat \colon E \cong E^\vee\), and \(E^\vee\) is the inverse of \(E\) in the Picard group \(\Pic(A)\), the group of isomorphism classes of projective modules of rank one under the tensor product, so the class of \(E\) has order at most two.
    For the coordinate ring \(A\) of a complex elliptic curve with one point removed, \(\Pic(A)\) is isomorphic to the group of points of the curve, which has elements of infinite order.
    The corresponding modules have no unimodular symmetric form, so by \cref{thm:daggers-are-metrics} there is no dagger on \(\fgpMod_A\) whose involution is the identity.
\end{remark}

\subsection{Dagger compact structure}
\label{sec:dagger-compact}

The tensor product of bundles is compatible with the metrics, and so is duality.
In \(\fgpMod_A\) every module \(E\) already has a dual \(E^\vee\), with evaluation \(\ev \colon E^\vee \otimes E \to A\), \(\alpha \otimes u \mapsto \alpha(u)\), and coevaluation \(\coev \colon A \to E \otimes E^\vee\), \(\coev(1) = \sum_i e_i \otimes \xi^i\) for any dual basis; the finitely generated projective modules are exactly the dualisable objects of the monoidal category of \(A\)-modules \cite[Example~1.4]{dold-puppe1984}, \cite[Theorem~A.2]{fritz2024}.
With forms, the dual of a Hermitian module can be taken to be its conjugate.

\begin{definition}
    \label{def:conjugate-module}
    The \emph{conjugate} of an \(A\)-module \(E\) is the module \(\conj E\) with the same underlying group, with elements written \(\conj u\), and with the action \(f \cdot \conj u = \conj{\conj f u}\).
    If \(E\) is Hermitian, so is \(\conj E\), with \(\langle \conj u, \conj v \rangle = \langle v, u \rangle\).
\end{definition}

For the trivial involution, \(\conj E = E\), so every Hermitian module becomes self-dual; for the coherence questions raised by such self-dualities, see \cite{selinger2010}.
If \((e_i, \xi^i)\) is a dual basis of \(E\), then \(\conj u = \sum_i \conj{\xi^i(u)} \cdot \conj{e_i}\), so \(\conj E\) is again finitely generated projective.
The form on \(\conj E\) is sesquilinear and Hermitian because the form on \(E\) is, and it is unimodular: a linear form \(\beta\) on \(\conj E\) gives the linear form \(u \mapsto \conj{\beta(\conj u)}\) on \(E\), which is \(\langle -, v \rangle\) for a unique \(v\), and then \(\beta = \langle -, \conj v \rangle\).
The map \(\conj v \mapsto v^\flat\) is an \(A\)-linear isomorphism \(\conj E \cong E^\vee\).

\begin{definition}
    \label{def:reciprocal-basis}
    Let \(E\) be Hermitian with dual basis \((e_i, \xi^i)\).
    The \emph{reciprocal basis} is \(e^i = (\xi^i)^\sharp\), so that \(\langle u, e^i \rangle = \xi^i(u)\).
\end{definition}

For an orthonormal basis (\cref{def:orthonormal-basis}) of a Hermitian space, with its dual basis, the reciprocal basis is \(e^i = \langle e_i, e_i \rangle\, e_i\).
In \cref{ex:sphere-dual-basis}, with the metric on \(TS^2\) induced by the dot product, \(\xi^i(u) = \langle u, \partial_i \rangle = \langle u, e_i \rangle\) because \(u(x)\) is orthogonal to \(x\), so the dual basis of \(\Gamma(TS^2)\) is its own reciprocal basis, \(e^i = e_i\).

\begin{proposition}
    \label{prop:reciprocal-expansion}
    For every \(u\) in a Hermitian module,
    \[
        u = \sum_i \langle u, e^i \rangle\, e_i = \sum_i \langle u, e_i \rangle\, e^i .
    \]
\end{proposition}

\begin{proof}
    The first expansion is the definition of a dual basis, since \(\langle u, e^i \rangle = \xi^i(u)\).
    For the second, we compare both sides after applying \(\flat\).
    For every \(w\), using that the form is semilinear in the second variable and Hermitian, and then the first expansion of \(w\),
    \[
        \Bigl\langle w, \sum_i \langle u, e_i \rangle\, e^i \Bigr\rangle
        = \sum_i \langle e_i, u \rangle \langle w, e^i \rangle
        = \Bigl\langle \sum_i \langle w, e^i \rangle\, e_i, u \Bigr\rangle
        = \langle w, u \rangle .
    \]
    Since \(\flat\) is injective, \(\sum_i \langle u, e_i \rangle\, e^i = u\).
\end{proof}

The second expansion says that \((e^i, e_i^\flat)\) is again a dual basis, whose reciprocal basis is \((e_i)\).
For the sphere, both expansions read \(u = \sum_i \langle u, e_i \rangle\, e_i\).

\begin{definition}[\cite{abramsky-coecke2004, selinger2007}]
    \label{def:dagger-compact}
    A symmetric monoidal category with a dagger, with unit object \(I\), is \emph{dagger symmetric monoidal} if \((\varphi \otimes \psi)^\dagger = \varphi^\dagger \otimes \psi^\dagger\) and the associators, unitors and symmetries \(\sigma_{X, Y} \colon X \otimes Y \to Y \otimes X\) are unitary.
    It is \emph{dagger compact} if moreover every object \(X\) has a dual object \(Y\), with morphisms \(\eta \colon I \to Y \otimes X\) and \(\varepsilon \colon X \otimes Y \to I\) such that the triangles
    \[
        \begin{tikzcd}
            X \rar["\id \otimes \eta"] \drar[equals] & X \otimes Y \otimes X \dar["\varepsilon \otimes \id"] \\
            & X
        \end{tikzcd}
        \qquad
        \begin{tikzcd}
            Y \rar["\eta \otimes \id"] \drar[equals] & Y \otimes X \otimes Y \dar["\id \otimes \varepsilon"] \\
            & Y
        \end{tikzcd}
    \]
    commute, up to associators and unitors, and such that \(\varepsilon = \eta^\dagger \circ \sigma_{X, Y}\).
\end{definition}

\begin{theorem}
    \label{thm:herm-dagger-compact}
    With the tensor product over \(A\) and the form \(\langle u \otimes u', v \otimes v' \rangle = \langle u, v \rangle \langle u', v' \rangle\), the category \(\Herm_A\) is dagger compact.
    The dual of \(E\) is \(\conj E\), with
    \[
        \varepsilon \colon E \otimes \conj E \to A, \quad u \otimes \conj v \mapsto \langle u, v \rangle,
        \qquad
        \eta \colon A \to \conj E \otimes E, \quad 1 \mapsto \sum_i \conj{e^i} \otimes e_i .
    \]
\end{theorem}

The proof is in \cref{sec:appendix-compact}; the triangle identities are the two expansions of \cref{prop:reciprocal-expansion}.
Under the isomorphism \(\conj E \cong E^\vee\), \(\conj v \mapsto v^\flat\), the cup \(\eta\) becomes \(1 \mapsto \sum_i \xi^i \otimes e_i\), which is \(\coev\) followed by the symmetry, and the cap \(\varepsilon\) becomes \(\ev\) preceded by the symmetry.
The compact structure is therefore the one of \(\fgpMod_A\) given by \(\ev\) and \(\coev\), and the form contributes the identification \(\conj E \cong E^\vee\) and the equation \(\varepsilon = \eta^\dagger \circ \sigma_{E, \conj E}\), which relates the cap to the cup.

\section{Orthogonal complements and normal bundles}
\label{sec:orthogonal-bundles}

\Cref{sec:orthogonality} expressed orthogonality, orthogonal squares and orthogonal complements by equations between morphisms and their daggers, in any dagger category with a zero object.
These definitions apply verbatim in \(\Herm_A\), and in particular to metric bundles.
We show that they have their expected fibrewise meaning, and that orthogonal complements exist over every commutative ring with involution, as images of projectors.
The main example is the normal bundle of a submanifold, which is the orthogonal complement of its tangent bundle.

Throughout, \(A = C^\infty(M)\) with the trivial involution, unless stated otherwise, and bundles carry bundle metrics.

An isometry \(i \colon E \to H\) of metric bundles satisfies \(i^\dagger \circ i = \id_E\), so it is injective on every fibre, and its image is a subbundle of \(H\) \cite[Proposition~12.20]{nestruev2020}.
The adjoint of a bundle map \(\varphi\) is computed fibrewise, because \(\langle \varphi(u), w \rangle\) takes the value \(\langle \varphi(u(x)), w(x) \rangle_x\) at \(x\), and every vector in a fibre is the value of a section \cite[Lemma~12.8]{nestruev2020}.

\begin{proposition}
    \label{prop:orthogonal-subbundles}
    Isometries \(i \colon E \to H\) and \(j \colon F \to H\) of metric bundles have fibrewise orthogonal images if and only if \(i \perp j\), that is, \(i^\dagger \circ j = 0\).
\end{proposition}

\begin{proof}
    For sections \(u\) of \(E\) and \(v\) of \(F\), and every point \(x\),
    \[
        \langle i(u(x)), j(v(x)) \rangle_x = \langle u, (i^\dagger \circ j)(v) \rangle(x).
    \]
    If \(i^\dagger \circ j = 0\), the right-hand side vanishes, and since every vector in \(E_x\) and \(F_x\) is the value of a section, the images of \(E_x\) and \(F_x\) are orthogonal in \(H_x\).
    Conversely, if the images are orthogonal, then \(\langle u, (i^\dagger \circ j)(v) \rangle = 0\) for all \(u\), and since the form of \(E\) is unimodular, \((i^\dagger \circ j)(v) = 0\) for all \(v\).
\end{proof}

\begin{proposition}
    \label{prop:complements-exist-modules}
    In \(\Herm_A\), for any commutative ring \(A\) with involution, every isometry \(i \colon E \to H\) has an orthogonal complement, an orthogonal splitting of the projector \(\id_H - i \circ i^\dagger\).
    For metric bundles (\(A = C^\infty(M)\)), it is the subbundle with fibres
    \[
        E^\perp_x = \{ w \in H_x \mid \langle i(v), w \rangle_x = 0 \text{ for all } v \in E_x \}.
    \]
\end{proposition}

\begin{proof}
    Over a field, \cref{prop:orthogonal-splitting} splits the projector by its image with the restricted form.
    We do the same here; the new point is that \(\flat\) on the image is surjective, which over a field followed from injectivity by counting dimensions.

    Write \(p = \id_H - i \circ i^\dagger\), the projector called \(e\) in \cref{prop:complements-exist-spaces}.
    It is self-adjoint, and \(p \circ p = p\) because \(i^\dagger \circ i = \id_E\), so it is a projector.
    Its image \(E^\perp\) is the kernel of \(i^\dagger\): on the one hand \(i^\dagger \circ p = i^\dagger - i^\dagger \circ i \circ i^\dagger = 0\), and on the other hand \(p(w) = w\) whenever \(i^\dagger(w) = 0\).
    Since \(H = E^\perp \oplus \Ker p\), the module \(E^\perp\) is a direct summand of \(H\), hence finitely generated projective.
    We give it the restriction of the form of \(H\), which is Hermitian.

    To see that it is unimodular, let \(\alpha\) be a linear form on \(E^\perp\).
    Then \(\alpha \circ p\) is a linear form on \(H\), so it equals \(\langle -, w \rangle\) for a unique \(w \in H\), and for \(u \in E^\perp\),
    \[
        \alpha(u) = \langle p(u), w \rangle = \langle u, p(w) \rangle,
    \]
    with \(p(w) \in E^\perp\).
    Hence \(\flat\) is surjective on \(E^\perp\).
    It is injective, because if \(v \in E^\perp\) satisfies \(\langle u, v \rangle = 0\) for all \(u \in E^\perp\), then \(\langle w, v \rangle = \langle w, p(v) \rangle = \langle p(w), v \rangle = 0\) for all \(w \in H\), so \(v = 0\).

    Let \(j \colon E^\perp \to H\) be the inclusion.
    It is an isometry because the form of \(E^\perp\) is the restricted one.
    For \(u \in E^\perp\) and \(w \in H\), we have \(\langle u, w \rangle = \langle p(u), w \rangle = \langle u, p(w) \rangle\), so \(j^\dagger(w) = p(w)\), and \(j \circ j^\dagger = p\).
    Thus \(j\) is an orthogonal splitting of \(p\).
    It is a kernel of \(i^\dagger\), since \(E^\perp = \Ker i^\dagger\) and every linear map into \(H\) that is killed by \(i^\dagger\) lands in \(E^\perp\).
    By \cref{prop:complement-is-kernel}, \(j\) is an orthogonal complement of \(i\).

    For bundles, a section \(w\) of \(H\) lies in \(\Ker i^\dagger\) if and only if \(\langle i(v), w \rangle = 0\) for all sections \(v\) of \(E\), that is, if and only if \(w(x)\) is orthogonal to the image of \(E_x\) at every point.
    Since \(p\) is a bundle map with \(p_x \circ p_x = p_x\), its fibres \(E^\perp_x = p_x(H_x)\) have locally constant dimension \(\tr(p_x)\), so \(E^\perp\) is a subbundle \cite[Proposition~12.20]{nestruev2020}, with the stated fibres.
\end{proof}

As for vector spaces in \cref{cor:complement-sequence}, the orthogonal complement \(j \colon E^\perp \to H\) gives a short exact sequence
\[
    \begin{tikzcd}
        0 \rar & E \rar["i"] & H \rar["j^\dagger"] & E^\perp \rar & 0 .
    \end{tikzcd}
\]
Here \(j^\dagger\) is surjective because \(j^\dagger \circ j = \id\), and its kernel is the kernel of \(p = j \circ j^\dagger\), which is the image of \(i \circ i^\dagger\), that is, of \(i\).

\begin{example}
    \label{ex:orthogonal-squares-bundles}
    Consider an orthogonal square of metric bundles
    \[
        \begin{tikzcd}
            W \rar["k"] \dar["h"'] & F \dar["j"] \\
            E \rar["i"] & H
        \end{tikzcd}
    \]
    as in \cref{def:orthogonal-square}, and regard \(W\), \(E\) and \(F\) as subbundles of \(H\), with \(W\) contained in \(E\) and \(F\).
    Let \(k^\perp\) be the orthogonal complement of \(k\) in \(F\), from \cref{prop:complements-exist-modules}, so that \(k \circ k^\dagger + k^\perp \circ (k^\perp)^\dagger = \id_F\).
    Using \(i \circ h = j \circ k\) and \(i^\dagger \circ i = \id_E\),
    \[
        i^\dagger \circ j
        = i^\dagger \circ j \circ k \circ k^\dagger + i^\dagger \circ j \circ k^\perp \circ (k^\perp)^\dagger
        = h \circ k^\dagger + i^\dagger \circ j \circ k^\perp \circ (k^\perp)^\dagger .
    \]
    Since \((k^\perp)^\dagger\) is surjective, the second diagonal \(i^\dagger \circ j = h \circ k^\dagger\) holds if and only if \(i^\dagger \circ (j \circ k^\perp) = 0\).
    By \cref{prop:orthogonal-subbundles}, this says that the orthogonal complement of \(W_x\) in \(F_x\) is orthogonal to \(E_x\), or equivalently to the orthogonal complement of \(W_x\) in \(E_x\), at every point.
    This is the relative orthogonality of \cref{ex:relatively-orthogonal-planes}, fibre by fibre.
\end{example}

\begin{example}
    \label{ex:orthogonal-square-circle}
    Over the circle \(S^1 \subseteq \R^2 \subseteq \R^3\), let \(H = S^1 \times \R^3\) with the dot product, and let \(W\) be spanned by the constant section \(\partial_3\), the third standard basis vector.
    For \(x \in S^1\), write \(x'\) for \(x\) turned by a right angle in the plane, and let \(c\) be a smooth function on \(S^1\).
    Let \(E_x\) be spanned by \(\partial_3\) and \(x\), and \(F_x\) by \(\partial_3\) and \(\cos(c(x))\, x + \sin(c(x))\, x'\), so that the plane \(F_x\) turns around the vertical axis as \(x\) moves.
    The orthogonal complement of \(W_x\) in \(F_x\) is spanned by \(\cos(c(x))\, x + \sin(c(x))\, x'\), whose inner product with \(x\) is \(\cos(c(x))\).
    By \cref{ex:orthogonal-squares-bundles}, the planes \(E_x\) and \(F_x\) are orthogonal relative to \(W_x\) if and only if \(\cos(c(x)) = 0\), and the inclusions form an orthogonal square if and only if this holds everywhere.
\end{example}

\begin{example}
    \label{ex:normal-bundle}
    Let \(\iota \colon M \to N\) be an isometric immersion of (pseudo-)\hspace{0pt}Riemannian manifolds, so that its differential is an isometry \(TM \to \iota^* TN\) of metric bundles over \(M\).
    The \emph{normal bundle} of \(M\) in \(N\) is the orthogonal complement of this isometry.
    Here \(\iota^* TN\) is the pullback of \(TN\) along \(\iota\), with fibre \(T_{\iota(x)} N\) over \(x\), and we write \(d\iota \colon TM \to \iota^* TN\) for the differential and \(TM^\perp\) for the normal bundle.
    By \cref{prop:complements-exist-modules} the normal bundle exists, and its fibre at \(x\) is the orthogonal complement of \(T_x M\) in \(T_{\iota(x)} N\).
    It sits in the short exact sequence
    \[
        \begin{tikzcd}
            0 \rar & TM \rar["d\iota"] & \iota^* TN \rar["j^\dagger"] & TM^\perp \rar & 0 ,
        \end{tikzcd}
    \]
    so it is a cokernel of \(d\iota\), and in addition it comes with the isometry \(j\) into \(\iota^* TN\), whose image is orthogonal to \(TM\).
    Its universal property involves only isometries and the dagger, so it is preserved by unitary bundle isomorphisms; for instance, a unitary automorphism of \(\iota^* TN\) that maps \(TM\) onto itself maps \(TM^\perp\) onto itself.
\end{example}

\begin{example}
    \label{ex:normal-bundle-sphere}
    For an isometric immersion \(\iota\) into \(N = \R^n\), the bundle \(\iota^* T\R^n\) is the trivial bundle \(M \times \R^n\) with the dot product.
    For the sphere \(S^2 \subseteq \R^3\), the projector \(\id - d\iota \circ d\iota^\dagger\) is \(w \mapsto \langle w, x \rangle\, x\) on the fibre over \(x\), and the normal bundle is the trivial line bundle \(S^2 \times \R\), with the isometry \(j(x, c) = c\, x\) into \(S^2 \times \R^3\) and \(j^\dagger(w) = \langle w, x \rangle\).
    The complementary projector \(d\iota \circ d\iota^\dagger\) is the orthogonal projection onto the tangent plane, and the dual basis of \cref{ex:sphere-dual-basis} is \(e_i = d\iota^\dagger(\partial_i)\), \(\xi^i = \langle d\iota(-), \partial_i \rangle\).
    In the same way, every isometry \(i\) of a metric bundle into a trivial bundle \(M \times \R^n\) with the dot product gives a dual basis \(e_i = i^\dagger(\partial_i)\), \(\xi^i = \langle i(-), \partial_i \rangle\), which is its own reciprocal basis.
\end{example}

\section{Lie groupoids and parallel transport}
\label{sec:lie-groupoids}

In \cref{sec:lie-groups}, a form on \(V\) made \(\GL(V)\) a dagger Lie group, and a representation was unitary when it preserved that form.
A Lie groupoid is a Lie group with many objects: its morphisms \(g \colon x \to y\) go between points of a manifold \(M\), and a Lie group is a Lie groupoid over a point.
This section repeats the constructions of \cref{sec:lie-groups} for Lie groupoids, each one introduced by comparison with its counterpart for groups.
The general linear group \(\GL(V)\) is replaced by the general linear groupoid \(\GL(E)\) of a vector bundle, whose morphisms are the linear isomorphisms between fibres, and a metric on \(E\) gives it a dagger, the fibrewise adjoint.
The new phenomenon is parallel transport: a connection on \(E\) gives a representation of the groupoid of paths in \(M\), and \cref{thm:metric-iff-dagger-transport} says that this representation is unitary if and only if the connection is metric.

Throughout, bundles are real vector bundles over a manifold \(M\) with bundle metrics, as in \cref{cor:daggers-on-bundles}; the complex Hermitian case is identical.

\begin{definition}
    \label{def:lie-groupoid}
    A \emph{Lie groupoid} \(G \rightrightarrows M\) is a groupoid whose morphisms form a manifold \(G\) and whose objects form the manifold \(M\), with smooth structure maps and with source and target \(s, t \colon G \to M\) surjective submersions.
    We write \(g \colon x \to y\) when \(s(g) = x\) and \(t(g) = y\).
\end{definition}

The structure maps are the composition \((h, g) \mapsto h \circ g\), defined when \(s(h) = t(g)\), the units \(x \mapsto \id_x\) and the inversion \(g \mapsto g^{-1}\).
A morphism of Lie groupoids over \(M\) is a smooth functor that is the identity on objects.
A standard reference is \cite{crainic-fernandes2011}.

\begin{example}
    \label{ex:lie-group-groupoid}
    A Lie group is a Lie groupoid over a point, and everything in this section reduces to \cref{sec:lie-groups} in this case.
\end{example}

\begin{example}
    \label{ex:trivial-groupoid}
    The \emph{trivial groupoid} \(M \rightrightarrows M\) has only identity morphisms.
\end{example}

\begin{example}
    \label{ex:pair-groupoid}
    The \emph{pair groupoid} \(M \times M \rightrightarrows M\) has one morphism \((x, y) \colon x \to y\) between any two points, with \((y, z) \circ (x, y) = (x, z)\).
\end{example}

\begin{example}
    \label{ex:fundamental-groupoid}
    The \emph{fundamental groupoid} \(\Pi_1(M) \rightrightarrows M\) has homotopy classes of paths \(\gamma\) from \(x\) to \(y\), relative to the endpoints, as morphisms \(x \to y\), composed by concatenation.
    Its loops at \(x\) form the fundamental group, and for simply connected \(M\) it is the pair groupoid.
\end{example}

\begin{example}
    \label{ex:action-groupoid}
    A smooth action of a Lie group \(H\) on \(M\), written \((g, x) \mapsto g \cdot x\), gives the \emph{action groupoid} \(H \ltimes M \rightrightarrows M\), with morphisms \((g, x) \colon x \to g \cdot x\) and composition \((h, g \cdot x) \circ (g, x) = (h g, x)\).
\end{example}

\begin{example}
    \label{ex:lie-groupoids}
    For a vector bundle \(E\), the \emph{general linear groupoid} \(\GL(E) \rightrightarrows M\) has as morphisms \(x \to y\) the linear isomorphisms \(E_x \to E_y\), composed as maps.
    Local frames of \(E\) around \(x\) and \(y\) identify the morphisms near \(x\) and \(y\) with invertible matrices, and this gives \(\GL(E)\) its smooth structure.
    Over a point, it is the general linear group \(\GL(V)\).
\end{example}

\begin{example}
    \label{ex:gauge-groupoid}
    Let \(B \to M\) be a principal \(H\)-bundle \cite{bleecker1981}, with fibres \(B_x\).
    Its \emph{gauge groupoid} \((B \times B)/H \rightrightarrows M\) is the quotient of \(B \times B\) by the diagonal action of \(H\), and its morphisms \(x \to y\) are the \(H\)-equivariant maps \(B_x \to B_y\).
    If \(E\) has constant rank \(n\), then \(\GL(E)\) is the gauge groupoid of its frame bundle, the principal \(\GL(n, \R)\)-bundle whose fibre over \(x\) is the set of bases of \(E_x\).
    If the rank varies, this description applies on each union of components where the rank has a fixed value; \(\GL(E)\) has no morphisms between fibres of different ranks.
\end{example}

\begin{definition}
    \label{def:dagger-lie-groupoid}
    A \emph{dagger Lie groupoid} is a Lie groupoid with a smooth dagger \(\dagger \colon G^\op \to G\) over the identity of \(M\).
\end{definition}

Here \(G^\op\) is the opposite groupoid, with the same morphisms and the source and target exchanged.

Explicitly, a dagger Lie groupoid has a smooth map \(g \mapsto g^\dagger\) on \(G\), sending \(g \colon x \to y\) to \(g^\dagger \colon y \to x\), with \((h \circ g)^\dagger = g^\dagger \circ h^\dagger\), \(\id_x^\dagger = \id_x\) and \(g^{\dagger\dagger} = g\).
Over a point, this is a dagger Lie group (\cref{def:dagger-lie-group}).
As for groups, the \emph{unitary} morphisms are those with \(g^\dagger = g^{-1}\), and they form a subgroupoid.

\begin{example}
    \label{ex:gl-groupoid-dagger}
    For a metric bundle \(E\), the groupoid \(\GL(E)\) is a dagger Lie groupoid with the fibrewise adjoint: for \(g \colon E_x \to E_y\), the adjoint \(g^\dagger \colon E_y \to E_x\) is defined by \(\langle g(v), w \rangle_y = \langle v, g^\dagger(w) \rangle_x\).
    Its unitary morphisms are the linear isometries between fibres.
    Every Lie groupoid also has the unitary dagger \(g^\dagger = g^{-1}\).
\end{example}

In local frames \(e_i\) around \(x\) and \(y\), the matrix of \(g^\dagger\) is obtained from the transpose of the matrix of \(g\) by multiplying with the matrices \(\langle e_i, e_j \rangle\) at \(y\) and their inverses at \(x\), so the adjoint is smooth.
Over a point, this is the dagger on \(\GL(V)\) from \cref{ex:dagger-lie-groups}.
The unitary morphisms of \(\GL(E)\) form the gauge groupoid of the orthonormal frame bundle of \(E\) (an \(\O(p, q)\)-bundle on each component where the metric has signature \((p, q)\)), just as the unitary subgroup of \(\GL(V)\) is the isometry group of the form of \(V\).

\begin{definition}
    \label{def:groupoid-representation}
    A \emph{representation} of a groupoid \(G \rightrightarrows M\) on a bundle \(E\) is a functor \(\pi \colon G \to \GL(E)\) over the identity of \(M\).
    It is \emph{unitary} if it is a dagger functor for the unitary dagger on \(G\) and the fibrewise adjoint on \(\GL(E)\), that is, \(\pi(g^{-1}) = \pi(g)^\dagger\).
\end{definition}

For Lie groupoids we ask the functor to be smooth.
For a Lie group over a point this recovers \cref{def:unitary-representation}.
Since \(\pi(g^{-1}) = \pi(g)^{-1}\), a representation is unitary if and only if every \(\pi(g) \colon E_x \to E_y\) is an isometry.
Every bundle is a representation of the trivial groupoid, with \(\pi(\id_x) = \id_{E_x}\).

\begin{example}
    \label{ex:pair-groupoid-representation}
    A representation of the pair groupoid gives linear isomorphisms \(\pi(x, y) \colon E_x \to E_y\) with \(\pi(y, z) \circ \pi(x, y) = \pi(x, z)\).
    For a fixed point \(x\), the map \(M \times E_x \to E\), \((y, v) \mapsto \pi(x, y)(v)\), is a trivialisation of \(E\), and the representation is unitary if and only if this trivialisation is an isometry for the constant metric on \(M \times E_x\).
\end{example}

\begin{example}
    \label{ex:action-groupoid-representation}
    A representation of the action groupoid \(H \ltimes M\) of \cref{ex:action-groupoid} is an \(H\)-equivariant bundle, with \(g\) acting by \(\pi(g, x) \colon E_x \to E_{g \cdot x}\), and it is unitary if and only if \(H\) acts by isometries between fibres.
\end{example}

Representations of the fundamental groupoid come from flat connections, as recalled below.

A \emph{connection} on \(E\) is an \(\R\)-bilinear map \((X, u) \mapsto \nabla_X u\) from vector fields and sections to sections, \(C^\infty(M)\)-linear in \(X\), with
\[
    \nabla_X(f u) = f\, \nabla_X u + X(f)\, u .
\]
It is \emph{metric} if \(X(\langle u, v \rangle) = \langle \nabla_X u, v \rangle + \langle u, \nabla_X v \rangle\) for all vector fields \(X\) and sections \(u, v\).
A section of \(E\) along a path \(\gamma \colon [0, 1] \to M\) is a map \(u\) with \(u(t) \in E_{\gamma(t)}\), and the connection gives its covariant derivative \(\nabla_{\gamma'} u\) along \(\gamma\), where \(\gamma'\) is the velocity of \(\gamma\).
This is the unique \(\R\)-linear operator with \(\nabla_{\gamma'}(h u) = h' u + h \nabla_{\gamma'} u\) for functions \(h\) on \([0, 1]\) and \(\nabla_{\gamma'}(v \circ \gamma)(t) = \nabla_{\gamma'(t)} v\) for sections \(v\) of \(E\), where \(\nabla_{X_y} v\), for a tangent vector \(X_y \in T_y M\), is \((\nabla_X v)(y)\) for any vector field \(X\) with value \(X_y\) at \(y\).
It is \emph{parallel} if \(\nabla_{\gamma'} u = 0\).
Along a piecewise smooth path \(\gamma\) from \(x\) to \(y\), every vector \(v \in E_x\) extends uniquely to a parallel section \(u\) along \(\gamma\) with \(u(0) = v\), and the \emph{parallel transport} \(P_\gamma \colon E_x \to E_y\) sends \(v\) to \(u(1)\).
It is a linear isomorphism, with \(P_{\gamma^{-1}} = P_\gamma^{-1}\) for the reversed path.
For background on connections and parallel transport, see \cite[Chapter~4]{lee2018}.
Two paths from \(x\) to \(y\) are \emph{thin homotopic} if there is a homotopy between them, fixing the endpoints, whose differential has rank at most one everywhere, so that it sweeps out no area.
Parallel transport depends only on the thin homotopy class of a path \cite[Definition~2.2 and Propositions~4.3 and~5.2]{schreiber-waldorf2009}.
Schreiber and Waldorf work with smooth paths with sitting instants, to which every piecewise smooth path can be reparametrised by a thin homotopy.

\begin{definition}
    \label{def:path-groupoid}
    The \emph{path groupoid} \(P(M)\) has the points of \(M\) as objects and thin homotopy classes of piecewise smooth paths as morphisms, with the unitary dagger.
    A connection \(\nabla\) on \(E\) gives the \emph{parallel transport} functor \(P(M) \to \GL(E)\), \(\gamma \mapsto P_\gamma\).
\end{definition}

Composition in \(P(M)\) is concatenation of paths, and the inverse of \(\gamma\) is the reversed path \(\gamma^{-1}\), which is also its dagger.
That parallel transport is a functor says that for paths \(\gamma \colon x \to y\) and \(\delta \colon y \to z\) the triangle
\[
    \begin{tikzcd}
        E_x \rar["P_\gamma"] \drar["P_{\delta \circ \gamma}"'] & E_y \dar["P_\delta"] \\
        & E_z
    \end{tikzcd}
\]
commutes, where \(\delta \circ \gamma\) is \(\gamma\) followed by \(\delta\).
Representations of the path groupoid are functors as in \cref{def:groupoid-representation}, with no smoothness condition.
A connection is determined by its parallel transport, and Schreiber and Waldorf characterise the functors that arise in this way from connections on principal bundles as the transport functors with \(BG\)-structure \cite[Theorem~5.4]{schreiber-waldorf2009}.
The next theorem characterises the metric ones among all connections, including those with curvature.

\begin{theorem}
    \label{thm:metric-iff-dagger-transport}
    A connection on a metric bundle is metric if and only if its parallel transport \(P(M) \to \GL(E)\) is a dagger functor.
\end{theorem}

Since the dagger of \(P(M)\) is \(\gamma \mapsto \gamma^{-1}\), the theorem says that the square
\[
    \begin{tikzcd}
        P(M)^\op \rar["P^\op"] \dar["\gamma \mapsto \gamma^{-1}"'] & \GL(E)^\op \dar["\dagger"] \\
        P(M) \rar["P"] & \GL(E)
    \end{tikzcd}
\]
commutes, where \(P\) denotes parallel transport, if and only if the connection is metric.

\begin{proof}
    Since \(P_{\gamma^{-1}} = P_\gamma^{-1}\), parallel transport is a dagger functor if and only if every \(P_\gamma\) is unitary, that is, an isometry \(E_x \to E_y\).

    Suppose first that the connection is metric, and let \(u\) and \(w\) be parallel sections along a path \(\gamma\).
    On each smooth piece of \(\gamma\), write \(u(t) = \sum_i u^i(t)\, e_i(\gamma(t))\) and \(w(t) = \sum_j w^j(t)\, e_j(\gamma(t))\) in a local frame \(e_i\) of \(E\).
    The two properties of \(\nabla_{\gamma'}\), together with the metric condition for the sections \(e_i\) and \(e_j\) at the point \(\gamma(t)\), give
    \[
        \frac{d}{dt} \langle u(t), w(t) \rangle = \langle \nabla_{\gamma'} u, w \rangle + \langle u, \nabla_{\gamma'} w \rangle = 0,
    \]
    so \(\langle u(t), w(t) \rangle\) is constant, and \(\langle P_\gamma(u(0)), P_\gamma(w(0)) \rangle = \langle u(0), w(0) \rangle\).
    Hence \(P_\gamma\) is an isometry.

    Conversely, suppose that every \(P_\gamma\) is unitary.
    Let \(X_x \in T_x M\), let \(\gamma\) be a smooth path with \(\gamma(0) = x\) and \(\gamma'(0) = X_x\), and write \(P_t \colon E_x \to E_{\gamma(t)}\) for the parallel transport along \(\gamma\) restricted to \([0, t]\).
    We first recover \(\nabla_{X_x}\) from parallel transport.
    For a basis \(e_1, \dots, e_n\) of \(E_x\), the sections \(P_t(e_i)\) along \(\gamma\) are parallel and form a basis of every fibre \(E_{\gamma(t)}\).
    Writing a section \(u\) of \(E\) as \(u(\gamma(t)) = \sum_i u^i(t)\, P_t(e_i)\), the second property of \(\nabla_{\gamma'}\) gives \(\nabla_{X_x} u = \sum_i (u^i)'(0)\, e_i\) and \(P_t^{-1}(u(\gamma(t))) = \sum_i u^i(t)\, e_i\), hence
    \[
        \nabla_{X_x} u = \frac{d}{dt}\Big|_{t = 0} P_t^{-1}\bigl(u(\gamma(t))\bigr).
    \]
    Since \(P_t\) is unitary, \(\langle u, w \rangle(\gamma(t)) = \langle P_t^{-1}(u(\gamma(t))), P_t^{-1}(w(\gamma(t))) \rangle_x\) for sections \(u\) and \(w\), where the right-hand side is a form on the fixed space \(E_x\).
    Differentiating at \(t = 0\) gives
    \[
        X_x(\langle u, w \rangle) = \langle \nabla_{X_x} u, w \rangle(x) + \langle u, \nabla_{X_x} w \rangle(x),
    \]
    and since \(x\) and \(X_x\) are arbitrary, the connection is metric.
\end{proof}

For Hermitian bundles, Schreiber and Waldorf show more: bundles with unitary connection are equivalent to transport functors into Hermitian spaces and isometries \cite[Theorem~5.8 and Corollary~5.9]{schreiber-waldorf2009}.
\Cref{thm:metric-iff-dagger-transport} is the groupoid version of \cref{prop:unitary-rep-differentiates}, with the path groupoid in place of the Lie group: preserving the metric along every path is the integrated form of the metric condition.

\begin{corollary}
    \label{cor:flat-metric-unitary}
    A flat connection on a metric bundle is metric if and only if its holonomy representation of the fundamental groupoid \(\Pi_1(M)\) is unitary.
\end{corollary}

\begin{proof}
    For a flat connection, parallel transport depends only on the homotopy class of a path, so it factors through the functor \(P(M) \to \Pi_1(M)\) that sends a thin homotopy class to its homotopy class, and the holonomy representation is the resulting functor \(\Pi_1(M) \to \GL(E)\).
    Every homotopy class contains a piecewise smooth path, so this functor is surjective on morphisms, and it preserves inverses.
    Hence the holonomy representation is unitary if and only if parallel transport is a dagger functor, which by \cref{thm:metric-iff-dagger-transport} holds if and only if the connection is metric.
\end{proof}

\begin{proposition}
    \label{prop:unitary-groupoid-reps}
    Unitary representations of a groupoid on metric bundles, with intertwiners as morphisms, form a dagger category with the fibrewise adjoint.
\end{proposition}

\begin{proof}
    An intertwiner from \(\pi\) on \(E\) to \(\pi'\) on \(F\) is a bundle map \(\varphi \colon E \to F\), with components \(\varphi_x \colon E_x \to F_x\), such that the square on the left below commutes for all \(g \colon x \to y\).
    \[
        \begin{tikzcd}
            E_x \rar["\pi(g)"] \dar["\varphi_x"'] & E_y \dar["\varphi_y"] \\
            F_x \rar["\pi'(g)"] & F_y
        \end{tikzcd}
        \qquad
        \begin{tikzcd}
            E_y \rar["\pi(g^{-1})"] & E_x \\
            F_y \rar["\pi'(g^{-1})"] \uar["\varphi_y^\dagger"] & F_x \uar["\varphi_x^\dagger"']
        \end{tikzcd}
    \]
    Identities and composites of intertwiners are intertwiners, so representations form a category.
    Taking adjoints of the left square and using \(\pi(g)^\dagger = \pi(g^{-1})\) and \(\pi'(g)^\dagger = \pi'(g^{-1})\) shows that the right square commutes, which is the intertwining condition for \(\varphi^\dagger\) and the morphism \(g^{-1} \colon y \to x\).
    Every morphism of \(G\) is of this form, so \(\varphi^\dagger\) is an intertwiner from \(\pi'\) to \(\pi\).
    The dagger axioms hold because they hold for bundle maps (\cref{prop:herm-module-dagger}).
\end{proof}

\section{Lie algebroids and metric connections}
\label{sec:lie-algebroids}

A Lie algebroid is a Lie algebra with many objects, in the same way as a Lie groupoid is a Lie group with many objects, and differentiating a Lie groupoid gives a Lie algebroid.
As in \cref{sec:lie-groups}, we differentiate the dagger, now with the Lie algebra \(\gl(V)\) replaced by the general linear algebroid \(\gl(E)\), whose sections are the derivations of a bundle \(E\).
A metric on \(E\) gives \(\gl(E)\) a dagger, and \cref{thm:differentiate-groupoid-dagger} shows that it is the derivative of the fibrewise adjoint on \(\GL(E)\).
Its formula \eqref{eq:gl-dagger} has, besides the adjoint term, a term coming from the derivative of the metric.
A flat connection is a representation of the tangent algebroid, and it is unitary if and only if it is metric, so metric connections are the infinitesimal form of the metric parallel transport of \cref{sec:lie-groupoids}.
We state the definitions for Lie--Rinehart algebras, the algebraic form of Lie algebroids, since \cref{part:diff} works in that generality.

Throughout, bundles are real, with bundle metrics and positive rank on every component, and \(A = C^\infty(M)\).

\subsection{Lie algebroids and Lie--Rinehart algebras}
\label{sec:lie-rinehart}

\begin{definition}
    \label{def:lie-rinehart}
    Let \(K\) be a commutative ring and \(A\) a commutative \(K\)-algebra.
    A \emph{Lie--Rinehart algebra} \((A, L)\) is a \(K\)-Lie algebra \(L\) that is also an \(A\)-module, with an \(A\)-linear Lie algebra map \(\rho \colon L \to \Der_K(A)\), the \emph{anchor}, such that
    \[
        [X, f Y] = f [X, Y] + X(f)\, Y,
    \]
    where \(X(f) = \rho(X)(f)\).
    A \emph{Lie algebroid} over \(M\) is a vector bundle \(L \to M\) whose sections form a Lie--Rinehart algebra \((C^\infty(M), \Gamma(L))\) over \(\R\).
\end{definition}

Here \(\Der_K(A)\) is the Lie algebra of \(K\)-linear derivations of \(A\), with the commutator bracket.
Lie--Rinehart algebras go back to Rinehart \cite{rinehart1963}, and the name is due to Huebschmann \cite{huebschmann1990}.
Classically, the anchor of a Lie algebroid is a bundle map \(L \to TM\), and \(X(f)\) is the derivative of \(f\) along the vector field that \(X\) is sent to.
A morphism of Lie algebroids over \(M\) is a bundle map that preserves brackets and anchors.

\begin{example}
    \label{ex:lie-algebroids}
    The tangent bundle \(TM\) is a Lie algebroid, with the bracket of vector fields and the identity as anchor, so that \(X(f)\) is the derivative of \(f\) along \(X\).
    Its Lie--Rinehart algebra \((C^\infty(M), \Gamma(TM))\) is the main example of \cref{part:diff}.
\end{example}

\begin{example}
    \label{ex:lie-algebra-algebroid}
    A Lie algebra \(\mathfrak g\) is a Lie algebroid over a point, and \((K, \mathfrak g)\) with the zero anchor is a Lie--Rinehart algebra.
\end{example}

\begin{example}
    \label{ex:endomorphism-algebroid}
    For a bundle \(E\), the endomorphism bundle \(\End(E)\) is a Lie algebroid with the commutator bracket and the zero anchor; the Leibniz rule says that the commutator is \(C^\infty(M)\)-bilinear.
\end{example}

The main example of this section is the general linear algebroid.

\begin{definition}
    \label{def:gl-algebroid}
    Let \(A\) be a commutative \(K\)-algebra and \(E\) an \(A\)-module.
    A \emph{derivation} of \(E\) with \emph{symbol} \(X \in \Der_K(A)\) is an additive map \(P \colon E \to E\) with \(P(f u) = f P(u) + X(f)\, u\).
    For a bundle \(E\), with \(K = \R\) and \(A = C^\infty(M)\), the derivations of \(\Gamma(E)\) are the sections of the \emph{general linear algebroid} \(\gl(E)\), with the commutator bracket and the symbol as anchor.
\end{definition}

A derivation is \(K\)-linear, since \(X(c) = c X(1) = 0\) for \(c \in K\) gives \(P(c u) = c P(u)\).
For a bundle, the symbol is determined by \(P\), since \(X(f)\, u = P(f u) - f P(u)\) and \(E\) has positive rank.
The commutator of derivations with symbols \(X\) and \(Y\) is a derivation with symbol \([X, Y]\), and derivations with symbol zero are the \(C^\infty(M)\)-linear maps, that is, the sections of \(\End(E)\).
The derivations fit into the Atiyah sequence \eqref{eq:atiyah} below, an extension of \(\Gamma(TM)\) by \(\Gamma(\End(E))\).
It splits as a sequence of modules because \(\Gamma(TM)\) is projective.
Consequently the derivations form a finitely generated projective module and, by \cref{thm:serre-swan}, are the sections of a bundle.
For a connection \(\nabla\), every \(\nabla_X\) is a derivation with symbol \(X\).
Over a point, \(\gl(E)\) is the Lie algebra \(\gl(V)\) of a vector space.

Every Lie groupoid \(G \rightrightarrows M\) has a Lie algebroid \(\Lie(G)\), just as a Lie group has a Lie algebra \cite{crainic-fernandes2011}.
Its fibre at \(x\) is the tangent space at \(\id_x\) of the source fibre \(s^{-1}(x)\), its anchor is the differential of the target map, and its bracket comes from the right-invariant vector fields on \(G\).
A morphism \(G \to G'\) of Lie groupoids over \(M\) gives a morphism \(\Lie(G) \to \Lie(G')\) of Lie algebroids, the restriction of its differential, and this makes \(\Lie\) a functor.
Over a point, \(\Lie(G)\) is the Lie algebra of the Lie group \(G\).
The pair groupoid and the fundamental groupoid both have the tangent algebroid \(TM\) as their Lie algebroid, since the source fibre of the pair groupoid over \(x\) is \(\{x\} \times M\), and \(\Pi_1(M)\) agrees with the pair groupoid near the units.
For the general linear groupoid, \(\Lie(\GL(E)) = \gl(E)\) \cite{mackenzie2005}, \cite[Lemma~2.33]{crainic-fernandes2011}.
The flow description will let us compare the dagger on \(\GL(E)\) with the one on \(\gl(E)\).
A derivation \(P\) with symbol \(X\) generates a local flow of linear isomorphisms \(\Phi_t \colon E_x \to E_{x_t}\), where \(t \mapsto x_t\) is the integral curve of \(X\) through \(x\) \cite[Theorem~3.1 and Corollary~3.3]{jotz2024}.
The derivation is recovered from this flow by
\begin{equation}
    \label{eq:derivation-flow}
    P(u)(x) = \frac{d}{dt}\Big|_{t = 0} \Phi_t^{-1}\bigl(u(x_t)\bigr).
\end{equation}
The element of \(\Lie(\GL(E))\) at \(x\) that corresponds to \(P\) is the velocity at \(t = 0\) of the curve \(t \mapsto \Phi_t\) in the source fibre of \(\GL(E)\) over \(x\).
For a connection, the flow of \(\nabla_X\) is parallel transport along the integral curves of \(X\), and \eqref{eq:derivation-flow} is the formula for \(\nabla\) from the proof of \cref{thm:metric-iff-dagger-transport}.

\subsection{Dagger Lie algebroids}
\label{sec:dagger-lie-algebroids}

\begin{definition}
    \label{def:dagger-lie-algebroid}
    Let \((A, L)\) be a Lie--Rinehart algebra and \(f \mapsto \conj f\) an involutive automorphism of \(A\).
    A \emph{dagger} on \((A, L)\) is an additive map \(X \mapsto X^\dagger\) on \(L\) such that
    \begin{gather*}
        (f X)^\dagger = \conj f\, X^\dagger, \qquad X^{\dagger\dagger} = X, \\
        [X, Y]^\dagger = [Y^\dagger, X^\dagger], \qquad X^\dagger(f) = -\conj{X(\conj f)}.
    \end{gather*}
    A dagger on a Lie algebroid \(L\) is a dagger on \((C^\infty(M), \Gamma(L))\) with the trivial involution, that is, a \(C^\infty(M)\)-linear involution of \(\Gamma(L)\) that reverses brackets and satisfies \(X^\dagger(f) = -X(f)\).
    For the trivial involution, the \emph{unitary dagger} is \(X^\dagger = -X\).
\end{definition}

Over a point, with \(A = \R\) and the zero anchor, this is a dagger Lie algebra (\cref{def:dagger-lie-algebra}).
As in \cref{rem:cartan-involution}, we can trade the dagger for \(\theta(X) = -X^\dagger\).
The conditions on the dagger say that \(\theta\) is additive, involutive and bracket-preserving, with \(\theta(f X) = \conj f\, \theta(X)\) and \((\theta(X))(f) = \conj{X(\conj f)}\), so a dagger on \((A, L)\) is the same thing as a real structure in the sense of \cref{def:real-structure}.
For a Lie algebroid, \(\theta\) is an involutive automorphism of \(L\) that preserves the anchor, and the unitary dagger corresponds to \(\theta = \id\).

\begin{proposition}
    \label{prop:gl-dagger}
    Let \(E\) be a finitely generated projective module over a commutative \(K\)-algebra \(A\), with a unimodular symmetric form.
    Every derivation \(P\) of \(E\) with symbol \(X\) has a unique derivation \(P^\dagger\), with symbol \(-X\), such that
    \begin{equation}
        \label{eq:gl-dagger}
        \langle P(u), v \rangle = \langle u, P^\dagger(v) \rangle + X(\langle u, v \rangle)
    \end{equation}
    for all \(u, v \in E\).
    The map \(P \mapsto P^\dagger\) is additive and involutive, satisfies \((f P)^\dagger = f P^\dagger\) and \([P, Q]^\dagger = [Q^\dagger, P^\dagger]\), and for a metric bundle it is a dagger on \(\gl(E)\).
\end{proposition}

\begin{proof}
    Fix \(v \in E\).
    The map \(u \mapsto \langle P(u), v \rangle - X(\langle u, v \rangle)\) is \(A\)-linear, because the two terms \(X(f) \langle u, v \rangle\) coming from the Leibniz rules cancel:
    \[
        \langle P(f u), v \rangle - X(\langle f u, v \rangle)
        = f \langle P(u), v \rangle + X(f) \langle u, v \rangle - X(f) \langle u, v \rangle - f X(\langle u, v \rangle).
    \]
    Since the form is unimodular, this linear form is \(\langle -, w \rangle\) for a unique \(w \in E\), and we put \(P^\dagger(v) = w\).
    This gives the unique map satisfying \eqref{eq:gl-dagger}, and it is additive.
    It is a derivation with symbol \(-X\), since
    \begin{align*}
        \langle P(u), f v \rangle - X(\langle u, f v \rangle)
        &= f \bigl(\langle P(u), v \rangle - X(\langle u, v \rangle)\bigr) - X(f) \langle u, v \rangle \\
        &= \langle u, f P^\dagger(v) - X(f)\, v \rangle .
    \end{align*}

    Next we check the remaining properties, each time using the uniqueness in \eqref{eq:gl-dagger}; for a metric bundle they are the conditions of \cref{def:dagger-lie-algebroid}.
    The map \(P \mapsto P^\dagger\) is additive, and \((f P)^\dagger = f P^\dagger\) because \(f P\) has symbol \(f X\).
    For \(P^{\dagger\dagger} = P\), note that \(P^\dagger\) has symbol \(-X\), and by \eqref{eq:gl-dagger} for \(P\) with \(u\) and \(v\) exchanged,
    \[
        \langle P^\dagger(u), v \rangle - (-X)(\langle u, v \rangle)
        = \langle P(v), u \rangle - X(\langle v, u \rangle) + X(\langle u, v \rangle)
        = \langle u, P(v) \rangle,
    \]
    where we used that the form is symmetric.
    For brackets, let \(Q\) be a second derivation, with symbol \(Y\).
    Applying \eqref{eq:gl-dagger} twice,
    \[
        \langle P(Q(u)), v \rangle
        = \langle u, Q^\dagger(P^\dagger(v)) \rangle + Y(\langle u, P^\dagger(v) \rangle) + X(\langle u, Q^\dagger(v) \rangle) + X(Y(\langle u, v \rangle)),
    \]
    and subtracting the same expression with \(P\) and \(Q\) exchanged gives
    \[
        \langle [P, Q](u), v \rangle = \langle u, [Q^\dagger, P^\dagger](v) \rangle + [X, Y](\langle u, v \rangle).
    \]
    Since \([P, Q]\) has symbol \([X, Y]\), this says \([P, Q]^\dagger = [Q^\dagger, P^\dagger]\).
    Finally, the symbol of \(P^\dagger\) is \(-X\), which is the condition on the anchor.
\end{proof}

On endomorphisms, which have symbol zero, the dagger of \(\gl(E)\) is the fibrewise adjoint of \cref{sec:vector-bundles}.
For a derivation with symbol \(X\), the term \(X(\langle u, v \rangle)\) is the correction to the adjoint equation of \cref{prop:herm-module-dagger}.

For a Lie group, the dagger fixes the identity, and \cref{thm:differentiate-dagger} differentiates it there.
For a Lie groupoid, the dagger sends the source fibre \(s^{-1}(x)\) to the target fibre \(t^{-1}(x)\), so we compose it with the inversion first.
If \(G\) is a dagger Lie groupoid, then \(g \mapsto (g^\dagger)^{-1}\) sends a morphism back to one with the same source and target,
\[
    \begin{tikzcd} x \rar["g"] & y \end{tikzcd}
    \quad \longmapsto \quad
    \begin{tikzcd} y \rar["g^\dagger"] & x \end{tikzcd}
    \quad \longmapsto \quad
    \begin{tikzcd} x \rar["(g^\dagger)^{-1}"] & y , \end{tikzcd}
\]
and it is a smooth functor over the identity of \(M\), because both the dagger and the inversion reverse the order of composition, and it is involutive.
Its derivative is therefore an involutive automorphism of \(\Lie(G)\) that preserves the anchor.

\begin{definition}
    \label{def:groupoid-dagger-derivative}
    The \emph{derivative} of the dagger of a dagger Lie groupoid \(G\) is the negative of the derivative of \(g \mapsto (g^\dagger)^{-1}\): if \(X \in \Lie(G)\) at \(x\) is the velocity of a curve \(g_t\) in \(s^{-1}(x)\) with \(g_0 = \id_x\), then
    \[
        X^\dagger = -\frac{d}{dt}\Big|_{t = 0} \bigl(g_t^\dagger\bigr)^{-1} .
    \]
\end{definition}

By the discussion after \cref{def:dagger-lie-algebroid}, this is a dagger on \(\Lie(G)\), and the automorphism is its \(\theta\).
For a Lie group, the inversion has derivative \(-\id\) at the identity, so this is the derivative of the dagger itself, as in \cref{thm:differentiate-dagger}.

\begin{theorem}
    \label{thm:differentiate-groupoid-dagger}
    The Lie functor sends the fibrewise adjoint on \(\GL(E)\) to the dagger of \cref{prop:gl-dagger} on \(\gl(E)\).
\end{theorem}

\begin{proof}
    Let \(P\) be a derivation with symbol \(X\), and let \(\Phi_t \colon E_x \to E_{x_t}\) be its flow, as in \eqref{eq:derivation-flow}.
    Applying the automorphism \(g \mapsto (g^\dagger)^{-1}\) to the flow gives the linear isomorphisms \((\Phi_t^\dagger)^{-1} \colon E_x \to E_{x_t}\).
    This is a curve in the source fibre over \(x\) through \(\id_x\), and its velocity at \(t = 0\) is the image of \(P\) under the derivative of the automorphism, that is, \(-P^\ddagger\), where \(P^\ddagger\) denotes the derivative of the fibrewise adjoint (\cref{def:groupoid-dagger-derivative}).
    Since the right-hand side of \eqref{eq:derivation-flow} at \(x\) depends only on the velocity of the curve \(t \mapsto \Phi_t\) at \(t = 0\), we get
    \[
        -P^\ddagger(v)(x) = \frac{d}{dt}\Big|_{t = 0} \Phi_t^\dagger\bigl(v(x_t)\bigr).
    \]
    For sections \(u\) and \(v\), the definition of the fibrewise adjoint gives
    \[
        \bigl\langle \Phi_t^{-1}\bigl(u(x_t)\bigr), \Phi_t^\dagger\bigl(v(x_t)\bigr) \bigr\rangle_x
        = \langle u, v \rangle(x_t).
    \]
    Both arguments on the left are curves in the fixed space \(E_x\), through \(u(x)\) and \(v(x)\) at \(t = 0\), and differentiating at \(t = 0\) gives
    \[
        \langle P(u), v \rangle(x) - \langle u, P^\ddagger(v) \rangle(x) = X(\langle u, v \rangle)(x).
    \]
    This is \eqref{eq:gl-dagger} with \(P^\ddagger\) in place of \(P^\dagger\), so \(P^\ddagger = P^\dagger\) by the uniqueness in \cref{prop:gl-dagger}.
\end{proof}

\subsection{Representations and metric connections}
\label{sec:algebroid-representations}

A bundle map \(\pi\) between Lie algebroids with daggers \emph{commutes with the daggers} if \(\pi(X^\dagger) = \pi(X)^\dagger\), and a \emph{dagger morphism} is a Lie algebroid morphism that commutes with the daggers.
For a Lie--Rinehart algebra \((A, L)\) and an \(A\)-module \(E\), an \emph{\(L\)-connection} on \(E\) is an \(A\)-linear map \(X \mapsto \nabla_X\) from \(L\) to the derivations of \(E\), such that the symbol of \(\nabla_X\) is the anchor of \(X\); it is \emph{flat} if \(\nabla_{[X, Y]} = [\nabla_X, \nabla_Y]\).
For a form on \(E\) and the trivial involution, it is \emph{metric} if \(X(\langle u, v \rangle) = \langle \nabla_X u, v \rangle + \langle u, \nabla_X v \rangle\).
For a Lie algebroid \(L\) and a bundle \(E\) we take \(A = C^\infty(M)\) and the sections of \(L\) and \(E\), and for \(L = TM\) these are the connections of \cref{sec:lie-groupoids}.

\begin{definition}
    \label{def:algebroid-representation}
    A \emph{representation} of a Lie algebroid \(L\) on \(E\) is a Lie algebroid morphism \(\pi \colon L \to \gl(E)\), written \(\pi(X) = \nabla_X\); it is a flat \(L\)-connection.
    It is \emph{unitary} if it is a dagger morphism for the unitary dagger on \(L\), that is, \(\nabla_X^\dagger = -\nabla_X\).
\end{definition}

Over a point, it is a representation of a Lie algebra, and unitarity is that of \cref{def:unitary-representation}.
A representation of the tangent algebroid is a flat connection.
For every connection, \(X \mapsto \nabla_X\) is a bundle map \(TM \to \gl(E)\) that preserves anchors, and it preserves brackets if and only if the connection is flat.

\begin{lemma}
    \label{lem:metric-connection-skew}
    A connection \(\nabla\) on a metric bundle is metric if and only if \(\nabla_X^\dagger = -\nabla_X\) in \(\gl(E)\) for every vector field \(X\), if and only if its connection forms in local orthonormal frames take values in skew-adjoint matrices.
\end{lemma}

\begin{proof}
    Since \(\nabla_X\) has symbol \(X\) and \(-\nabla_X\) has symbol \(-X\), the uniqueness in \cref{prop:gl-dagger} says that \(\nabla_X^\dagger = -\nabla_X\) if and only if
    \[
        \langle \nabla_X u, v \rangle = -\langle u, \nabla_X v \rangle + X(\langle u, v \rangle)
    \]
    for all sections \(u\) and \(v\), which is the metric condition.

    For the connection forms, let \(e_1, \dots, e_n\) be a local orthonormal frame on an open set \(U\), with \(\langle e_i, e_j \rangle = \pm \delta_{ij}\); such frames exist by the Gram--Schmidt process, choosing at each step a local section \(e\) with \(\langle e, e \rangle\) nowhere zero, which is possible because the restricted form is unimodular.
    The connection forms \(\omega^i_j\) are the \(1\)-forms on \(U\) with \(\nabla_X e_j = \sum_i \omega^i_j(X)\, e_i\).
    The expression \(X(\langle u, v \rangle) - \langle \nabla_X u, v \rangle - \langle u, \nabla_X v \rangle\) is \(C^\infty(M)\)-linear in \(u\) and in \(v\), by the Leibniz rules, so the connection is metric if and only if it vanishes on the frames, that is, if and only if
    \begin{align*}
        0 &= X(\langle e_j, e_k \rangle) - \langle \nabla_X e_j, e_k \rangle - \langle e_j, \nabla_X e_k \rangle \\
        &= -\sum_i \omega^i_j(X) \langle e_i, e_k \rangle - \sum_i \omega^i_k(X) \langle e_j, e_i \rangle
    \end{align*}
    for all \(j\) and \(k\), where \(X(\langle e_j, e_k \rangle) = 0\) because the frame is orthonormal.
    This says that the matrix \(\omega(X)\) with entries \(\omega^i_j(X)\) is skew-adjoint for the form on \(\R^n\) whose matrix is \(\langle e_i, e_j \rangle\), that is, it lies in the unitary subalgebra of \(\gl(n, \R)\) for this form (\cref{ex:gl-adjoint}).
\end{proof}

The lemma applies to all connections, including those with curvature.
For a Riemannian metric, the connection forms of a metric connection in orthonormal frames are antisymmetric matrices, with values in \(\o(n)\); in signature \((p, q)\) they take values in the Lie algebra \(\o(p, q)\) of the group \(\O(p, q)\) of isometries of \(\R^n\) with the form of signature \((p, q)\).

\begin{theorem}
    \label{thm:unitary-algebroid-rep}
    A representation of the tangent algebroid \(TM\) on a metric bundle, that is, a flat connection, is unitary if and only if it is metric.
\end{theorem}

\begin{proof}
    By \cref{def:algebroid-representation}, the representation is unitary if and only if \(\nabla_X^\dagger = -\nabla_X\) for every vector field \(X\), and by \cref{lem:metric-connection-skew} this holds if and only if the connection is metric.
\end{proof}

Together with \cref{cor:flat-metric-unitary}, this is the analogue for groupoids of \cref{prop:unitary-rep-differentiates}: the holonomy representation of \(\Pi_1(M)\) is unitary if and only if its derivative, the flat connection, is unitary.
In \cref{prop:unitary-rep-differentiates} the group had to be connected for the converse; here \(M\) may have several components, since every morphism of \(\Pi_1(M)\) is represented by a path, and the metric condition integrates along paths.
For connections with curvature, \cref{thm:metric-iff-dagger-transport} and \cref{lem:metric-connection-skew} give the same correspondence, with the path groupoid in place of \(\Pi_1(M)\).

The tangent algebroid and \(\gl(E)\) are related by the \emph{Atiyah sequence}
\begin{equation}
    \label{eq:atiyah}
    \begin{tikzcd}
        0 \rar & \End(E) \rar & \gl(E) \rar & TM \rar & 0 ,
    \end{tikzcd}
\end{equation}
where the second map is the inclusion and the third is the symbol.
It is exact because the derivations with symbol zero are the endomorphisms, and every vector field \(X\) is the symbol of \(\nabla_X\) for any connection \(\nabla\).
With the fibrewise adjoint on \(\End(E)\), the dagger of \cref{prop:gl-dagger} on \(\gl(E)\) and the unitary dagger on \(TM\), both maps are dagger morphisms, since the dagger of \(\gl(E)\) restricts to the fibrewise adjoint on endomorphisms and \(P^\dagger\) has symbol \(-X\).

\begin{corollary}
    \label{cor:atiyah}
    Splittings of the Atiyah sequence \eqref{eq:atiyah} by vector bundle maps are connections, and by Lie algebroid maps are flat connections.
    A splitting commutes with the daggers, for the unitary dagger on \(TM\), if and only if the connection is metric; it is a dagger morphism if and only if the connection is flat and metric.
\end{corollary}

\begin{proof}
    A splitting by a vector bundle map is a \(C^\infty(M)\)-linear map \(X \mapsto \nabla_X\) from vector fields to derivations such that \(\nabla_X\) has symbol \(X\), as in the diagram
    \[
        \begin{tikzcd}
            0 \rar & \End(E) \rar & \gl(E) \rar["\text{symbol}"] & TM \rar \lar[bend left, "\nabla"] & 0 ,
        \end{tikzcd}
    \]
    and this is a connection.
    It preserves the anchor, so it is a Lie algebroid map if and only if it preserves brackets, that is, if and only if the connection is flat.
    It commutes with the daggers if and only if \(\nabla_{-X} = \nabla_X^\dagger\) for all \(X\), and by \cref{lem:metric-connection-skew} this holds if and only if the connection is metric.
    Combining the two statements, it is a dagger morphism if and only if the connection is flat and metric.
\end{proof}

\begin{proposition}
    \label{prop:unitary-algebroid-reps}
    Unitary representations of a Lie algebroid on metric bundles, with intertwiners, form a dagger category with the fibrewise adjoint.
\end{proposition}

\begin{proof}
    An intertwiner from \(\nabla\) on \(E\) to \(\nabla'\) on \(F\) is a bundle map \(\varphi \colon E \to F\) such that the square
    \[
        \begin{tikzcd}
            \Gamma(E) \rar["\nabla_X"] \dar["\varphi"'] & \Gamma(E) \dar["\varphi"] \\
            \Gamma(F) \rar["\nabla'_X"] & \Gamma(F)
        \end{tikzcd}
    \]
    commutes for all sections \(X\) of \(L\).
    Identities and composites of intertwiners are intertwiners.
    For unitary representations, \eqref{eq:gl-dagger} with \(\nabla_X^\dagger = -\nabla_X\) reads \(X(\langle u, v \rangle) = \langle \nabla_X u, v \rangle + \langle u, \nabla_X v \rangle\).
    Using this for \(\nabla'\), then the intertwining property, then this for \(\nabla\), we get for sections \(u\) of \(E\) and \(w\) of \(F\)
    \begin{align*}
        \langle u, \varphi^\dagger(\nabla'_X w) \rangle
        &= X(\langle \varphi(u), w \rangle) - \langle \nabla'_X(\varphi(u)), w \rangle \\
        &= X(\langle u, \varphi^\dagger(w) \rangle) - \langle \nabla_X u, \varphi^\dagger(w) \rangle
        = \langle u, \nabla_X(\varphi^\dagger(w)) \rangle .
    \end{align*}
    Since the metric of \(E\) is unimodular, \(\varphi^\dagger \circ \nabla'_X = \nabla_X \circ \varphi^\dagger\), so \(\varphi^\dagger\) is an intertwiner.
    The dagger axioms hold because they hold for bundle maps (\cref{prop:herm-module-dagger}).
\end{proof}

\begin{remark}
    \label{rem:towards-divergence}
    The term \(X(\langle u, v \rangle)\) in \eqref{eq:gl-dagger} is not of the form \(\langle u, Q(v) \rangle\) for an operator \(Q\), since it is not \(C^\infty(M)\)-linear in \(u\).
    Classically, it is absorbed by integrating against a density \(\mu\), a smooth positive measure on \(M\), given in local coordinates by a positive smooth function times Lebesgue measure.
    The divergence \(\dvg_\mu(X)\) of a vector field is defined by
    \[
        \int_M X(f)\, \mu = -\int_M f \dvg_\mu(X)\, \mu
    \]
    for compactly supported \(f\), and then \eqref{eq:gl-dagger} gives
    \[
        \int_M \langle P(u), v \rangle\, \mu = \int_M \bigl\langle u, P^\dagger(v) - \dvg_\mu(X)\, v \bigr\rangle\, \mu
    \]
    for compactly supported sections, so \(P^\dagger - \dvg_\mu(X)\) is the formal adjoint of \(P\).
    Doing this for all differential operators, of all orders, is the subject of \cref{part:diff}.
\end{remark}

\part{Differential operators}
\label{part:diff}

\section{Enveloping algebras and divergences}
\label{sec:enveloping-algebra}

The daggers of \cref{part:bundles} act on bundle maps, which have order zero, and on derivations, which have order one.
We now consider every order, beginning with operators on functions.
For a positive density \(\mu\) on a manifold, integration by parts gives the familiar first-order formula
\[
    f^\dagger = f, \qquad X^\dagger = -X - \dvg_\mu(X)
\]
for a real function \(f\) and vector field \(X\).
The minus sign reverses differentiation, and the divergence term accounts for the density.
Since functions and vector fields generate the algebra of differential operators, these formulas determine its formal adjoint.

The algebraic version starts from a Lie--Rinehart algebra \((A, L)\).
Its enveloping algebra \(D\) is filtered by order, and we classify the anti-involutions that preserve this filtration (\cref{thm:anti-involutions-D}).
The classification replaces the identity on real functions by a real structure on \(A\) and \(L\), while retaining a divergence as separate data.

Throughout \cref{part:diff}, \(K\) is a commutative ring, \((A, L)\) is a Lie--Rinehart algebra over \(K\) (\cref{def:lie-rinehart}), and \(L\) is finitely generated projective as an \(A\)-module.
For a Lie algebroid this holds by \cref{thm:serre-swan}.
We write \(X(f)\) for the anchor of \(X \in L\) applied to \(f \in A\).

\subsection{The enveloping algebra}
\label{sec:enveloping-algebra-def}

\begin{definition}
    \label{def:enveloping-algebra}
    The \emph{enveloping algebra} \(D = D(A, L)\) is the \(K\)-algebra generated by \(A\) and \(L\) subject to
    \begin{gather*}
        f \cdot g = f g, \qquad 1_A = 1, \qquad f \cdot X = f X, \\
        X \cdot f - f \cdot X = X(f), \qquad X \cdot Y - Y \cdot X = [X, Y].
    \end{gather*}
    Its \emph{order filtration} is \(D^0 = A\) and \(D^{k+1} = D^k + D^k \cdot L\), where \(D^k \cdot L\) is the additive span of the products.
\end{definition}

Classically, for the tangent algebroid of a manifold \(M\), \(D\) is the algebra of differential operators on \(M\) generated by functions and vector fields, and \(D^k\) consists of the operators of order at most \(k\).
The product \(X \cdot f\) is the operator \(g \mapsto X(f g)\), and it differs from the function \(X(f)\), which acts by multiplication, by the term \(f \cdot X\).

The submodule \(D^k\) is the left \(A\)-submodule of \(D\) spanned by the products of at most \(k\) elements of \(L\).
The relation \(X \cdot f = f \cdot X + X(f)\) shows by induction that \(D^k\) is closed under multiplication by \(A\) on the right as well, and then that \(D^k D^l \subseteq D^{k + l}\).
The filtration is exhaustive, because the union of the \(D^k\) contains \(A\) and \(L\) and is closed under sums and products.

\begin{example}
    \label{ex:enveloping-examples}
    Let \(A = K[x]\) and let \(L = A \partial\) be free on one generator \(\partial\), with \((f \partial)(g) = f g'\) and \([f \partial, g \partial] = (f g' - g f') \partial\).
    This is the polynomial version of the tangent algebroid of the line, and we call it the \emph{polynomial line}.
    Its enveloping algebra is the \emph{Weyl algebra}, generated by \(x\) and \(\partial\) with the single relation \(\partial x - x \partial = 1\), and \(D^k\) consists of the operators \(\sum_{m \le k} p_m \partial^m\) with \(p_m \in K[x]\).
\end{example}

\begin{example}
    \label{ex:enveloping-lie-algebra}
    A Lie algebra \(\mathfrak g\) over \(K\) is a Lie--Rinehart algebra with \(A = K\) and zero anchor.
    Its enveloping algebra is the universal enveloping algebra \(U(\mathfrak g)\), and \(D^k\) is spanned by the products of at most \(k\) elements of \(\mathfrak g\).
\end{example}

A dagger on \(D\) will act on the scalars by the involution of \(A\), which may move the elements of \(K\), so we use the universal property of \(D\) in the following form.

\begin{lemma}
    \label{lem:enveloping-universal}
    Let \(B\) be a ring, \(\varphi \colon A \to B\) a ring homomorphism and \(\psi \colon L \to B\) an additive map with
    \begin{gather*}
        \varphi(f) \psi(X) = \psi(f X), \qquad
        \psi(X) \varphi(f) - \varphi(f) \psi(X) = \varphi(X(f)), \\
        \psi(X) \psi(Y) - \psi(Y) \psi(X) = \psi([X, Y]),
    \end{gather*}
    and with \(\varphi(K)\) central in \(B\).
    Then there is a unique ring homomorphism \(D \to B\) extending \(\varphi\) and \(\psi\).
\end{lemma}

\begin{proof}
    Since \(\varphi(K)\) is central, \(B\) is a \(K\)-algebra through \(c \mapsto \varphi(c)\), and \(\varphi\) is a homomorphism of \(K\)-algebras.
    The map \(\psi\) is \(K\)-linear, because \(\psi(c X) = \varphi(c) \psi(X)\) for \(c \in K\).
    The algebra \(D\) is the quotient of the tensor algebra of the \(K\)-module \(A \oplus L\) by the relations of \cref{def:enveloping-algebra}, and the hypotheses say that the \(K\)-algebra map from the tensor algebra determined by \(\varphi\) and \(\psi\) sends these relations to zero.
    Uniqueness holds because \(A\) and \(L\) generate \(D\).
\end{proof}

For \(B = \End_K(A)\), with \(f\) acting by multiplication and \(X\) by \(g \mapsto X(g)\), the hypotheses of the lemma are the axioms of a Lie--Rinehart algebra.
This gives the action of \(D\) on \(A\), and we write \(P(f)\) for the action of \(P \in D\) on \(f \in A\).
Classically this is the action of differential operators on functions.
For a Lie algebra with zero anchor, an element of \(U(\mathfrak g)\) acts on \(A = K\) by multiplication with its constant term.
Thus the action need not be faithful: a non-zero \(X \in \mathfrak g\) acts as zero on \(K\).
The algebra \(D\) records such elements even when their actions on functions agree.

The order filtration has an \emph{associated graded} algebra \(\gr D = \bigoplus_k D^k / D^{k-1}\), with \(D^{-1} = 0\), and \(\Sym_A L\) denotes the symmetric algebra of the \(A\)-module \(L\).
Classically, \(D^k / D^{k-1}\) consists of the principal symbols of operators of order \(k\), which are symmetric \(k\)-vector fields.

\begin{theorem}[Poincaré--Birkhoff--Witt, {\cite[Theorem~3.1]{rinehart1963}}]
    \label{thm:pbw}
    The canonical map \(\Sym_A L \to \gr D\) is an isomorphism of graded \(A\)-algebras.
\end{theorem}

The canonical map sends \(X_1 \cdots X_k\) to the class of the product \(X_1 \cdot \ldots \cdot X_k\) in \(D^k / D^{k-1}\).
It is well defined and surjective by the relations of \(D\), which say that the products of elements of \(L\) commute up to terms of lower order, and Rinehart proves its injectivity using that \(L\) is projective.

\begin{corollary}
    \label{cor:pbw}
    \
    \begin{enumerate}
        \item The map \(A \oplus L \to D\), \((f, X) \mapsto f + X\), is injective with image \(D^1\).
        \item Each \(D^k\) is finitely generated projective as a left \(A\)-module.
        \item For \(k \le l\), \(D^k\) is a direct summand of \(D^l\) as a left \(A\)-module.
    \end{enumerate}
\end{corollary}

\begin{proof}
    Since \(L\) is finitely generated projective, so is each \(\Sym^k_A L \cong D^k / D^{k-1}\), and therefore each short exact sequence of left \(A\)-modules
    \[
        \begin{tikzcd}
            0 \rar & D^{k-1} \rar & D^k \rar & \Sym^k_A L \rar & 0
        \end{tikzcd}
    \]
    splits.
    This gives \(D^k \cong D^{k-1} \oplus \Sym^k_A L\), and (2) and (3) follow by induction on \(k\).
    For (1), the image of the map is \(A + A \cdot L = D^1\).
    If \(f + X = 0\), then \(X\) lies in \(D^0\), so its class in \(D^1 / D^0 \cong L\) vanishes and \(X = 0\); then \(f = 0\), since \(A \to D^0\) is the isomorphism of \cref{thm:pbw} in degree zero.
\end{proof}

\subsection{Daggers on the enveloping algebra}
\label{sec:anti-involutions-D}

A \emph{dagger} on \(D\) is an anti-involution of the ring \(D\), which makes it a dagger ring (\cref{def:dagger-ring}).
It is \emph{filtration-preserving} if \((D^k)^\dagger \subseteq D^k\) for all \(k\).
A dagger that preserves \(D^0\) and \(D^1\) preserves the whole filtration, since \(D^{k+1} = D^k + D^k \cdot L\) and, by induction on \(k\), \((D^k \cdot L)^\dagger \subseteq L^\dagger (D^k)^\dagger \subseteq D^1 D^k \subseteq D^{k+1}\).

The motivating example is the formal adjoint on \(\R^n\).
Let \(\dvg(X) = \sum_i \partial_i X^i\) be the usual divergence of a vector field \(X = \sum_i X^i \partial_i\).
For compactly supported functions, integration by parts gives
\[
    \int_{\R^n} X(f)\, g \, dx = -\int_{\R^n} f \bigl(X(g) + \dvg(X)\, g\bigr)\, dx .
\]
The formal adjoint therefore sends a function \(f\) to itself and a vector field \(X\) to \(-X - \dvg(X)\), and it reverses products.
We show that every filtration-preserving dagger on \(D\) has this shape, up to involutions of \(A\) and \(L\), and that the data which appear are the following.

\begin{definition}
    \label{def:real-structure}
    A \emph{real structure} on \((A, L)\) consists of an involutive ring automorphism \(f \mapsto \conj f\) of \(A\) and an additive, involutive, bracket-preserving map \(\theta \colon L \to L\) with
    \[
        \theta(f X) = \conj f\, \theta(X), \qquad (\theta(X))(f) = \conj{X(\conj f)}.
    \]
    A \emph{real divergence} for it is an additive map \(\dvg \colon L \to A\) with
    \begin{gather*}
        \dvg(f X) = f \dvg(X) + X(f), \qquad
        \dvg([X, Y]) = X(\dvg(Y)) - Y(\dvg(X)), \\
        \dvg(\theta(X)) = \conj{\dvg(X)}.
    \end{gather*}
\end{definition}

We call the first two identities for \(\dvg\) the \emph{Leibniz rule} and the \emph{bracket identity}, and the third the \emph{reality condition}.
The conditions on \(\theta\) say that \(X \mapsto -\theta(X)\) is a dagger on \((A, L)\) in the sense of \cref{def:dagger-lie-algebroid}, so a real structure is an involution of \(A\) together with such a dagger, as in \cref{rem:cartan-involution} for Lie algebras.
We write it as \(-\theta\) because \(X^\dagger\) will denote the dagger on \(D\).
The name comes from the main example, \(A = C^\infty(M, \C)\) with complex vector fields, where the bar and \(\theta\) are complex conjugation.
The identities of \(A\) and \(L\) always form a real structure, the \emph{trivial} one.
A real divergence is additional data, and \cref{sec:divergences} explains how it arises from a positive density.

\begin{theorem}
    \label{thm:anti-involutions-D}
    Filtration-preserving daggers on \(D\) correspond bijectively to real structures on \((A, L)\) with a real divergence, via
    \[
        f^\dagger = \conj f, \qquad X^\dagger = -\theta(X) - \conj{\dvg(X)}
        \qquad \text{for } f \in A \text{ and } X \in L.
    \]
\end{theorem}

\begin{proof}
    We first show that a filtration-preserving dagger determines a real structure and a real divergence satisfying the displayed formulas.
    Since \(\dagger\) preserves \(D^0 = A\), we may put \(\conj f = f^\dagger\).
    This map is additive and involutive and fixes \(1\), and since \(A\) is commutative, \(\conj{f g} = g^\dagger f^\dagger = \conj f\, \conj g\).
    Next, \(X^\dagger\) lies in \(D^1\), which is \(A \oplus L\) by \cref{cor:pbw}, so \(X^\dagger = g + Y\) for unique \(g \in A\) and \(Y \in L\).
    Setting \(\theta(X) = -Y\) and \(\dvg(X) = -\conj g\) gives the displayed formula, and these are the only maps that do, because the bar is bijective.
    Both maps are additive, because \(\dagger\) is additive and the decomposition of \(D^1\) is unique.

    The remaining conditions come from applying \(\dagger\) to the relations of \(D\).
    Each relation becomes an identity in \(D^1 = A \oplus L\), whose components in \(L\) and in \(A\) give one condition each.
    Throughout we use that functions commute with each other in \(D\), and that \(Z \cdot f - f \cdot Z = Z(f)\) for \(Z \in L\) and \(f \in A\).

    Applying \(\dagger\) to \(X \cdot f - f \cdot X = X(f)\) gives
    \[
        \conj{X(f)} = \conj f\, X^\dagger - X^\dagger\, \conj f = \theta(X) \cdot \conj f - \conj f \cdot \theta(X) = (\theta(X))(\conj f).
    \]
    Replacing \(f\) by \(\conj f\) gives \((\theta(X))(f) = \conj{X(\conj f)}\).

    Applying \(\dagger\) to \(f \cdot X = f X\), and using the anchor identity, gives
    \[
        -\theta(f X) - \conj{\dvg(f X)}
        = X^\dagger\, \conj f
        = -\conj f\, \theta(X) - \conj{X(f)} - \conj f\, \conj{\dvg(X)}.
    \]
    The component in \(L\) is \(\theta(f X) = \conj f\, \theta(X)\), and the component in \(A\) is the Leibniz rule.

    Applying \(\dagger\) to \(X \cdot Y - Y \cdot X = [X, Y]\) gives
    \begin{align}
        -\theta([X, Y]) - \conj{\dvg([X, Y])}
        &= Y^\dagger X^\dagger - X^\dagger Y^\dagger \notag \\
        &= -[\theta(X), \theta(Y)] \notag \\
        \label{eq:enveloping-bracket-check}
        &\qquad + (\theta(Y))\bigl(\conj{\dvg(X)}\bigr) - (\theta(X))\bigl(\conj{\dvg(Y)}\bigr).
    \end{align}
    The component in \(L\) says that \(\theta\) preserves brackets.
    By the anchor identity, the component in \(A\) is the conjugate of \(\dvg([X, Y]) = X(\dvg(Y)) - Y(\dvg(X))\), the bracket identity.

    Applying \(\dagger\) to \(X^\dagger = -\theta(X) - \conj{\dvg(X)}\) gives
    \begin{equation}
        \label{eq:enveloping-square-check}
        X = \theta(\theta(X)) + \conj{\dvg(\theta(X))} - \dvg(X).
    \end{equation}
    The component in \(L\) says that \(\theta\) is involutive, and the component in \(A\) is the reality condition.

    Conversely, let a real structure with a real divergence be given, and put \(\psi(X) = -\theta(X) - \conj{\dvg(X)}\).
    We apply \cref{lem:enveloping-universal} to the bar and \(\psi\), as maps into the opposite ring \(D^\op\), which is \(D\) with the order of multiplication reversed.
    Written in \(D\), its hypotheses read
    \begin{gather*}
        \conj f\, \psi(X) - \psi(X)\, \conj f = \conj{X(f)}, \qquad
        \psi(X)\, \conj f = \psi(f X), \\
        \psi(Y) \psi(X) - \psi(X) \psi(Y) = \psi([X, Y]).
    \end{gather*}
    These identities follow from the anchor, scalar-multiplication and bracket computations, read in reverse.
    The first is \(\theta(X) \cdot \conj f - \conj f \cdot \theta(X) = (\theta(X))(\conj f) = \conj{X(f)}\).
    For the second, the first gives \(\psi(X)\, \conj f = \conj f\, \psi(X) - \conj{X(f)} = -\conj f\, \theta(X) - \conj{f \dvg(X) + X(f)}\), which is \(\psi(f X)\) by the conditions on \(\theta(f X)\) and the Leibniz rule.
    For the third, the computation in \cref{eq:enveloping-bracket-check} gives \(-[\theta(X), \theta(Y)] + \conj{Y(\dvg(X)) - X(\dvg(Y))}\), which is \(\psi([X, Y])\) because \(\theta\) preserves brackets and by the bracket identity.
    The bar sends \(K\) into the centre of \(D\).
    For \(c \in K\), the element \(\conj c\) commutes with \(A\).
    It also commutes with every \(X \in L\), because \(X(\conj c) = \conj{(\theta(X))(c)} = 0\), as derivations over \(K\) vanish on \(K\).
    The lemma gives a ring homomorphism \(D \to D^\op\), that is, an additive map \(P \mapsto P^\dagger\) with \(1^\dagger = 1\) and \((P Q)^\dagger = Q^\dagger P^\dagger\), which extends the bar and \(\psi\).
    Its square is a ring homomorphism \(D \to D\) that fixes every \(f\).
    It also fixes every \(X\) by \cref{eq:enveloping-square-check}, since \(\theta\) is involutive and \(\dvg\) satisfies the reality condition.
    As \(A\) and \(L\) generate \(D\), the square is the identity, and \(\dagger\) is a dagger on \(D\).
    It preserves \(D^0\) and \(D^1\) by construction, and hence the whole filtration.

    The two constructions are inverse to each other.
    The dagger built from given data has \(f^\dagger = \conj f\) and \(X^\dagger = \psi(X)\), and the first part returns the same data.
    A filtration-preserving dagger agrees on the generators \(A\) and \(L\) with the one built from its own data, and the two are therefore equal.
\end{proof}

\subsection{Divergences and densities}
\label{sec:divergences}

The classical source of real divergences is a positive density.
A \emph{positive density} on a manifold \(M\) is a measure that is given in every chart by \(\rho\, \lvert dx^1 \cdots dx^n \rvert\) with \(\rho\) smooth and positive.
Left Haar measures on a Lie group are positive densities.

\begin{example}
    \label{ex:classical-divergence}
    Let \(A = C^\infty(M, \C)\) and \(L = \Gamma(TM \otimes \C)\), with complex conjugation as real structure.
    A positive density \(\mu\) gives a real divergence \(\dvg_\mu\), defined by
    \[
        \int_M X(f)\, \mu = -\int_M f \dvg_\mu(X)\, \mu \qquad \text{for compactly supported } f,
    \]
    and the dagger \(X^\dagger = -\conj X - \conj{\dvg_\mu(X)}\) is the formal adjoint for the pairing
    \[
        \int_M f \conj g\, \mu .
    \]
\end{example}

In a chart where \(\mu = \rho\, \lvert dx^1 \cdots dx^n \rvert\), the divergence is given by \(\dvg_\mu(X) = \rho^{-1} \sum_i \partial_i (\rho X^i)\).
The conditions of \cref{def:real-structure} follow by integrating by parts.
The Leibniz rule \(\dvg_\mu(f X) = f \dvg_\mu(X) + X(f)\) comes from
\[
    \int_M f X(g)\, \mu = \int_M X(f g)\, \mu - \int_M g X(f)\, \mu,
\]
the bracket identity from expanding the integral of \([X, Y](g) = X(Y(g)) - Y(X(g))\) in the same way, and \(\dvg_\mu(\conj X) = \conj{\dvg_\mu(X)}\) holds because \(\rho\) is real.
For the formal adjoint, integration by parts gives
\[
    \int_M X(f)\, \conj g\, \mu = \int_M f\, \conj{\bigl(-\conj X(g) - \conj{\dvg_\mu(X)}\, g\bigr)}\, \mu .
\]

For a unit \(w\) of \(A\) with \(\conj w = w\), we write \(d \log w\) for the map \(L \to A\), \(X \mapsto w^{-1} X(w)\).
On a manifold, for a positive function \(w\), this is the differential of the function \(\log w\), an exact \(1\)-form.

\begin{proposition}
    \label{prop:divergences-differ}
    Two real divergences for the same real structure differ by an \(A\)-linear map \(\alpha \colon L \to A\) with \(\alpha([X, Y]) = X(\alpha(Y)) - Y(\alpha(X))\) and \(\alpha(\theta(X)) = \conj{\alpha(X)}\), a real closed \(1\)-form, and adding such a form to a real divergence gives a real divergence.
    For a manifold with a positive density \(\mu\) and a positive function \(w\), \(\dvg_{w \mu} = \dvg_\mu + d \log w\).
\end{proposition}

\begin{proof}
    Let \(\alpha\) be the difference of two real divergences.
    It is additive, and the term \(X(f)\) cancels in \(\alpha(f X) = f \alpha(X)\), so \(\alpha\) is \(A\)-linear.
    The other two identities are the differences of the bracket identities and of the reality conditions.
    Conversely, if \(\alpha\) is such a form and \(\dvg\) a real divergence, then \(\dvg + \alpha\) is a real divergence, because \(\alpha\) is \(A\)-linear and satisfies the bracket and reality identities.
    For the second statement,
    \[
        \int_M X(f)\, w \mu = \int_M X(f w)\, \mu - \int_M f X(w)\, \mu = -\int_M f \bigl(\dvg_\mu(X) + w^{-1} X(w)\bigr)\, w \mu,
    \]
    and hence \(\dvg_{w \mu}(X) = \dvg_\mu(X) + (d \log w)(X)\).
\end{proof}

\begin{example}
    \label{ex:weighted-line-divergence}
    On the polynomial line with \(K = \Q\) and the trivial real structure, fix \(c \in \Q\) and put \(\dvg(f \partial) = f' + c f\).
    This is a real divergence: \(\dvg(g f \partial) = g \dvg(f \partial) + f g'\), and both \(\dvg([f \partial, g \partial])\) and \((f \partial)(\dvg(g \partial)) - (g \partial)(\dvg(f \partial))\) equal \(f g'' - g f'' + c (f g' - g f')\).
    The corresponding dagger on the Weyl algebra has \(x^\dagger = x\) and \(\partial^\dagger = -\partial - c\).
    Letting the operators act on smooth functions through the inclusion \(\Q[x] \subseteq C^\infty(\R)\), it is the formal adjoint for the pairing
    \[
        \int_\R f g\, e^{c x}\, dx
    \]
    of compactly supported smooth functions, since
    \[
        \int_\R f' g\, e^{c x}\, dx = -\int_\R f (g' + c g)\, e^{c x}\, dx .
    \]
    It differs from the divergence \(f \partial \mapsto f'\) of the Lebesgue measure by the closed form \(c\, dx\), where \(dx(f \partial) = f\).
    For \(c \neq 0\) this form is not \(d \log w\) for a unit \(w\) of \(\Q[x]\), since the units are the non-zero constants, while over \(C^\infty(\R)\) it is \(d \log e^{c x}\).
    This is the \emph{weighted line}, which returns in \cref{sec:formal-adjoints}.
\end{example}

\begin{example}
    \label{ex:circle-divergence}
    Let \(M\) be the circle \(\R / 2\pi\Z\), with \(A = C^\infty(M)\), the trivial real structure, and \(L\) the vector fields, which are the multiples \(f \partial\) of the unit vector field \(\partial\).
    The map \(\alpha(f \partial) = f\) is \(A\)-linear, and it is closed: for \(X = f \partial\) and \(Y = g \partial\) we have \([X, Y] = (f g' - g f') \partial\), so both \(\alpha([X, Y])\) and \(X(\alpha(Y)) - Y(\alpha(X))\) equal \(f g' - g f'\).
    Classically \(\alpha\) is the \(1\)-form \(dt\), where \(t\) is the angle coordinate; the angle itself is defined only locally, but \(dt\) is defined on the whole circle.
    It is not exact: for a function \(h\), the function \((d h)(\partial) = h'\) has integral zero over the circle, while \(\alpha(\partial) = 1\) has integral \(2\pi\).
    In particular \(\alpha \neq d \log w\) for every positive function \(w\).
    Let \(\mu\) be the standard positive density \(\lvert dt \rvert\) on the circle.
    By \cref{prop:divergences-differ}, \(\dvg_\mu + \alpha\) is a real divergence, but it is not the divergence of any positive density, because every positive density is \(w \mu\) for a positive function \(w\), with divergence \(\dvg_\mu + d \log w\).
\end{example}

For the tangent algebroid with its standard real structure, every real divergence is locally the divergence of a positive density, by \cref{prop:divergences-differ} and the Poincaré lemma.
Here the standard real structure is the identity on real functions and vector fields, or complex conjugation on their complexifications.
Globally, such a divergence differs from the divergence of a positive density by a real closed \(1\)-form.

\begin{example}
    \label{ex:lie-algebra-divergence}
    Let \(\mathfrak g\) be a real Lie algebra with a dagger (\cref{def:dagger-lie-algebra}).
    Regard it as a Lie--Rinehart algebra with \(A = \R\) and zero anchor.
    Take the trivial involution on \(\R\) and put \(\theta(X) = -X^\dagger\).
    A \emph{character} of \(\mathfrak g\) is a linear map \(\mathfrak g \to \R\) that vanishes on brackets.
    With zero anchor, the conditions of \cref{def:real-structure} say that a real divergence is a character with \(\dvg(\theta(X)) = \dvg(X)\), that is, \(\dvg(X^\dagger) = -\dvg(X)\).
    The dagger on \(U(\mathfrak g)\) then extends \(X \mapsto X^\dagger - \dvg(X)\).
    If \(\mathfrak g\) is semisimple, then \([\mathfrak g, \mathfrak g] = \mathfrak g\), every character vanishes, and a dagger on \(\mathfrak g\) extends to a unique filtration-preserving dagger on \(U(\mathfrak g)\).
\end{example}

\begin{example}
    \label{ex:haar-divergence}
    Let \(G\) be a Lie group with Lie algebra \(\mathfrak g\).
    We identify \(U(\mathfrak g)\) with the algebra of left-invariant differential operators on \(G\), where \(X \in \mathfrak g\) acts as the left-invariant vector field
    \[
        \tilde X(f)(y) = \frac{d}{dt} f(y \exp t X) \Big|_{t = 0}.
    \]
    Let \(\mu_L\) be a left Haar measure.
    The \emph{modular function} \(\Delta_G \colon G \to \R_{> 0}\) is defined, as in \cite[(8.26) and the proof of Proposition~8.27]{knapp2002}, by
    \[
        \int_G f(y z)\, d\mu_L(y) = \Delta_G(z) \int_G f(y)\, d\mu_L(y),
    \]
    and \(\Delta_G(z) = \lvert \det \Ad(z) \rvert\) by \cite[Proposition~8.27]{knapp2002}.
    Hence \(\Delta_G(\exp t X) = e^{t \tr(\ad X)}\), and differentiating the definition at \(z = \exp t X\) gives
    \[
        \int_G \tilde X(h)\, \mu_L = \tr(\ad X) \int_G h\, \mu_L
    \]
    for compactly supported \(h\).
    Applied to \(h = f g\), this shows that the formal adjoint of \(\tilde X\) for the pairing
    \[
        \int_G f g\, \mu_L
    \]
    is \(-\tilde X + \tr(\ad X)\).
    The formal adjoint therefore restricts to a dagger on \(U(\mathfrak g)\), with \(\theta = \id\) and
    \[
        \dvg(X) = -\tr(\ad X) = -\frac{d}{dt} \Delta_G(\exp t X) \Big|_{t = 0}.
    \]
    This divergence vanishes if and only if \(\tr(\ad X) = 0\) for all \(X\), that is, if \(\mathfrak g\) is unimodular.
    For connected \(G\), this is equivalent to \(G\) being unimodular, which means \(\Delta_G = 1\).

    The divergence depends on the choice of Haar measure.
    The measure \(\Delta_G \mu_L\) is a right Haar measure \cite[Corollary~8.30(c)]{knapp2002}, and for it the left-invariant vector fields are skew-adjoint.
    Since \(\tilde X(\Delta_G) = \tr(\ad X)\, \Delta_G\), the formula above gives
    \[
        \int_G \tilde X(h)\, \Delta_G \mu_L = \int_G \tilde X(h \Delta_G)\, \mu_L - \int_G h\, \tilde X(\Delta_G)\, \mu_L = 0 .
    \]
    This is an instance of \cref{prop:divergences-differ} on \(G\), for the density \(\mu_L\) and \(w = \Delta_G\), with \((d \log \Delta_G)(\tilde X) = \tr(\ad X)\).
    The character \(\tr \ad\) is the modular class of \(\mathfrak g\) in the sense of Huebschmann \cite[Example~6.11]{huebschmann1999}.
\end{example}

Evens, Lu and Weinstein introduced the modular class of a Lie algebroid \cite{evens-lu-weinstein1999}, and Huebschmann extended it to Lie--Rinehart algebras, together with the study of right-module structures \cite{huebschmann1999}.
These results concern the cohomological role of divergence data, while the classification in \cref{thm:anti-involutions-D} asks when that data defines an involution on the filtered enveloping algebra.

\begin{remark}
    \label{rem:huebschmann}
    Additive maps \(\dvg \colon L \to A\) satisfying the Leibniz rule and the bracket identity, but not necessarily the reality condition, are the same thing as right \(D\)-module structures on \(A\) that extend the multiplication of \(A\), via \(f \cdot X = -X(f) - f \dvg(X)\); cf.\ \cite[(2.0.3)--(2.0.4), Proposition~2.8 and \S6]{huebschmann1999}.
    Checking the right-module relations for this formula gives exactly these two identities.
    Classically, the right module is the module of densities under minus the Lie derivative, with \(A\) identified with the densities through \(f \mapsto f \mu\).
\end{remark}

\section{Jets and differential operators}
\label{sec:jets}

We now pass from operators on functions to operators between bundles.
Classically, an operator of order at most \(k\) from \(E\) to \(F\) is a bundle map from the \(k\)-jet bundle of \(E\) to \(F\): its value at a point depends on derivatives of the input section up to order \(k\).
Composition first prolongs the inner operator, computing derivatives of its output from higher derivatives of its input.
Algebraically, jets come from the filtration of \(D\), and prolongation comes from its multiplication.
These form a comonad graded by order, whose co-Kleisli category is \(\DiffOp\).
We identify the operators from \(A\) to itself with \(D\) (\cref{thm:diff-A-A}) and describe general operators through their tests (\cref{prop:tests}).
For a general Lie--Rinehart algebra, an operator is a map out of the algebraic jet module; its action on sections need not determine that map.

Classically, the \(1\)-jet of a function \(f\) on the line at a point records \(f\) and \(f'\) at that point.
To record these jets across the line, we use the map \(P \mapsto P(f)\), whose values are functions.

\begin{example}
    \label{ex:jets-line}
    Let \(A = K[x]\) and \(L = A \partial\) be the polynomial line.
    By \cref{cor:pbw}, \(D^1 = A \oplus A \partial\) is free as a left \(A\)-module on \(1\) and \(\partial\), so a map of left \(A\)-modules \(\omega \colon D^1 \to A\) is determined by the pair \((g_0, g_1) = (\omega(1), \omega(\partial))\).
    The function \(f\) gives such a map, \(P \mapsto P(f)\), with \((g_0, g_1) = (f, f')\), its value and its derivative.
    An arbitrary pair \((g_0, g_1)\) is a section of the jet bundle.
    It is the jet of a function precisely when \(g_1 = g_0'\).
    These maps carry two module structures.
    Multiplying the values of \(\omega\) by \(f\) gives \(f (g_0, g_1) = (f g_0, f g_1)\).
    Precomposing \(\omega\) with right multiplication by \(f\) in \(D^1\) gives
    \[
        (g_0, g_1) f = (f g_0, f g_1 + f' g_0),
    \]
    since \(\omega(\partial \cdot f) = \omega(f \partial + f') = f g_1 + f' g_0\).
    The second structure is the Leibniz rule: the jet \((h, h')\) of \(h\), multiplied by \(f\) in this way, is the jet \((f h, (f h)')\) of \(f h\).
    The maps from these \(\omega\) to \(A\) that are linear for the first structure are \((g_0, g_1) \mapsto p_0 g_0 + p_1 g_1\) with \(p_0, p_1 \in K[x]\), and applied to the jet of \(f\) they give \(p_0 f + p_1 f'\).
    They are the first-order operators \(p_0 + p_1 \partial\), and this is the description of operators as maps out of jets.
    Among them, those that are also linear for the second structure are the ones with \(p_1 = 0\), the operators of order zero.
\end{example}

\subsection{The jet comonad}
\label{sec:jets-general}

For the tangent algebroid, the \(k\)-jet of a function \(f\) is determined by the functions \(P(f)\) for operators \(P\) of order at most \(k\).
The map \(P \mapsto P(f)\) is linear over \(A\) acting on the left, because \((g P)(f) = g\, P(f)\), and we take such maps as the definition of jets.

\begin{definition}
    \label{def:jets}
    The module of \emph{\(k\)-jets of functions} is \(J^k A = \Hom_A(D^k, A)\), the maps of left \(A\)-modules.
    It has the \emph{pointwise} structure \((f \omega)(P) = f \omega(P)\) and the \emph{right} structure \((\omega f)(P) = \omega(P \cdot f)\).
    For a finitely generated projective module \(E\),
    \[
        J^k E = J^k A \otimes_A E,
    \]
    the tensor product taken over the right structure, with the pointwise structure as module structure.
\end{definition}

For \(E = A\), we identify \(J^k A \otimes_A A\) with \(J^k A\) by \(\omega \otimes f \mapsto \omega f\).
This identification preserves the pointwise module structure.
These actions differ in when the function is multiplied.
The pointwise action multiplies the value of a jet after evaluation; the right action multiplies its input before evaluation.
The tensor product uses the right action so that the jet of \(f u\) obeys the Leibniz rule.
Classically, \(J^k E\) is the module of sections of the \(k\)-jet bundle of \(E\) \cite{saunders1989, grothendieck1967}.
General elements of this module need not be jets of sections.
For \(A = C^\infty(\R^n)\) and the vector fields, \(D^k\) is free on the operators \(\partial^I\), where \(I\) is a multi-index with \(\lvert I \rvert \le k\).
The dual functionals form a basis of \(J^k A\) and give the usual jet coordinates.
The right structure is defined because \(D^k\) is closed under right multiplication by \(A\), and the two structures commute.

Let \(j^k(1) \in J^k A\) be the map \(P \mapsto P(1)\), the jet of the constant function \(1\).
Then \(j^k(1) f\) is the map \(P \mapsto (P \cdot f)(1) = P(f)\), the \emph{\(k\)-jet} \(j^k(f)\) of \(f\).
The \(k\)-jet of a section \(u \in E\) is \(j^k(u) = j^k(1) \otimes u\).
The tensor product over the right structure is what makes this consistent: \(j^k(f u) = j^k(1) \otimes f u = j^k(f) \otimes u\), so the jet of a section \(\sum_i f_i u_i\) is \(\sum_i j^k(f_i) \otimes u_i\), however the section is written.

In \(J^k E\) we have \(\omega f \otimes u = \omega \otimes f u\) and \(f (\omega \otimes u) = f \omega \otimes u\).
Maps out of \(J^k E\) are given below on elements \(\omega \otimes u\), by formulas that are biadditive, take the same value on \((\omega f, u)\) and \((\omega, f u)\), and are linear for the pointwise structure.
In particular, a linear map \(\varphi \colon E \to F\) gives \(J^k(\varphi) \colon J^k E \to J^k F\), \(\omega \otimes u \mapsto \omega \otimes \varphi(u)\), and \(J^k\) is an additive functor.

A jet of a section is read off through the components of the section: for a linear form \(\alpha \in E^\vee\), the function \(\alpha(u)\) has the derivatives \(P(\alpha(u))\), and the next definition extends this to all jets.

\begin{definition}
    \label{def:jet-pairing}
    A jet \(j \in J^k E\) is evaluated on \(P \in D^k\) and \(\alpha \in E^\vee\) by
    \[
        (\omega \otimes u)(P \otimes \alpha) = \omega(P \cdot \alpha(u)).
    \]
\end{definition}

Here \(P \cdot \alpha(u)\) is the product in \(D\), an element of \(D^k\); applied to \(1\) it gives the function \(P(\alpha(u))\).
The formula is well defined on \(J^k E\), since \((\omega f)(P \cdot \alpha(u)) = \omega(P \cdot \alpha(f u))\).
The value depends on \(P \otimes \alpha\) as an element of \(D^k \otimes_A E^\vee\), the tensor product over the right \(A\)-structure of \(D^k\), because \(j(P \cdot f \otimes \alpha) = j(P \otimes f \alpha)\).
It is \(A\)-linear in \(j\), and \(A\)-linear for the left structure of \(D^k\): \(j(g P \otimes \alpha) = g\, j(P \otimes \alpha)\).
On jets of sections it is \(j^k(u)(P \otimes \alpha) = P(\alpha(u))\).
Classically, evaluating the jet of a section \(u\) on \(P \otimes \alpha\) applies the operator \(P\) to the component \(\alpha(u)\) of the section.

For a linear map \(\varphi \colon E \to F\), a jet \(j \in J^k E\) and a linear form \(\lambda\) on \(F\), we have
\begin{equation}
    \label{eq:jet-functor-values}
    J^k(\varphi)(j)(P \otimes \lambda) = j(P \otimes \lambda \circ \varphi),
\end{equation}
as one sees on the elements \(\omega \otimes u\); we apply this also when \(E\) and \(F\) are themselves jet modules.

\begin{example}
    \label{ex:pairing-line}
    On the polynomial line, take \(E = A\) and \(\alpha = \id_A\).
    The element \(\omega \otimes f\) of \(J^1 A \otimes_A A\) has the values \(\omega(f)\) on \(1 \otimes \id_A\) and \(\omega(\partial \cdot f) = f g_1 + f' g_0\) on \(\partial \otimes \id_A\), where \((g_0, g_1)\) are the values of \(\omega\).
    These are the values of the jet \(\omega f\) of \cref{ex:jets-line}, in accordance with \(\omega \otimes f = \omega f \otimes 1\), and \(J^1 A \otimes_A A\) is \(J^1 A\) with its pointwise structure.
\end{example}

Jets are determined by their values, and the next proposition gives the formula.
It uses dual bases (\cref{def:dual-basis}); by \cref{cor:pbw}, \(D^k\) has a dual basis \((b_a, \beta^a)\) as a left \(A\)-module, with \(\beta^a \in J^k A\).

\begin{proposition}
    \label{prop:jet-reconstruction}
    Let \((e_i, \xi^i)\) be a dual basis of \(E\) and \((b_a, \beta^a)\) one of \(D^k\).
    Every \(j \in J^k E\) satisfies \(j = \sum_{a, i} j(b_a \otimes \xi^i)\, \beta^a \otimes e_i\).
    In particular jets are determined by their values, \(J^k E \cong (D^k \otimes_A E^\vee)^\vee\), and \(J^k E\) is finitely generated projective.
\end{proposition}

\begin{proof}
    Both sides are additive in \(j\), and we may take \(j = \omega \otimes u\).
    Then \(\omega \otimes u = \sum_i \omega\, \xi^i(u) \otimes e_i\), and for \(f \in A\) the dual basis of \(D^k\) gives \(\omega f = \sum_a (\omega f)(b_a)\, \beta^a = \sum_a \omega(b_a \cdot f)\, \beta^a\).
    Taking \(f = \xi^i(u)\) gives the formula, since \(\omega(b_a \cdot \xi^i(u)) = (\omega \otimes u)(b_a \otimes \xi^i)\).
    The formula shows that the elements \(\beta^a \otimes e_i\), together with the linear forms \(j \mapsto j(b_a \otimes \xi^i)\), form a dual basis of \(J^k E\), which is therefore finitely generated projective.
    The isomorphism \(J^k E \cong (D^k \otimes_A E^\vee)^\vee\) is proved in \cref{sec:appendix-jets}.
\end{proof}

To compose two operators, we need the jets of the output of the first one, and these are computed from jets of higher order of its input.
Classically, the \(l\)-jet of the \(k\)-jet of a section depends on its derivatives up to order \(k + l\).
In terms of values, applying \(P \in D^k\) to the function \(Q(\alpha(u))\), for \(Q \in D^l\), is applying \(P \cdot Q \in D^{k+l}\) to \(\alpha(u)\).
The comultiplication below is therefore dual to the multiplication of \(D\), the restriction is dual to the inclusion \(D^k \subseteq D^l\), and the counit is evaluation at \(1 \in D^0\).\footnote{%
    The letters \(\varepsilon\) and \(\delta\) for the counit and the comultiplication are the standard notation for comonads; they are unrelated to the cap \(\varepsilon\) of the dagger compact structure of \cref{def:dagger-compact} and to the Kronecker \(\delta^i_j\).%
}

\begin{definition}
    \label{def:jet-maps}
    The \emph{restrictions} \(\pi_{l,k} \colon J^l E \to J^k E\) for \(k \le l\), the \emph{counit} \(\varepsilon \colon J^0 E \to E\) and the \emph{comultiplication} \(\delta_{k,l} \colon J^{k+l} E \to J^k J^l E\) are the natural maps determined by
    \begin{gather*}
        \pi_{l,k}(j)(P \otimes \alpha) = j(P \otimes \alpha), \qquad
        \alpha(\varepsilon(j)) = j(1 \otimes \alpha), \\
        \delta_{k,l}(j)\bigl(P \otimes (Q \otimes \alpha)\bigr) = j(P Q \otimes \alpha),
    \end{gather*}
    for \(P \in D^k\), \(Q \in D^l\) and \(\alpha \in E^\vee\), where \(Q \otimes \alpha\) is regarded as the linear form \(j' \mapsto j'(Q \otimes \alpha)\) on \(J^l E\).
\end{definition}

By \cref{prop:jet-reconstruction}, applied to \(E\) and to \(J^l E\), there is at most one map with each of these properties, and \cref{sec:appendix-jets} gives explicit formulas that show that the maps exist and are natural.
For instance, \(\pi_{l,k}(\omega \otimes u)\) is the restriction of \(\omega\) to \(D^k\), tensored with \(u\), and \(\varepsilon(\omega \otimes u) = \omega(1)\, u\).
On jets of sections, \(\pi_{l,k}(j^l(u)) = j^k(u)\), \(\varepsilon(j^0(u)) = u\) and \(\delta_{k,l}(j^{k+l}(u)) = j^k(j^l(u))\), as one sees by comparing values.
Classically, \(\pi_{l,k}\) forgets the derivatives of order above \(k\), \(\varepsilon\) takes the value of a \(0\)-jet, and \(\delta_{k,l}\) regards a jet of order \(k + l\) as a \(k\)-jet of \(l\)-jets.

\begin{example}
    \label{ex:comultiplication-line}
    On the polynomial line, \(D^2\) is free on \(1\), \(\partial\) and \(\partial^2\), so a \(2\)-jet \(\omega\) is given by the triple \((\omega(1), \omega(\partial), \omega(\partial^2))\), which is \((f, f', f'')\) for the jet of \(f\).
    By \cref{prop:jet-reconstruction}, an element of \(J^1 J^1 A\) is determined by its four values on \(P \otimes (Q \otimes \id_A)\) with \(P, Q \in \{1, \partial\}\), and \(\delta_{1,1}(\omega)\) takes the value \(\omega(P Q)\) there.
    For the jet of \(f\), these values are
    \[
        \begin{array}{c|cc}
            & Q = 1 & Q = \partial \\
            \hline
            P = 1 & f & f' \\
            P = \partial & f' & f''
        \end{array}
    \]
    The rows are the value \((f, f')\) and the derivative \((f', f'')\) of the \(1\)-jet of \(f\), so \(\delta_{1,1}\) sends the \(2\)-jet of \(f\) to the \(1\)-jet of its \(1\)-jet.
\end{example}

The ungraded jet comonad and its PDE coalgebras go back to Marvan \cite{marvan1987}.
Khavkine and Schreiber give a synthetic treatment \cite{khavkine-schreiber2017}, and Wallbridge describes its co-Kleisli morphisms as linear differential operators \cite{wallbridge2020}.
Flood, Mantegazza and Winther construct jet functors for differential calculi, with composition maps giving the order bound for differential operators \cite[Theorems~5.44 and~8.38; Proposition~10.3]{flood-mantegazza-winther2025}.
Here the jet functors come from a filtered Lie--Rinehart enveloping algebra, and comultiplication is dual to its multiplication.
The maps \(\pi\), \(\varepsilon\) and \(\delta\) satisfy the axioms of a comonad, graded by the order.
Recall that a comonad on a category is an endofunctor \(J\) with natural transformations \(\varepsilon \colon J \Rightarrow \Id\) and \(\delta \colon J \Rightarrow J J\) satisfying the counit and coassociativity laws; in the graded version the functor carries an order.

\begin{definition}
    \label{def:graded-comonad}
    A \emph{comonad graded by \((\N, \le, +)\)} on a category \(\cat C\) consists of endofunctors \(J^k\), natural transformations \(\pi_{l,k} \colon J^l \Rightarrow J^k\) for \(k \le l\), \(\varepsilon \colon J^0 \Rightarrow \Id\) and \(\delta_{k,l} \colon J^{k+l} \Rightarrow J^k J^l\), satisfying
    \begin{gather*}
        \pi_{k,k} = \id, \qquad \pi_{l,k} \circ \pi_{m,l} = \pi_{m,k}, \qquad
        J^k(\varepsilon) \circ \delta_{k,0} = \id, \qquad \varepsilon \circ \delta_{0,k} = \id, \\
        J^k(\delta_{l,m}) \circ \delta_{k,l+m} = \delta_{k,l} \circ \delta_{k+l,m},
    \end{gather*}
    and, for \(k \le m\) and \(l \le n\),
    \[
        \delta_{k,l} \circ \pi_{m+n, k+l} = J^k(\pi_{n,l}) \circ \pi_{m,k} \circ \delta_{m,n}.
    \]
    This is the dual of the graded monads of \cite{katsumata2014, fujii-katsumata-mellies2016}.
\end{definition}

For jets, the counit laws say that the \(k\)-jet of the \(0\)-jet, and the \(0\)-jet of the \(k\)-jet, are the original \(k\)-jet.
Coassociativity says that iterated jets can be formed in either order.
Compatibility with restrictions says that forgetting derivatives commutes with forming jets of jets.
The diagrams expressing these laws appear with their verification in \cref{sec:appendix-jets}.

\begin{proposition}
    \label{prop:jet-comonad}
    The jet functors on \(\fgpMod_A\), with \(\pi\), \(\varepsilon\) and \(\delta\), form a comonad graded by \((\N, \le, +)\).
\end{proposition}

The proof is in \cref{sec:appendix-jets}: comparing values, the axioms reduce to the unit and associativity laws of \(D\) and to the inclusions \(D^k \subseteq D^l\).

\subsection{The category of differential operators}
\label{sec:diffop}

The \emph{co-Kleisli category} of an ordinary comonad \(J\) has the same objects, the maps \(J X \to Y\) as morphisms from \(X\) to \(Y\), and the composite of \(P \colon J X \to Y\) and \(Q \colon J Y \to Z\) is \(Q \circ J(P) \circ \delta\).
For the jet comonad, a morphism of degree \(k\) is an operator of order at most \(k\), and the degrees add under composition.
The morphisms therefore come in one set for each degree, with maps that raise the degree, and they form a locally graded category.

\begin{definition}
    \label{def:locally-graded-category}
    A \emph{locally graded category} \(\cat C\) over \((\N, \le, +)\) has objects and sets \(\Hom_k(X, Y)\) for \(k \in \N\).
    It has change-of-degree maps \(\iota_{k,l} \colon \Hom_k(X, Y) \to \Hom_l(X, Y)\) for \(k \le l\), identities \(\id_X \in \Hom_0(X, X)\), and compositions
    \[
        \Hom_l(Y, Z) \times \Hom_k(X, Y) \to \Hom_{l+k}(X, Z),
    \]
    such that composition is associative and unital, the \(\iota\) are functorial, and \(\iota(g) \circ \iota(f) = \iota(g \circ f)\) \cite{mcdermott-uustalu2022}.
    It is \emph{preadditive} if the hom-sets are abelian groups and composition and \(\iota\) are additive.
\end{definition}

Here functoriality means \(\iota_{k,k} = \id\) and \(\iota_{l,m} \circ \iota_{k,l} = \iota_{k,m}\), and the last condition holds for all change-of-degree maps for which both sides are defined.
This is the notion of \cite[Definition~9]{mcdermott-uustalu2022} for the grading \((\N, \le, +)\), where the change-of-degree maps are called coercions.

\begin{definition}
    \label{def:dagger-locally-graded}
    A \emph{dagger} on a locally graded category consists of maps \(\Hom_k(X, Y) \to \Hom_k(Y, X)\), \(P \mapsto P^\dagger\), that satisfy the conditions of \cref{def:dagger-category} and \(\iota(P)^\dagger = \iota(P^\dagger)\).
\end{definition}

The condition \((Q \circ P)^\dagger = P^\dagger \circ Q^\dagger\) makes sense because both sides have degree \(l + k = k + l\).
A dagger preserves degrees by definition, and \cref{ex:colimit-daggers} shows that this is a restriction.

\begin{definition}
    \label{def:diffop}
    The locally graded category \(\DiffOp\) of \emph{differential operators} has the finitely generated projective \(A\)-modules as objects and
    \[
        \Diff^k(E, F) = \Hom_A(J^k E, F),
    \]
    the \emph{operators of order at most \(k\)} relative to \((A,L)\).
    The composite of \(P \in \Diff^k(E, F)\) and \(Q \in \Diff^l(F, H)\) is \(Q \circ J^l(P) \circ \delta_{l,k}\), the identity is \(\varepsilon\), and \(\iota_{k,l}(P) = P \circ \pi_{l,k}\).
\end{definition}

The composite \(Q \circ P\) in \(\DiffOp\) is the composite of linear maps
\[
    \begin{tikzcd}
        J^{l+k} E \rar["\delta_{l,k}"] & J^l J^k E \rar["J^l(P)"] & J^l F \rar["Q"] & H .
    \end{tikzcd}
\]
This is the co-Kleisli category of the graded comonad of \cref{prop:jet-comonad}, and it is a preadditive locally graded category; the verification is in \cref{sec:appendix-jets}.
An operator \(P \in \Diff^k(E, F)\) acts on sections by \(P(u) = P(j^k(u))\), and since \(\delta_{l,k}\) sends \(j^{l+k}(u)\) to \(j^l(j^k(u))\), the composite acts by \((Q \circ P)(u) = Q(P(u))\).
This action can lose information.
For \(A = K\) and a non-zero Lie algebra \(L\) with zero anchor, \(X \in L\) gives a non-zero element of \(\Diff^1(A, A) \cong D^1\) but acts as zero on every element of \(A\).
Classically, operators are bundle maps out of jet bundles, and the composite of \(P\) and \(Q\) is the prolongation \(J^l(P) \circ \delta_{l,k}\) of \(P\) to \(l\)-jets followed by \(Q\).

\begin{example}
    \label{ex:composition-line}
    On the polynomial line, \cref{ex:jets-line} identifies the first-order operator \(p_0 + p_1 \partial\) with the map \(P \colon J^1 A \to A\), \((g_0, g_1) \mapsto p_0 g_0 + p_1 g_1\).
    Let \(Q = q_0 + q_1 \partial\) be a second one.
    On a \(2\)-jet \(\omega\) with values \((g_0, g_1, g_2)\), the comultiplication gives the element of \(J^1 J^1 A\) with the values of \cref{ex:comultiplication-line}.
    By \eqref{eq:jet-functor-values}, the prolongation satisfies \(J^1(P)(j')(R \otimes \id_A) = j'(R \otimes P)\) for \(R \in D^1\).
    Applied to the comultiplication of \(\omega\), it gives the \(1\)-jet whose value on \(R\) is \(\omega(R \cdot (p_0 + p_1 \partial))\).
    Its values on \(1\) and \(\partial\) are
    \[
        p_0 g_0 + p_1 g_1
        \qquad \text{and} \qquad
        \omega\bigl(\partial \cdot (p_0 + p_1 \partial)\bigr) = p_0' g_0 + (p_0 + p_1') g_1 + p_1 g_2,
    \]
    which for the jet of \(f\) are \(P(f)\) and its derivative \(P(f)'\).
    Applying \(Q\) gives
    \[
        (Q \circ P)(g_0, g_1, g_2) = q_0 (p_0 g_0 + p_1 g_1) + q_1 \bigl(p_0' g_0 + (p_0 + p_1') g_1 + p_1 g_2\bigr),
    \]
    the operator of the product \((q_0 + q_1 \partial)(p_0 + p_1 \partial)\) in the Weyl algebra.
\end{example}

\begin{proposition}
    \label{prop:order-zero}
    Precomposition with \(\varepsilon\) is an isomorphism \(\Hom_A(E, F) \cong \Diff^0(E, F)\), compatible with identities, composition and addition.
\end{proposition}

\begin{proof}
    Since \(D^0 = A\), the jets \(J^0 A = \Hom_A(A, A)\) are the multiplications by elements of \(A\), and on \(D^0\) the pointwise and right structures agree.
    Hence \(\varepsilon(\omega \otimes u) = \omega(1)\, u\) is an isomorphism \(J^0 E \cong E\), with inverse \(u \mapsto \id_A \otimes u\), and precomposition with it is a bijection \(\Hom_A(E, F) \cong \Diff^0(E, F)\).
    It sends \(\id_E\) to \(\varepsilon\), the identity of \(\DiffOp\), and it is additive.
    Let \(P \in \Diff^k(E, F)\), and let \(\varphi \colon F \to H\) and \(\psi \colon H \to E\) be linear.
    The naturality of \(\varepsilon\) and the counit laws give
    \begin{gather*}
        (\varphi \circ \varepsilon) \circ P = \varphi \circ \varepsilon \circ J^0(P) \circ \delta_{0,k} = \varphi \circ P \circ \varepsilon \circ \delta_{0,k} = \varphi \circ P, \\
        P \circ (\psi \circ \varepsilon) = P \circ J^k(\psi) \circ J^k(\varepsilon) \circ \delta_{k,0} = P \circ J^k(\psi),
    \end{gather*}
    where the left-hand sides are composites in \(\DiffOp\).
    For \(k = 0\) and \(P = \chi \circ \varepsilon\), the first line shows that composition of order-zero operators is composition of linear maps.
\end{proof}

From now on, order-zero operators are identified with linear maps; in particular \(\hat u \colon A \to E\), \(f \mapsto f u\), and \(\alpha \colon E \to A\) are operators of order zero.
The proof shows that composing an operator with linear maps on either side keeps its order, and that \(P \circ \hat u\) acts on functions by \((P \circ \hat u)(f) = P(f u)\).

\begin{theorem}
    \label{thm:diff-A-A}
    For \(P \in D^k\), let \(\Phi(P) \colon J^k A \to A\) be the map \(\omega \otimes f \mapsto \omega(P \cdot f)\).
    \begin{enumerate}
        \item \(\Phi\) is an isomorphism \(D^k \cong \Diff^k(A, A)\).
        \item It is compatible with the inclusions \(D^k \subseteq D^l\) and the maps \(\iota_{k,l}\).
        \item It sends products in \(D\) to composites of operators: \(\Phi(Q \cdot P) = \Phi(Q) \circ \Phi(P)\).
    \end{enumerate}
\end{theorem}

\begin{proof}
    The map \(\Phi(P)\) is the evaluation \(j \mapsto j(P \otimes \id_A)\), with \(\id_A \in A^\vee\), and it is therefore \(A\)-linear, additive in \(P\), and compatible with the inclusions and restrictions by \cref{def:jet-maps}.
    For \(P \in D^k\) and \(Q \in D^l\), the composite \(\Phi(Q) \circ \Phi(P) = \Phi(Q) \circ J^l(\Phi(P)) \circ \delta_{l,k}\) sends \(j\) to
    \[
        J^l(\Phi(P))(\delta_{l,k}(j))(Q \otimes \id_A)
        = \delta_{l,k}(j)\bigl(Q \otimes (P \otimes \id_A)\bigr)
        = j(Q \cdot P \otimes \id_A),
    \]
    where we used that \(J^l(\varphi)(j')(Q \otimes \alpha) = j'(Q \otimes \alpha \circ \varphi)\) for jets \(j'\) of order \(l\) and that \(\id_A \circ \Phi(P)\) is the linear form \(P \otimes \id_A\) on \(J^k A\).
    Hence \(\Phi(Q) \circ \Phi(P) = \Phi(Q \cdot P)\).
    For \(f \in A\), both \(\Phi(f)\) and \(f \circ \varepsilon\) send \(\omega \otimes g\) to \(\omega(1) f g\), and \(\Phi\) is the identification of \cref{prop:order-zero} on \(D^0\).

    It remains to show that \(\Phi\) is bijective, and we use the dual basis \((b_a, \beta^a)\) of \(D^k\).
    If \(\Phi(P) = 0\), then \(\omega(P) = 0\) for every \(\omega \in J^k A\), and \(P = \sum_a \beta^a(P)\, b_a = 0\).
    Given an operator \(R \colon J^k A \to A\), put \(P = \sum_a R(\beta^a \otimes 1)\, b_a\).
    Every \(\omega \in J^k A\) satisfies \(\omega = \sum_a \omega(b_a)\, \beta^a\) for the pointwise structure, so \(R(\omega \otimes 1) = \sum_a \omega(b_a)\, R(\beta^a \otimes 1) = \omega(P)\), and since \(\omega \otimes f = \omega f \otimes 1\), we get \(R = \Phi(P)\).
\end{proof}

From now on we identify \(\Diff^k(A, A)\) with \(D^k\).
The operator of \(P \in D^k\) acts on functions as \(P\) does, since \(\Phi(P)(j^k(f)) = j^k(1)(P \cdot f) = P(f)\).
For the polynomial line this is the identification of \cref{ex:jets-line}, and \cref{ex:composition-line} is an instance of the multiplicativity.

Here is the matrix picture behind the next proposition.
For \(E = F = A^2\) with its standard basis, choose four elements \(Q^i{}_j \in D^k\).
Using the coordinate injections \(\hat e_i\) and projections \(\xi^j\), form \(P = \sum_{i,j} \hat e_i \circ Q^i{}_j \circ \xi^j\).
Its test from the \(j\)-th input component to the \(i\)-th output component is \(Q^i{}_j\).
Its action on a pair of sections is
\[
    P(f_1, f_2) = \bigl(Q^1_1(f_1) + Q^1_2(f_2),\; Q^2_1(f_1) + Q^2_2(f_2)\bigr).
\]
The proposition says that this matrix description works with dual bases even when the modules are not free.

\begin{proposition}
    \label{prop:tests}
    An operator \(P \in \Diff^k(E, F)\) is determined by its \emph{tests} \(\alpha \circ P \circ \hat u \in D^k\), for \(u \in E\) and \(\alpha \in F^\vee\).
    Explicitly, for dual bases \((e_i, \xi^i)\) of \(E\) and \((e'_j, \xi'^j)\) of \(F\),
    \[
        P = \sum_{i, j} \hat e'_j \circ (\xi'^j \circ P \circ \hat e_i) \circ \xi^i .
    \]
\end{proposition}

\begin{proof}
    The linear map \(\sum_i \hat e_i \circ \xi^i\) sends \(u\) to \(\sum_i \xi^i(u)\, e_i = u\), and it is therefore \(\id_E\); likewise for \(F\).
    Composing \(P\) with these identities on both sides, and using that composition in \(\DiffOp\) is associative and additive, gives the formula.
\end{proof}

The reconstruction formula of \cref{prop:tests} expresses \(P\) as the sum over \(i\) and \(j\) of the composites
\[
    \begin{tikzcd}[column sep = huge]
        E \rar["\xi^i"] & A \rar["\xi'^j \circ P \circ \hat e_i"] & A \rar["\hat e'_j"] & F .
    \end{tikzcd}
\]
After choosing dual bases, operators have finite matrix presentations over \(D\), with entries the tests \(\xi'^j \circ P \circ \hat e_i\).
For free modules these are ordinary matrix entries.
For projective modules a dual basis may be redundant, and the entries satisfy the relations of the chosen dual bases.
Applied to a function \(f\), the test \(\alpha \circ P \circ \hat u\) gives \(\alpha(P(f u))\).

\begin{lemma}
    \label{lem:change-of-degree-injective}
    All change-of-degree maps \(\iota_{k,l} \colon \Diff^k(E, F) \to \Diff^l(E, F)\) are injective.
\end{lemma}

\begin{proof}
    The map \(\iota_{k,l}\) is precomposition with \(\pi_{l,k}\), and it is injective as soon as \(\pi_{l,k}\) is surjective.
    By \cref{cor:pbw}, there is a map of left \(A\)-modules \(r \colon D^l \to D^k\) that restricts to the identity on \(D^k\).
    For \(\omega \in J^k A\), the jet \(\omega \circ r \in J^l A\) restricts to \(\omega\), so \(\pi_{l,k}(\omega \circ r \otimes u) = \omega \otimes u\), and \(\pi_{l,k}\) is surjective.
\end{proof}

By \cref{lem:change-of-degree-injective}, we regard \(\Diff^k(E, F)\) as a subgroup of \(\Diff^l(E, F)\) for \(k \le l\): an operator of order at most \(k\) is also one of order at most \(l\).

For each pair of objects \(E,F\), take the filtered colimit of the groups \(\Diff^k(E,F)\) and their change-of-degree maps in the category of abelian groups.
Since these maps are injective, the colimit is their union:
\[
    \colim_{k\in\N} \Diff^k(E, F) = \bigcup_{k\in\N} \Diff^k(E,F).
\]
These groups are the hom-groups of an ordinary preadditive category with the same objects, which we call the \emph{colimit category}.
Composition is induced by the graded composition \(\Diff^l(F,H)\times\Diff^k(E,F)\to\Diff^{k+l}(E,H)\).
Compatibility with the change-of-degree maps makes it independent of the degrees chosen to represent the operators.
The next two examples are daggers on this category that change the order of operators, and they are the reason for working with the locally graded category.

\begin{example}
    \label{ex:colimit-daggers}
    Let \(K = \Q\), \(A = \Q[x]\) and \(L = A \partial\), so that \(D\) is the Weyl algebra.
    Every finitely generated projective module over \(\Q[x]\) is free, and after a basis of every module is fixed, the morphisms \(A^n \to A^m\) in the colimit category are the \(m \times n\) matrices over \(D\), by \cref{prop:tests}.
    The swap \(x \leftrightarrow \partial\) extends to an anti-involution \(P \mapsto P^\ddagger\) of the ring \(D\), since it respects \(\partial x - x \partial = 1\).
    Applying it to every entry and transposing, \((P^i{}_j) \mapsto ((P^j{}_i)^\ddagger)\), is a dagger on the colimit category, and it sends the operator \(x\) of order zero to the operator \(\partial\) of order one.
\end{example}

\begin{example}
    \label{ex:colimit-daggers-abelian}
    Let \(A = K = \Q\) and \(L = \Q^2\), abelian with zero anchor, so that \(D = \Q[X, Y]\) and \(D^k\) consists of the polynomials of degree at most \(k\).
    The automorphism \(X \mapsto X\), \(Y \mapsto -Y + X^2\) is involutive, and applying it to every entry and transposing gives, as in \cref{ex:colimit-daggers}, a dagger on the colimit category.
    It fixes the operators of order zero and sends \(Y\), of order one, to an operator of order two.
\end{example}

\section{Differential operators and formal adjoints}
\label{sec:formal-adjoints}

We now classify daggers on the whole category \(\DiffOp\).
On the object \(A\), such a dagger restricts to one on \(D\), described in \cref{thm:anti-involutions-D} by a real structure and a divergence.
On every other module, it also requires a unimodular Hermitian form.
Together these data determine the formal adjoint of each operator.
The proof uses the \emph{tests} of \cref{prop:tests}: an operator is determined by composing it with a section at its input and a linear form at its output, leaving an operator from \(A\) to \(A\).

\subsection{Classification of daggers}
\label{sec:classification}

\begin{proposition}
    \label{prop:dagger-additive-diffop}
    Every dagger on \(\DiffOp\) is additive in each degree.
\end{proposition}

\begin{proof}
    The proof of \cref{prop:dagger-additive} uses only that composition is additive, that there is a zero object, and that the objects \(F \oplus F\) are biproducts.
    In \(\DiffOp\) composition is additive, and the zero module is a zero object, since \(J^k 0 = 0\) and so \(\Diff^k(0, F) = 0 = \Diff^k(E, 0)\).
    The direct sum \(F \oplus F\), with the linear injections and projections, is a biproduct whose structure maps have degree zero, by \cref{prop:order-zero}.
    In that proof every composite consists of one morphism of degree \(k\) and otherwise of morphisms of degree zero, so it stays in degree \(k\), and the argument applies verbatim in each degree.
\end{proof}

\begin{definition}
    \label{def:dagger-data}
    A \emph{metric-divergence structure} on \((A, L)\) consists of
    \begin{itemize}
        \item a real structure \((f \mapsto \conj f, \theta)\) (\cref{def:real-structure});
        \item a real divergence \(\dvg\);
        \item a unimodular Hermitian form on every finitely generated projective \(A\)-module, with the standard form \(\langle f, g \rangle = f \conj g\) on \(A\).
    \end{itemize}
\end{definition}

For free modules with their standard forms, the classification has a familiar shape.
If \(P \colon A^n \to A^m\) is represented by a matrix \((P^i{}_j)\) over \(D\), its adjoint is the transposed matrix \(((P^j{}_i)^\dagger)\).
For the trivial real structure, a vector-field entry satisfies \(X^\dagger = -X - \dvg(X)\); the divergence is the correction to ordinary matrix transposition.

\begin{theorem}
    \label{thm:formal-adjoints}
    Daggers on \(\DiffOp\) correspond bijectively to metric-divergence structures on \((A, L)\).
    Under this correspondence:
    \begin{itemize}
        \item on \(\Diff^k(A, A) = D^k\) it restricts to the dagger on \(D\) of \cref{thm:anti-involutions-D};
        \item the forms are \(\langle u, v \rangle = \hat v^\dagger \circ \hat u\), so that \(v^\flat = \hat v^\dagger\);
        \item the dagger of \(P \in \Diff^k(E, F)\) is the unique operator \(P^\dagger \in \Diff^k(F, E)\) satisfying the \emph{adjunction formula}
        \begin{equation}
            \label{eq:adjunction}
            u^\flat \circ P^\dagger \circ \hat v = (v^\flat \circ P \circ \hat u)^\dagger \qquad \text{in } D^k
        \end{equation}
        for all \(u \in E\) and \(v \in F\).
    \end{itemize}
\end{theorem}

The adjunction formula compares the two rows of the diagram
\[
    \begin{tikzcd}
        A \rar["\hat u"] & E \rar["P"] & F \rar["v^\flat"] & A \\
        A & E \lar["u^\flat"] & F \lar["P^\dagger"] & A \lar["\hat v"]
    \end{tikzcd}
\]
whose composites are elements of \(D^k\).
The lower row is the upper row with every arrow replaced by its dagger, since \(\hat u^\dagger = u^\flat\) and \((v^\flat)^\dagger = \hat v\), and the formula says that its composite is the dagger of the composite of the upper row.
For operators of order zero the adjunction formula reads \(\langle \varphi^\dagger(v), u \rangle = \conj{\langle \varphi(u), v \rangle}\), which is the definition of the adjoint of a linear map.
For operators of higher order, the tests are elements of \(D^k\) rather than functions, and the dagger on the right-hand side is the dagger on \(D\), which contains the divergence.

\begin{proof}
    We begin with a dagger on \(\DiffOp\), and as in the proof of \cref{thm:daggers-are-forms} we first look at the object \(A\).
    By \cref{prop:dagger-additive-diffop,thm:diff-A-A}, the dagger restricts to additive maps \(D^k \to D^k\) that are compatible with the inclusions \(D^k \subseteq D^l\), reverse products and square to the identity, and \(1^\dagger = 1\) because \(1\) is the identity of \(A\).
    Since the filtration is exhaustive, these maps form a filtration-preserving dagger on \(D\), and \cref{thm:anti-involutions-D} gives the real structure and the real divergence.
    In particular \(f^\dagger = \conj f\) for \(f \in A = D^0\).

    Next we construct the forms.
    For \(u, v \in E\), the composite \(\hat v^\dagger \circ \hat u\) is an operator of order zero from \(A\) to \(A\), that is, a function, and we put \(\langle u, v \rangle = \hat v^\dagger \circ \hat u\).
    As in the proof of \cref{thm:daggers-are-forms}, this form is biadditive, linear in \(u\), semilinear in \(v\) and Hermitian, because the dagger is additive by \cref{prop:dagger-additive-diffop}, reverses composites, squares to the identity and is the bar on \(D^0 = A\).
    Since \(\langle u, v \rangle = \hat v^\dagger(u)\), the map \(\flat\) is the composite of the bijections \(v \mapsto \hat v\) and \(\hat v \mapsto \hat v^\dagger\) between operators of order zero, so the form is unimodular.
    On \(A\) itself \(\hat g = g\), so \(\langle f, g \rangle = \conj g\, f\) is the standard form.
    The adjunction formula holds because \(u^\flat = \hat u^\dagger\) and \(\hat v = (v^\flat)^\dagger\):
    \[
        u^\flat \circ P^\dagger \circ \hat v = \hat u^\dagger \circ P^\dagger \circ (v^\flat)^\dagger = (v^\flat \circ P \circ \hat u)^\dagger .
    \]
    The adjunction formula determines \(P^\dagger\), because every linear form on \(E\) is \(u^\flat\) for some \(u\), and an operator is determined by its tests (\cref{prop:tests}).

    Conversely, let a metric-divergence structure be given, and write \(P \mapsto P^\dagger\) for the dagger on \(D\) given by \cref{thm:anti-involutions-D}.
    For \(f, g \in A\) and \(Q \in D\) we have \(f\, Q^\dagger g = (\conj g\, Q \conj f)^\dagger\).
    Choose dual bases \((e_i, \xi^i)\) of \(E\) and \((e'_j, \xi'^j)\) of \(F\).
    Let \(e^i\) and \(e'^j\) be their reciprocal bases (\cref{def:reciprocal-basis}), so that \((e^i)^\flat = \xi^i\) and \((e'^j)^\flat = \xi'^j\).
    We use the dual-basis expansion to form a candidate adjoint.
    The adjunction formula prescribes its entries when tested against the output vectors \(e'_j\) and the reciprocal input vectors \(e^i\), so we put
    \[
        P^\dagger = \sum_{i, j} \hat e_i \circ \bigl((e'_j)^\flat \circ P \circ \widehat{e^i}\bigr)^\dagger \circ \xi'^j
        \qquad \text{in } \Diff^k(F, E).
    \]
    To check that this operator satisfies the adjunction formula, note that \(u^\flat \circ \hat e_i = \langle e_i, u \rangle\) and \(\xi'^j \circ \hat v = \xi'^j(v)\) are functions, so
    \begin{align*}
        u^\flat \circ P^\dagger \circ \hat v
        &= \sum_{i, j} \langle e_i, u \rangle \bigl((e'_j)^\flat \circ P \circ \widehat{e^i}\bigr)^\dagger\, \xi'^j(v) \\
        &= \Bigl(\sum_{i, j} \conj{\xi'^j(v)}\, (e'_j)^\flat \circ P \circ \widehat{e^i} \circ \conj{\langle e_i, u \rangle}\Bigr)^\dagger
        = (v^\flat \circ P \circ \hat u)^\dagger,
    \end{align*}
    since \(\sum_j \conj{\xi'^j(v)}\, (e'_j)^\flat = \bigl(\sum_j \xi'^j(v)\, e'_j\bigr)^\flat = v^\flat\) and \(\sum_i \conj{\langle e_i, u \rangle}\, e^i = \sum_i \langle u, e_i \rangle\, e^i = u\) by \cref{prop:reciprocal-expansion}.
    As before, the adjunction formula determines \(P^\dagger\), which is therefore independent of the dual bases.
    We now check the dagger axioms using the same formula.
    The change-of-degree maps commute with composition by operators of order zero and correspond to the inclusions \(D^k \subseteq D^l\).
    These inclusions commute with the dagger on \(D\), so \(\iota(P)^\dagger = \iota(P^\dagger)\).
    The identity satisfies \(u^\flat \circ \hat v = \langle v, u \rangle = \conj{\langle u, v \rangle} = (v^\flat \circ \hat u)^\dagger\), so \(\id^\dagger = \id\).
    Applying the dagger on \(D\) to \eqref{eq:adjunction} gives \(v^\flat \circ P \circ \hat u = (u^\flat \circ P^\dagger \circ \hat v)^\dagger\), which is the adjunction formula for \(P^\dagger\), so \(P^{\dagger\dagger} = P\).
    For contravariance, let \(Q \in \Diff^l(F, H)\) and \(w \in H\).
    The two reciprocal-basis expansions of \cref{prop:reciprocal-expansion} give
    \[
        \id_F = \sum_j \hat e'_j \circ (e'^j)^\flat
        = \sum_j \widehat{e'^j} \circ (e'_j)^\flat .
    \]
    Inserting the first expansion and applying \eqref{eq:adjunction} twice gives
    \begin{align*}
        u^\flat \circ P^\dagger \circ Q^\dagger \circ \hat w
        &= \sum_j \bigl((e'_j)^\flat \circ P \circ \hat u\bigr)^\dagger \bigl(w^\flat \circ Q \circ \widehat{e'^j}\bigr)^\dagger \\
        &= \Bigl(w^\flat \circ Q \circ \Bigl(\sum_j \widehat{e'^j} \circ (e'_j)^\flat\Bigr) \circ P \circ \hat u\Bigr)^\dagger
        = (w^\flat \circ Q \circ P \circ \hat u)^\dagger,
    \end{align*}
    so \((Q \circ P)^\dagger = P^\dagger \circ Q^\dagger\).

    It remains to show that the two constructions are inverse to each other.
    Starting from a metric-divergence structure, the dagger on \(D\) satisfies the adjunction formula on \(D^k = \Diff^k(A, A)\), since \(f^\flat = \conj f\) and \(\conj f\, Q^\dagger g = (\conj g\, Q f)^\dagger\).
    Hence the constructed dagger restricts to the given one on \(D\), and \cref{thm:anti-involutions-D} returns the real structure and the divergence.
    The linear form \(u^\flat\) satisfies the adjunction formula for \(\hat u\), because for \(f \in A\) and \(v \in E\) both \(f^\flat \circ u^\flat \circ \hat v\) and \((v^\flat \circ \hat u \circ f)^\dagger\) equal \(\conj f\, \langle v, u \rangle = \conj{f \langle u, v \rangle}\).
    Hence \(\hat u^\dagger = u^\flat\), and the dagger returns the given forms.
    Starting from a dagger, we showed that it satisfies the adjunction formula for its own metric-divergence structure.
    By uniqueness it is the dagger constructed from that structure.
\end{proof}

\begin{example}
    \label{ex:classical-formal-adjoint}
    Let \(A = C^\infty(M, \C)\) and \(L = \Gamma(TM \otimes \C)\) with complex conjugation, and take Hermitian bundle metrics on all bundles and the divergence \(\dvg_\mu\) of a positive density \(\mu\), as in \cref{ex:classical-divergence}.
    Then \eqref{eq:adjunction} is the defining property of the formal adjoint:
    \[
        \int_M \langle P(u), v \rangle\, \mu = \int_M \langle u, P^\dagger(v) \rangle\, \mu
    \]
    for compactly supported sections \(u\) and \(v\).
    To see this, let \(Q = v^\flat \circ P \circ \hat u\), which sends a function \(f\) to \(\langle P(f u), v \rangle\).
    Applied to a function \(g\), the two sides of \eqref{eq:adjunction} give \(\langle P^\dagger(g v), u \rangle = Q^\dagger(g)\), and \(Q^\dagger\) is the formal adjoint of \(Q\) for the pairing
    \[
        \int_M f \conj g\, \mu ,
    \]
    since the formal adjoint for this pairing reverses products and agrees with \(\dagger\) on \(A\) and \(L\), which generate \(D\) (\cref{ex:classical-divergence}).
    Multiplying by \(\conj f\) and integrating gives
    \[
        \int_M \langle P^\dagger(g v), f u \rangle\, \mu
        = \int_M Q^\dagger(g)\, \conj f\, \mu
        = \conj{\int_M Q(f)\, \conj g\, \mu}
        = \int_M \langle g v, P(f u) \rangle\, \mu,
    \]
    which is the conjugate of the formula above for the sections \(f u\) and \(g v\), and every section is a sum of sections of this form.
    With this standard real structure, any other real divergence differs from \(\dvg_\mu\) by a real closed \(1\)-form (\cref{prop:divergences-differ}).
    A general dagger on \(\DiffOp\) may also have indefinite forms and any real structure satisfying \cref{def:real-structure}.
\end{example}

\begin{example}
    \label{ex:weighted-line-dagger}
    On the weighted line of \cref{ex:weighted-line-divergence}, over \(A = \Q[x]\) with \(\dvg(f \partial) = f' + c f\), every finitely generated projective module is free.
    Declaring a basis of each module orthonormal, with the basis \(1\) of \(A\), gives a metric-divergence structure, and by \cref{thm:formal-adjoints} a dagger on \(\DiffOp\).
    In orthonormal bases the reciprocal basis is the basis itself.
    The formula in the proof therefore gives \((P^i{}_j) \mapsto ((P^j{}_i)^\dagger)\) on matrices over \(D\): the transpose, with the dagger \(x^\dagger = x\), \(\partial^\dagger = -\partial - c\) applied to the entries.
    Over \(C^\infty(\R)\), this is the formal adjoint for the density \(e^{c x}\, dx\) and the constant metrics.
\end{example}

\begin{corollary}
    \label{cor:order-zero-recovers}
    For \(L = 0\), \cref{thm:formal-adjoints} is \cref{thm:daggers-are-metrics}.
\end{corollary}

\begin{proof}
    For \(L = 0\) the enveloping algebra is \(D = A\), with \(D^k = A\) for every \(k\).
    Then \(J^k E = E\), all restrictions are the identity, and \(\Diff^k(E, F) = \Hom_A(E, F)\) in every degree, with bijective change-of-degree maps.
    A dagger on \(\DiffOp\) is therefore the same thing as a dagger on \(\fgpMod_A\).
    A real structure on \((A, 0)\) is an involutive ring automorphism of \(A\), with \(\theta = 0\), and the only real divergence is zero.
    A metric-divergence structure is therefore an involution of \(A\) together with a unimodular Hermitian form on every module, standard on \(A\), and the adjunction formula in order zero says that the dagger is the adjoint for these forms.
\end{proof}

\begin{proposition}
    \label{prop:reweighting}
    Let \(\dagger\) be a dagger on \(\DiffOp\) and \(w\) a unit of \(A\) with \(\conj w = w\).
    Then \(P^\ddagger = w^{-1} \circ P^\dagger \circ w\) is a dagger on \(\DiffOp\) with the same real structure and forms.
    Its divergence is \(\dvg + d \log w\), with \(d \log w\) as in \cref{sec:divergences}.
\end{proposition}

\begin{proof}
    Here \(w\) and \(w^{-1}\) stand for multiplication on the source and target of \(P^\dagger\).
    By the adjunction formula in order zero, the dagger of multiplication by \(w\) is multiplication by \(\conj w = w\).
    The factors \(w\) and \(w^{-1}\) cancel in composites.
    This gives \((Q \circ P)^\ddagger = P^\ddagger \circ Q^\ddagger\) and \(\id^\ddagger = \id\).
    Involutivity follows from \(P^{\ddagger\ddagger} = w^{-1} \circ w \circ P \circ w^{-1} \circ w = P\).
    Since \(w\) has order zero, \(\ddagger\) preserves the order and commutes with the change-of-degree maps.
    The forms are unchanged, because \(\hat v^\ddagger \circ \hat u = w^{-1} \langle w u, v \rangle = \langle u, v \rangle\).
    On \(D\) the new dagger is \(Q \mapsto w^{-1} Q^\dagger w\).
    It fixes the functions.
    For \(X \in L\), using \(\theta(X) \cdot w = w \cdot \theta(X) + (\theta(X))(w)\),
    \[
        w^{-1} X^\dagger w = -\theta(X) - w^{-1} (\theta(X))(w) - \conj{\dvg(X)} = -\theta(X) - \conj{\dvg(X) + w^{-1} X(w)},
    \]
    since \((\theta(X))(w) = \conj{X(\conj w)} = \conj{X(w)}\) and \(\conj{w^{-1}} = w^{-1}\).
    By \cref{thm:anti-involutions-D}, the real structure is unchanged and the divergence is \(\dvg + d \log w\).
\end{proof}

Classically, for a positive function \(w\), this is the change of density \(\mu \mapsto w \mu\) of \cref{prop:divergences-differ}.
Consequently, as soon as \(A\) has a real unit \(w\) with \(d \log w \neq 0\), for example a positive non-constant function on a manifold, the real structure and the forms do not determine the divergence.

\begin{remark}
    \label{rem:existence-daggers-diffop}
    By \cref{thm:formal-adjoints}, a dagger on \(\DiffOp\) with a given real structure exists if and only if every finitely generated projective module admits a unimodular Hermitian form and a real divergence exists.
    The first condition concerns only \(A\), and it can fail, as \cref{rem:existence-of-metrics} shows.
    Over a manifold, with complex conjugation, both always hold: every bundle has a Hermitian metric, and every manifold has a positive density, whose divergence is real by \cref{ex:classical-divergence}.
\end{remark}

\subsection{First-order operators and connections}
\label{sec:first-order}

Let the real structure be trivial, so that the forms are symmetric.
Let \(X \in L\), and let \(P\) be a derivation of \(E\) whose symbol is the anchor of \(X\) (\cref{def:gl-algebroid}), so that \(P(f u) = f P(u) + X(f)\, u\).
It defines the operator \(\tilde P \in \Diff^1(E, E)\), \(\omega \otimes u \mapsto \omega(1)\, P(u) + \omega(X)\, u\).

The operator \(\tilde P\) is well defined because \(\omega(X \cdot f) = f \omega(X) + X(f)\, \omega(1)\).
On sections, \(\tilde P(u) = P(u)\), since the jet \(j^1(1)\) takes the value \(1\) on \(1\) and the value \(X(1) = 0\) on \(X\).
Its tests are
\[
    \alpha \circ \tilde P \circ \hat u = \alpha(P(u)) + \alpha(u)\, X \qquad \text{in } D^1,
\]
as one checks by applying both sides to \(\omega \otimes f\) and using \(X \cdot f = f \cdot X + X(f)\).
By \cref{prop:gl-dagger}, \(P\) has an adjoint derivation \(P^\dagger\), whose symbol is the anchor of \(-X\), and we form \(\widetilde{P^\dagger}\) with \(-X\).

\begin{proposition}
    \label{prop:derivation-dagger}
    With the trivial real structure, the dagger of \(\tilde P\) in \(\DiffOp\) is
    \[
        \tilde P^\dagger = \widetilde{P^\dagger} - \dvg(X)\, \id_E,
    \]
    where \(P^\dagger\) is the dagger of \cref{prop:gl-dagger}.
\end{proposition}

\begin{proof}
    By \cref{thm:formal-adjoints} it suffices to compare tests.
    With the trivial real structure, \(f^\dagger = f\) and \(X^\dagger = -X - \dvg(X)\), so \((f X)^\dagger = X^\dagger f = -f X - X(f) - \dvg(X) f\) in \(D^1\).
    The adjunction formula gives
    \begin{align*}
        u^\flat \circ \tilde P^\dagger \circ \hat v
        &= \bigl(\langle P(u), v \rangle + \langle u, v \rangle\, X\bigr)^\dagger \\
        &= \langle P(u), v \rangle - \langle u, v \rangle\, X - X(\langle u, v \rangle) - \dvg(X)\, \langle u, v \rangle.
    \end{align*}
    The tests of \(\widetilde{P^\dagger} - \dvg(X) \id_E\) are \(\langle P^\dagger(v), u \rangle - \langle v, u \rangle\, X - \dvg(X)\, \langle v, u \rangle\), and since the forms are symmetric, the two agree by the defining property of \(P^\dagger\).
\end{proof}

\begin{example}
    \label{ex:weighted-line-derivation}
    On the weighted line, give \(E=A\) its standard symmetric form and let \(P=\partial+h\) for \(h\in A\).
    The adjoint derivation is \(P^\dagger=-\partial+h\), whereas the operator \(\tilde P\) has formal adjoint \(-\partial+h-c\).
    For \(h=0\), the connection \(\nabla_\partial=\partial\) is metric, and its operator is skew-adjoint precisely when \(c=0\).
\end{example}

For an \(L\)-connection \(\nabla\) on \(E\) (\cref{sec:algebroid-representations}), the uniqueness in \cref{prop:gl-dagger} shows, as in \cref{lem:metric-connection-skew}, that \(\nabla\) is metric if and only if \(\nabla_X^\dagger = -\nabla_X\) for every \(X \in L\).

\begin{corollary}
    \label{cor:skew-adjoint-connection}
    With the trivial real structure, for a metric \(L\)-connection \(\nabla\) on \(E\), the operator \(\tilde\nabla_X\) is skew-adjoint in \(\DiffOp\) if and only if \(\dvg(X)\) acts as zero on \(E\), that is, \(\dvg(X)\, u = 0\) for all \(u \in E\).
\end{corollary}

\begin{proof}
    By \cref{prop:derivation-dagger} and \(\nabla_X^\dagger = -\nabla_X\), we have \(\tilde\nabla_X^\dagger = -\tilde\nabla_X - \dvg(X)\, \id_E\).
    Hence \(\tilde\nabla_X\) is skew-adjoint if and only if \(\dvg(X)\, \id_E = 0\).
\end{proof}

If \(E\) is a bundle of positive rank on every component of \(M\), a function that acts as zero on \(E\) vanishes, so the condition is \(\dvg(X) = 0\).
Classically, with the Riemannian divergence, the corollary applies to divergence-free fields, such as Killing fields, while for a general vector field the formal adjoint of \(\nabla_X\) differs from \(-\nabla_X\) by multiplication with \(\dvg(X)\).

\subsection{Green's formula and invariant integrals}
\label{sec:green-formula}

For a vector field on a manifold, integration by parts writes the difference between an operator and its formal adjoint as a divergence.
The same statement holds for every dagger on \(\DiffOp\), with the divergence of its metric-divergence structure.
The set \(\dvg(L) = \{\dvg(X) \mid X \in L\}\) is an additive subgroup of \(A\).

\begin{proposition}
    \label{prop:green-formula}
    For \(P \in \Diff^k(E, F)\), \(u \in E\) and \(v \in F\), there is an \(X \in L\) such that
    \[
        \langle P(u), v \rangle - \langle u, P^\dagger(v) \rangle = \dvg(X).
    \]
\end{proposition}

\begin{proof}
    We first take \(Q \in D\) and \(f, g \in A\), and show that
    \[
        Q(f)\, \conj g - f\, \conj{Q^\dagger(g)} \in \dvg(L).
    \]
    Both sides are additive in \(Q\).
    For \(Q = h \in A\), the difference is zero.
    For \(Q = X \in L\), the formulas for \(X^\dagger\) and the real structure give
    \[
        X(f)\, \conj g - f\, \conj{X^\dagger(g)}
        = X(f)\, \conj g + f X(\conj g) + f\, \conj g\, \dvg(X)
        = \dvg(f\, \conj g\, X).
    \]
    For a product \(Q R\), the difference is the sum of the difference for \(Q\) at \((R(f), g)\) and that for \(R\) at \((f, Q^\dagger(g))\).
    Since \(A\) and \(L\) generate \(D\), the claim follows.

    Now take \(Q = v^\flat \circ P \circ \hat u\).
    Its value at \(1\) is \(\langle P(u), v \rangle\), and by \eqref{eq:adjunction} its dagger is \(u^\flat \circ P^\dagger \circ \hat v\).
    The claim for \(Q\) with \(f = g = 1\), together with Hermitian symmetry, gives the result.
\end{proof}

\begin{corollary}
    \label{cor:green-current}
    For each \(P \in \Diff^k(E, F)\), there is a biadditive map \(B_P \colon E \times F \to L\) such that
    \[
        \langle P(u), v \rangle - \langle u, P^\dagger(v) \rangle
        = \dvg(B_P(u,v)).
    \]
    If the involution preserves \(K\), this map can be chosen \(K\)-sesquilinear, and for the trivial real structure it can be chosen \(K\)-bilinear.
\end{corollary}

\begin{proof}
    Fix an expression for a scalar operator \(Q \in D\) as a sum of products of functions and elements of \(L\).
    For a multiplication operator \(h\), set \(B_h(f,g)=0\), and for \(X \in L\), set \(B_X(f,g)=f\conj g\,X\).
    Extend over sums and products by
    \begin{align*}
        B_{Q+R}(f,g)&=B_Q(f,g)+B_R(f,g),\\
        B_{Q R}(f,g)&=B_Q(R(f),g)+B_R(f,Q^\dagger(g)).
    \end{align*}
    The generator and product calculations in the proof of \cref{prop:green-formula} show that
    \[
        \dvg(B_Q(f,g))=Q(f)\conj g-f\conj{Q^\dagger(g)}.
    \]
    These maps are biadditive.
    If the involution preserves \(K\), the action of every element of \(D\) is \(K\)-linear, so the recursion also makes \(B_Q\) linear in its first input and semilinear in its second.

    Choose dual bases \((e_i,\xi^i)\) of \(E\) and \((e'_j,\xi'^j)\) of \(F\).
    For each entry \(Q^j{}_i=\xi'^j\circ P\circ\hat e_i\), fix such an expression and its map \(B_{Q^j{}_i}\).
    Define
    \[
        B_P(u,v)=\sum_{i,j} B_{Q^j{}_i}\bigl(\xi^i(u),\langle v,e'_j\rangle\bigr).
    \]
    By \cref{prop:tests}, \(P=\sum_{i,j}\hat e'_j\circ Q^j{}_i\circ\xi^i\), and taking daggers gives the corresponding expansion of \(P^\dagger\).
    The two pairings therefore satisfy
    \begin{align*}
        \langle P(u),v\rangle
        &=\sum_{i,j} Q^j{}_i(\xi^i(u))\,\langle e'_j,v\rangle,\\
        \langle u,P^\dagger(v)\rangle
        &=\sum_{i,j}\xi^i(u)\,\conj{(Q^j{}_i)^\dagger(\langle v,e'_j\rangle)}.
    \end{align*}
    Their difference is \(\dvg(B_P(u,v))\) by the scalar identity.
    The asserted additivity and sesquilinearity follow from those of the scalar currents and the forms.
\end{proof}

\begin{remark}
    \label{rem:green-non-canonical}
    The current \(B_P\) constructed in \cref{cor:green-current} depends on the chosen dual bases and expressions for the scalar entries.
    Different currents can have the same divergence, since \(\dvg\) need not be injective.
    In classical examples one can often give a current directly in geometric terms.
    For the Riemannian density and the Laplacian \(\Delta=\dvg\circ\grad\), for instance,
    \[
        g\,\Delta(f)-f\,\Delta(g)
        =\dvg\bigl(g\grad(f)-f\grad(g)\bigr).
    \]
\end{remark}

\begin{example}
    \label{ex:weighted-line-green}
    On the weighted line of \cref{ex:weighted-line-divergence}, the operator \(\partial^2\) has adjoint \((\partial+c)^2\).
    Its Green identity is
    \[
        f''g-f(g''+2cg'+c^2g)
        =\dvg\bigl((f'g-fg'-cfg)\partial\bigr).
    \]
    This follows by differentiating the coefficient of \(\partial\) and adding \(c\) times that coefficient.
\end{example}

Suppose the involution of \(A\) preserves \(K\).
An \emph{invariant integral} is a \(K\)-linear map \(\lambda \colon A \to K\) satisfying \(\lambda(\conj f) = \conj{\lambda(f)}\) and \(\lambda(\dvg(X)) = 0\) for all \(f \in A\) and \(X \in L\).
On a compact manifold without boundary with a positive density \(\mu\), integration against \(\mu\) is an invariant integral for \(\dvg_\mu\).

\begin{corollary}
    \label{cor:invariant-integral-adjoint}
    Every invariant integral satisfies
    \[
        \lambda(\langle P(u), v \rangle)
        = \lambda(\langle u, P^\dagger(v) \rangle)
    \]
    for \(P \in \Diff^k(E, F)\), \(u \in E\) and \(v \in F\).
    Thus the dagger is an adjoint for the Hermitian pairing \((u, v) \mapsto \lambda(\langle u, v \rangle)\) on sections.
\end{corollary}

\begin{proof}
    Apply \(\lambda\) to \cref{prop:green-formula}.
    The pairing is Hermitian because \(\lambda\) commutes with the involution of scalars.
\end{proof}

\begin{example}
    \label{ex:circle-no-invariant-integral}
    On the circle of \cref{ex:circle-divergence}, let \(\dvg(f\partial) = f' + f\).
    The map \(f \mapsto f' + f\) is surjective on smooth periodic functions: for every periodic \(g\), the return map of the equation \(f' + f = g\) has linear part \(e^{-2\pi}\), hence a unique fixed point.
    Hence \(\dvg(L) = A\), and every invariant integral is zero.
    In contrast, for the divergence of the standard density \(\lvert dt \rvert\), integration over the circle is a non-zero invariant integral.
\end{example}

\section{Conclusion and outlook}
\label{sec:outlook}

The three parts of this paper describe the geometric data carried by daggers on a sequence of algebraic structures.
On a single vector space an anti-involution of the matrix ring is a form up to scale, and on all vector spaces, or all vector bundles, a dagger is a Hermitian metric on every object.
Differentiating daggers on Lie groups and groupoids gives daggers on Lie algebras and algebroids, and compatibility with them is unitarity of representations and metricity of connections.
On differential operators, a dagger is a unimodular Hermitian form on every finitely generated projective module together with a divergence and a real structure, and it acts as the formal adjoint.
Green's formula expresses the difference between the pairings of an operator and its adjoint against sections as a divergence.

\paragraph{Changing divergences.}
By \cref{prop:divergences-differ}, real divergences for a fixed real structure differ by real closed \(1\)-forms.
Reweighting by a unit accounts for the differences of the form \(d\log w\) (\cref{prop:reweighting}), while \cref{ex:circle-divergence} gives a closed form that does not arise this way.
For a general real closed \(1\)-form, the corresponding change of formal adjoint need not be conjugation by a unit.
A next question is how to compare these daggers and identify which closed forms yield equivalent adjoint structures.

\paragraph{Information geometry.}
A statistical manifold carries a metric on \(E = TM\) and a pair of connections \(\nabla\) and \(\nabla^*\) that are dual in the sense that
\[
    X(\langle u, v \rangle) = \langle \nabla_X u, v \rangle + \langle u, \nabla^*_X v \rangle .
\]
By \eqref{eq:gl-dagger}, this says \(\nabla^*_X = -\nabla_X^\dagger\) in \(\gl(E)\), so duality of connections is the involution \(P \mapsto -P^\dagger\) on derivations, and metric connections are its fixed points.
A next question is how this involution relates the curvatures of dual connections, both on manifolds and on Lie algebroids.

\paragraph{Differential equations.}
An operator \(P \in \Diff^k(E, F)\) defines the equation \(P(u) = 0\), and a dagger on \(\DiffOp\) assigns to it the adjoint equation \(P^\dagger(v) = 0\).
\Cref{prop:green-formula} relates them by a divergence, and \cref{cor:invariant-integral-adjoint} turns this into an adjoint identity when an invariant integral exists.
The jet comonad of \cref{sec:jets} gives two further settings for the relation between an equation and its adjoint.

Differential equations have been described as coalgebras of the ungraded jet comonad \cite{marvan1987, khavkine-schreiber2017}, and recent synthetic work concerns the infinite jet bundles involved \cite{giotopoulos-et-al2026}.
The graded jet comonad of \cref{prop:jet-comonad} has a co-Eilenberg--Moore category, which is locally graded, dually to the Eilenberg--Moore category of a graded monad \cite{fujii-katsumata-mellies2016, mcdermott-uustalu2022}.
A graded coalgebra can encode a family of modules \(E_l\) with compatible maps
\[
    E_{k+l} \longrightarrow J^k E_l,
\]
satisfying counit and coassociativity identities.
The questions are how such families describe prolongation and integrability \cite{goldschmidt1967}, and whether a dagger on \(\DiffOp\) induces a construction on these coalgebras.

Tensoring with the filtered enveloping algebra gives the functors \(T^k E = D^k \otimes_A E\), which form a monad graded by \((\N, \le, +)\), with the multiplication of \(D\).
By \cref{prop:jet-reconstruction}, \(J^k E \cong (T^k E^\vee)^\vee\), so this monad is dual to the jet comonad.
A related question is whether the dagger relates graded coalgebras of \(J\) to graded algebras of \(T\), and how such a relation interacts with their co-Kleisli and Kleisli categories.

\paragraph{Hodge theory.}
On a compact Riemannian manifold without boundary, the Hodge decomposition splits differential forms into exact, coexact and harmonic parts, orthogonal to each other for the \(L^2\) product.
The exterior differential on forms gives a complex of operators in \(\DiffOp\) with \(d^2 = 0\).
Writing \(d_p\) for the differential from \(p\)-forms to \((p+1)\)-forms, the operators \(d_{p-1}\) and \(d_p^\dagger\) are orthogonal because
\[
    d_{p-1}^\dagger \circ d_p^\dagger = (d_p \circ d_{p-1})^\dagger = 0.
\]
The exact and coexact spaces are not \(A\)-submodules in general, so this decomposition is not a splitting in \(\DiffOp\).
For the usual \(L^2\) orthogonality it uses an invariant integral, positivity and a Green operator, whose existence is the analytic part of Hodge theory.

\paragraph{String diagrams.}
The dagger compact structure of Hermitian modules (\cref{thm:herm-dagger-compact}) makes the string diagrams of categorical quantum mechanics \cite{abramsky-coecke2004, abramsky-coecke2008, heunen-vicary2019} available for bundles with metrics.
Raising and lowering indices with the reciprocal basis are the cup and the cap of this structure.
A next question is how string diagrams express covariant differentiation and curvature contractions in Riemannian geometry and general relativity.

\printbibliography

\appendix

\section{Additional proofs}
\label{sec:appendix}

\subsection{Dagger compact structure}
\label{sec:appendix-compact}

We prove \cref{thm:herm-dagger-compact}.
Throughout, \(A\) is a commutative ring with an involution \(f \mapsto \conj f\), and all modules are Hermitian.
Sums of pure tensors \(u \otimes u'\) exhaust \(E \otimes E'\), so a sesquilinear form on a tensor product is determined by its values on pairs of pure tensors, and two linear maps out of a tensor product agree as soon as they agree on pure tensors.

We begin with the form on \(E \otimes E'\).
A sesquilinear form on a module \(E\) is the same thing as an \(A\)-linear map \(\conj E \to E^\vee\), \(\conj v \mapsto \langle -, v \rangle\), and it is unimodular if and only if this map is an isomorphism.
The maps \(\conj u \otimes \conj{u'} \mapsto \conj{u \otimes u'}\) and \(\alpha \otimes \beta \mapsto (u \otimes u' \mapsto \alpha(u) \beta(u'))\) are isomorphisms \(\conj E \otimes \conj{E'} \cong \conj{E \otimes E'}\) and \(E^\vee \otimes E'^\vee \cong (E \otimes E')^\vee\), the second because the modules are finitely generated projective.
Composing the inverse of the first, the tensor product of the isomorphisms \(\conj E \cong E^\vee\) and \(\conj{E'} \cong E'^\vee\), and the second, we obtain an isomorphism \(\conj{E \otimes E'} \cong (E \otimes E')^\vee\).
It is the \(\flat\) of a unimodular sesquilinear form on \(E \otimes E'\), and this form satisfies \(\langle u \otimes u', v \otimes v' \rangle = \langle u, v \rangle \langle u', v' \rangle\).
The form is Hermitian, since \(\langle v, u \rangle \langle v', u' \rangle = \conj{\langle u, v \rangle \langle u', v' \rangle}\).
The module \(E \otimes E'\) is finitely generated projective, so it is a Hermitian module, and the unit \(A\) carries the standard form.

Next, for linear maps \(\varphi \colon E \to F\) and \(\psi \colon E' \to F'\), the adjoint of \(\varphi \otimes \psi\) is \(\varphi^\dagger \otimes \psi^\dagger\), because on pure tensors
\begin{align*}
    \langle \varphi(u) \otimes \psi(u'), w \otimes w' \rangle
    &= \langle \varphi(u), w \rangle \langle \psi(u'), w' \rangle \\
    &= \langle u, \varphi^\dagger(w) \rangle \langle u', \psi^\dagger(w') \rangle
    = \langle u \otimes u', \varphi^\dagger(w) \otimes \psi^\dagger(w') \rangle .
\end{align*}
A linear isomorphism that preserves the forms is unitary, and it suffices to check this on pure tensors.
The associator \((u \otimes u') \otimes u'' \mapsto u \otimes (u' \otimes u'')\) and the symmetry \(\sigma(u \otimes u') = u' \otimes u\) preserve the forms, since both sides are products of the same three, respectively two, factors.
The left unitor \(A \otimes E \to E\), \(f \otimes u \mapsto f u\), preserves the forms, since
\[
    \langle f u, g v \rangle = f \conj g \langle u, v \rangle = \langle f, g \rangle \langle u, v \rangle = \langle f \otimes u, g \otimes v \rangle,
\]
and the right unitor is treated in the same way.
Hence \(\Herm_A\) is dagger symmetric monoidal.

For the duals, recall from \cref{sec:dagger-compact} that \(\conj E\) is a Hermitian module.
The map \(\varepsilon\) is well defined and \(A\)-linear on \(E \otimes \conj E\), because \(\langle f u, v \rangle = f \langle u, v \rangle\) and \(\varepsilon(u \otimes f \cdot \conj v) = \langle u, \conj f v \rangle = f \langle u, v \rangle\).
Let \((e_i, \xi^i)\) be a dual basis of \(E\), with reciprocal basis \((e^i)\).
Suppressing associators and unitors, the first triangle identity is the composite
\[
    E \xrightarrow{\id \otimes \eta} E \otimes \conj E \otimes E \xrightarrow{\varepsilon \otimes \id} E,
    \qquad
    v \mapsto \sum_i v \otimes \conj{e^i} \otimes e_i \mapsto \sum_i \langle v, e^i \rangle\, e_i ,
\]
which is the identity by the first expansion of \cref{prop:reciprocal-expansion}.
The second triangle identity is the composite
\[
    \conj E \xrightarrow{\eta \otimes \id} \conj E \otimes E \otimes \conj E \xrightarrow{\id \otimes \varepsilon} \conj E,
    \qquad
    \conj v \mapsto \sum_i \conj{e^i} \otimes e_i \otimes \conj v \mapsto \sum_i \langle e_i, v \rangle \cdot \conj{e^i}
    = \conj{\sum_i \langle v, e_i \rangle\, e^i} ,
\]
using \(f \cdot \conj{e^i} = \conj{\conj f e^i}\) and \(\conj{\langle e_i, v \rangle} = \langle v, e_i \rangle\).
It is the identity by the second expansion of \cref{prop:reciprocal-expansion}.

It remains to show that \(\varepsilon = \eta^\dagger \circ \sigma\), where \(\sigma \colon E \otimes \conj E \to \conj E \otimes E\) is the symmetry.
By the computation after \cref{prop:herm-module-dagger}, the adjoint of \(\eta = \widehat{\eta(1)}\) is \(\eta(1)^\flat\).
Hence, for \(u \in E\) and \(\conj w \in \conj E\),
\begin{align*}
    \eta^\dagger(\sigma(u \otimes \conj w))
    &= \Bigl\langle \conj w \otimes u, \sum_i \conj{e^i} \otimes e_i \Bigr\rangle
    = \sum_i \langle e^i, w \rangle \langle u, e_i \rangle \\
    &= \Bigl\langle \sum_i \langle u, e_i \rangle\, e^i, w \Bigr\rangle
    = \langle u, w \rangle ,
\end{align*}
by the definition of the form on \(\conj E\) and the second expansion of \cref{prop:reciprocal-expansion}.
This is \(\varepsilon(u \otimes \conj w)\).
Since \(\varepsilon\) is defined by the form alone and \(\eta = (\varepsilon \circ \sigma^{-1})^\dagger\), the cup \(\eta\) is independent of the choice of dual basis, and \(\Herm_A\) is dagger compact.

\subsection{The jet comonad}
\label{sec:appendix-jets}

We complete the proof of \cref{prop:jet-reconstruction}, construct the maps of \cref{def:jet-maps}, prove \cref{prop:jet-comonad}, and show that \(\DiffOp\) is a preadditive locally graded category.
As in \cref{sec:jets}, a dot denotes the product in \(D\).
By \cref{prop:jet-reconstruction}, all jet modules, iterated or not, are finitely generated projective, and an element of \(J^k E\) is determined by its values \(j(P \otimes \alpha)\).
The dual basis of \(J^l E\) given there consists of elements \(\beta^a \otimes e_i\) and linear forms \(Q \otimes \alpha\) with \(Q \in D^l\), so an element of \(J^k J^l E\) is determined by its values on the elements \(P \otimes (Q \otimes \alpha)\), and similarly for jets of higher iterates.

We first complete the proof of \cref{prop:jet-reconstruction}.
The values of a jet define a map from \(J^k E\) to the dual of \(D^k \otimes_A E^\vee\), where the tensor product is taken over the right structure of \(D^k\), and the left structure of \(D^k\) makes it a left \(A\)-module.
The reconstruction formula shows that this map is injective.
It is surjective: for a linear form \(\lambda\) on \(D^k \otimes_A E^\vee\), the jet \(\sum_{a, i} \lambda(b_a \otimes \xi^i)\, \beta^a \otimes e_i\) takes the value
\[
    \sum_{a, i} \lambda(b_a \otimes \xi^i)\, \beta^a(P \cdot \alpha(e_i))
    = \sum_i \lambda(P \cdot \alpha(e_i) \otimes \xi^i)
    = \lambda(P \otimes \alpha)
\]
on \(P \otimes \alpha\), using \(\sum_a \beta^a(Q)\, b_a = Q\) for \(Q \in D^k\) and \(\sum_i \alpha(e_i)\, \xi^i = \alpha\).

Next we construct the maps of \cref{def:jet-maps}.
Let \((b_a, \beta^a)\) be a dual basis of \(D^l\), and define
\begin{gather*}
    \pi_{l,k}(\omega \otimes u) = \omega|_{D^k} \otimes u, \qquad
    \varepsilon(\omega \otimes u) = \omega(1)\, u, \\
    \delta_{k,l}(\omega \otimes u) = \sum_a \omega(- \cdot b_a) \otimes (\beta^a \otimes u),
\end{gather*}
where \(\omega(- \cdot b_a) \colon D^k \to A\) is the map \(P \mapsto \omega(P \cdot b_a)\), which is \(A\)-linear and defined because \(D^k D^l \subseteq D^{k+l}\).
The three formulas are biadditive and linear for the pointwise structure.
The first two are balanced because \((\omega f)|_{D^k} = \omega|_{D^k} f\) and \(\omega(f) = f \omega(1)\).
On \(P \otimes \alpha\), the first two formulas take the values \(\omega(P \cdot \alpha(u))\) and \(\omega(1)\, \alpha(u) = \omega(\alpha(u))\), as required.
On \(P \otimes (Q \otimes \alpha)\), the third takes the value
\[
    \sum_a \omega\bigl(P \cdot \beta^a(Q \cdot \alpha(u)) \cdot b_a\bigr) = \omega(P \cdot Q \cdot \alpha(u)),
\]
because \(\sum_a \beta^a(R)\, b_a = R\) for \(R \in D^l\).
The value \(\omega(P \cdot Q \cdot \alpha(u) f)\) is the same for \(\omega f \otimes u\) and for \(\omega \otimes f u\), so the third formula is balanced as well, and \(\delta_{k,l}\) is a well-defined map with the property required in \cref{def:jet-maps}.
Naturality in \(E\) follows by comparing values, using \eqref{eq:jet-functor-values}, and by uniqueness none of the maps depends on the dual basis.

We now verify the comonad laws.
The counit laws, coassociativity and the compatibility of \(\delta\) with the restrictions say that the following diagrams commute:
\[
    \begin{tikzcd}
        J^k \rar["\delta_{k,0}"] \drar["\id"'] & J^k J^0 \dar["J^k(\varepsilon)"] \\
        & J^k
    \end{tikzcd}
    \qquad
    \begin{tikzcd}
        J^k \rar["\delta_{0,k}"] \drar["\id"'] & J^0 J^k \dar["\varepsilon"] \\
        & J^k
    \end{tikzcd}
\]
\[
    \begin{tikzcd}[column sep = large]
        J^{k+l+m} \rar["\delta_{k,l+m}"] \dar["\delta_{k+l,m}"'] & J^k J^{l+m} \dar["J^k(\delta_{l,m})"] \\
        J^{k+l} J^m \rar["\delta_{k,l}"] & J^k J^l J^m
    \end{tikzcd}
    \qquad
    \begin{tikzcd}[column sep = large]
        J^{m+n} \rar["\delta_{m,n}"] \dar["\pi_{m+n,k+l}"'] & J^m J^n \dar["J^k(\pi_{n,l}) \circ \pi_{m,k}"] \\
        J^{k+l} \rar["\delta_{k,l}"] & J^k J^l
    \end{tikzcd}
\]

\begin{proof}[Proof of \cref{prop:jet-comonad}]
    The jet functors preserve finitely generated projective modules by \cref{prop:jet-reconstruction}, and we have seen that \(\pi\), \(\varepsilon\) and \(\delta\) are natural.
    The identities \(\pi_{k,k} = \id\) and \(\pi_{l,k} \circ \pi_{m,l} = \pi_{m,k}\) hold because the restriction of a map along \(D^k \subseteq D^k\) is the map itself, and restrictions compose.
    For the remaining identities we compare values, and the defining properties of \cref{def:jet-maps} say that, as linear forms,
    \[
        \alpha \circ \varepsilon = 1 \otimes \alpha, \qquad
        (Q \otimes \alpha) \circ \pi_{n,l} = Q \otimes \alpha, \qquad
        \bigl(Q \otimes (R \otimes \alpha)\bigr) \circ \delta_{l,m} = Q R \otimes \alpha,
    \]
    for \(Q \in D^l\), \(R \in D^m\) and \(l \le n\).
    The first counit law holds because
    \[
        J^k(\varepsilon)(\delta_{k,0}(j))(P \otimes \alpha) = \delta_{k,0}(j)\bigl(P \otimes (1 \otimes \alpha)\bigr) = j(P \otimes \alpha),
    \]
    and the second because \(\lambda(\varepsilon(j')) = j'(1 \otimes \lambda)\) for every linear form \(\lambda\), which for \(\lambda = Q \otimes \alpha\) gives
    \[
        (Q \otimes \alpha)\bigl(\varepsilon(\delta_{0,k}(j))\bigr) = \delta_{0,k}(j)\bigl(1 \otimes (Q \otimes \alpha)\bigr) = j(Q \otimes \alpha).
    \]
    Both sides of the coassociativity law take the value \(j(P Q R \otimes \alpha)\) on \(P \otimes (Q \otimes (R \otimes \alpha))\):
    \begin{align*}
        J^k(\delta_{l,m})(\delta_{k,l+m}(j))\bigl(P \otimes (Q \otimes (R \otimes \alpha))\bigr)
        &= \delta_{k,l+m}(j)\bigl(P \otimes (Q R \otimes \alpha)\bigr) \\
        &= j\bigl(P (Q R) \otimes \alpha\bigr), \\
        \delta_{k,l}(\delta_{k+l,m}(j))\bigl(P \otimes (Q \otimes (R \otimes \alpha))\bigr)
        &= \delta_{k+l,m}(j)\bigl(P Q \otimes (R \otimes \alpha)\bigr) \\
        &= j\bigl((P Q) R \otimes \alpha\bigr),
    \end{align*}
    and these agree because \(D\) is associative.
    Finally, for \(k \le m\), \(l \le n\), \(P \in D^k\) and \(Q \in D^l\), both sides of the last identity take the value \(j(P Q \otimes \alpha)\) on \(P \otimes (Q \otimes \alpha)\):
    \[
        J^k(\pi_{n,l})\bigl(\pi_{m,k}(\delta_{m,n}(j))\bigr)\bigl(P \otimes (Q \otimes \alpha)\bigr)
        = \delta_{m,n}(j)\bigl(P \otimes (Q \otimes \alpha)\bigr)
        = j(P Q \otimes \alpha),
    \]
    and \(\delta_{k,l}(\pi_{m+n,k+l}(j))(P \otimes (Q \otimes \alpha)) = \pi_{m+n,k+l}(j)(P Q \otimes \alpha) = j(P Q \otimes \alpha)\).
\end{proof}

Next we verify that \(\DiffOp\), as defined in \cref{def:diffop}, is a preadditive locally graded category.
Let \(P \in \Diff^k(E, F)\), \(Q \in \Diff^l(F, H)\) and \(R \in \Diff^m(H, E)\).
The change-of-degree maps are functorial because the restrictions are: \(\iota_{k,k}(P) = P \circ \pi_{k,k} = P\) and \(\iota_{l,m}(\iota_{k,l}(P)) = P \circ \pi_{l,k} \circ \pi_{m,l} = P \circ \pi_{m,k}\).
The unit laws follow from the counit laws and the naturality of \(\varepsilon\):
\[
    P \circ \id_E = P \circ J^k(\varepsilon) \circ \delta_{k,0} = P, \qquad
    \id_F \circ P = \varepsilon \circ J^0(P) \circ \delta_{0,k} = P \circ \varepsilon \circ \delta_{0,k} = P.
\]
Associativity follows from coassociativity and the naturality of \(\delta_{m,l}\):
\begin{align*}
    R \circ (Q \circ P)
    &= R \circ J^m(Q) \circ J^m J^l(P) \circ J^m(\delta_{l,k}) \circ \delta_{m,l+k} \\
    &= R \circ J^m(Q) \circ J^m J^l(P) \circ \delta_{m,l} \circ \delta_{m+l,k} \\
    &= R \circ J^m(Q) \circ \delta_{m,l} \circ J^{m+l}(P) \circ \delta_{m+l,k}
    = (R \circ Q) \circ P.
\end{align*}
Composition is compatible with the change of degree of each factor separately, and hence with \(\iota(Q) \circ \iota(P) = \iota(Q \circ P)\).
For \(k \le n\), the last identity of \cref{def:graded-comonad}, with \(\pi_{l,l} = \id\), gives
\begin{align*}
    Q \circ \iota_{k,n}(P)
    &= Q \circ J^l(P) \circ J^l(\pi_{n,k}) \circ \delta_{l,n} \\
    &= Q \circ J^l(P) \circ \delta_{l,k} \circ \pi_{l+n,l+k}
    = \iota_{l+k,l+n}(Q \circ P),
\end{align*}
and for \(l \le n\), the naturality of \(\pi_{n,l}\) and the same identity, with \(\pi_{k,k} = \id\), give
\begin{align*}
    \iota_{l,n}(Q) \circ P
    &= Q \circ \pi_{n,l} \circ J^n(P) \circ \delta_{n,k}
    = Q \circ J^l(P) \circ \pi_{n,l} \circ \delta_{n,k} \\
    &= Q \circ J^l(P) \circ \delta_{l,k} \circ \pi_{n+k,l+k}
    = \iota_{l+k,n+k}(Q \circ P).
\end{align*}
The hom-sets \(\Hom_A(J^k E, F)\) are abelian groups, the change-of-degree maps are additive, and composition is additive in each variable because the functors \(J^l\) are additive, so \(\DiffOp\) is preadditive.

\end{document}